\documentclass[10pt]{article}

\usepackage{amsmath,amssymb,amscd,amsthm,makeidx,graphicx,url,bm,bbm,mathrsfs,palatino}
\usepackage{a4wide,xy}
\usepackage{enumerate,comment}
\usepackage{titlesec}
\titleformat{\paragraph}[runin]
{\normalfont\normalsize\bfseries}{\theparagraph}{1em}{}
\titlespacing*{\paragraph}
{0pt}{3.25ex plus 1ex minus .2ex}{1.5ex plus .2ex}

\usepackage[%pdftex,
breaklinks=true,
pdftitle={Arithmetic and representation theory of wild character varieties},
pdfauthor={Tam\'as Hausel, Martin Mereb, Michael Lennox Wong}]{hyperref}
\usepackage{color}
\definecolor{tocolor}{rgb}{.1,.1,.5}
\definecolor{urlcolor}{rgb}{.2,.2,.6}
\definecolor{linkcolor}{rgb}{.1,.4,.6}
\definecolor{citecolor}{rgb}{.6,.3,.1}
\hypersetup{colorlinks=true, urlcolor=urlcolor, linkcolor=linkcolor, citecolor=citecolor}

\let\oldmarginpar\marginpar
\renewcommand\marginpar[1]{\-\oldmarginpar[\raggedleft\footnotesize #1]
	{\raggedright\footnotesize #1}}

\makeindex

\xyoption{all}
\input{xypic}

\numberwithin{equation}{subsection}

\newtheorem{theorem}[equation]{Theorem}
\newtheorem{proposition}[equation]{Proposition}
\newtheorem{corollary}[equation]{Corollary}
\newtheorem{conjecture}[equation]{Conjecture}
\newtheorem{lemma}[equation]{Lemma}

\newtheorem{problem}[equation]{Problem}
\theoremstyle{remark}
\newtheorem{remark}[equation]{Remark}

\theoremstyle{definition}
\newtheorem{definition}[equation]{Definition}

\newcounter{margin}
	{\end{itshape}  \bigskip}

\def\beq{\begin{eqnarray}}
\def\eeq{\end{eqnarray}}
\def\bes{\begin{eqnarray*}}
	\def\ees{\end{eqnarray*}}

\def\bG{\mathbf{G}}

\def\bU{{\bm U}}

\def\muhat{{\bm \mu}}
\def\r{{\bf r}}

\DeclareMathOperator{\Spec}{Spec} 
\DeclareMathOperator{\Hom}{Hom}
 
\DeclareMathOperator{\Ind}{Ind}
\DeclareMathOperator{\Res}{Res} 

\DeclareMathOperator{\Irr}{Irr}
\DeclareMathOperator{\Rep}{Rep}
\DeclareMathOperator{\sgn}{sgn}

\DeclareMathOperator{\Infl}{Infl}

\def\A{{\bf A}}

\def\C{\mathbb{C}}
\def\M{{\mathcal{M}}}
\def\pazM{{\mathscr{M}}}
\def\calQ{{\mathcal{Q}}}
\def\calA{{\mathcal{A}}}
\def\calI{{\mathcal{I}}}
\def\calX{{\mathcal{X}}}
\def\calU{{\mathcal{U}}}
\def\pazU{{\mathscr{U}}}

\def\calP{\mathcal{P}}
\def\calH{\mathcal{H}}

\def\v{\mathbf{v}}

\def\P{\mathcal{P}}
\def\H{\mathbb{H}}
\def\N{\mathbb{Z}_{\geq 0}}

\def\F{\mathbb{F}}
\def\Q{\mathbb{Q}}
\newcommand{\T}{\mathrm{T}}

\def\t{\mathfrak{t}}
\def\calC{{\mathcal C}}
\def\calG{{\mathcal G}}

\def\Z{\mathbb{Z}}

\def\K{\mathbb{K}}

\def\gl{{\mathfrak g\mathfrak l}}

\newcommand{\nc}{\newcommand}
\def\bU{\mathbf U}

\newcommand{\g}{\mathfrak{g}}
\usepackage{leftidx}

\newcommand{\calR}{\mathcal{R}}
\newcommand{\calS}{\mathcal{S}}

\newcommand{\Gr}{\textnormal{Gr}}

\newcommand{\xib}{{\bm \xi}}
\newcommand{\reg}{\textnormal{reg}}

\newcommand{\II}{\mathbb{I}}
\newcommand{\HHH}{\mathscr{H}}
\newcommand{\YY}{\mathsf{Y}}
\newcommand{\HH}{\mathsf{H}}
\newcommand{\NN}{\mathcal{N}}
\newcommand{\FF}{\mathcal{F}}

\newcommand{\End}{\textnormal{End}}
\newcommand{\tr}{\textnormal{tr}}
\newcommand{\opp}{\textnormal{op}}
\renewcommand{\1}{\mathbbm{1}}
\newcommand{\bT}{\mathbf{T}}
\newcommand{\bB}{\mathbf{B}}
\nc{\op}[1]{\mathop{\mathchoice{\mbox{\rm #1}}{\mbox{\rm #1}}
		{\mbox{\rm \scriptsize #1}}{\mbox{\rm \tiny #1}}}\nolimits}
\nc{\al}{\alpha}

\nc{\ep}{\varepsilon} 
\nc{\ga}{\gamma} 
\nc{\Ga}{\Gamma}
\nc{\la}{\lambda} 
\nc{\La}{\Lambda} 
\nc{\si}{\sigma}
\nc{\Sig}{{\Gamma}} 
\nc{\Om}{\Omega} 
\nc{\om}{\omega}

\nc{\SL}{\mathrm{SL}} 
\nc{\GL}{\mathrm{GL}} 
\nc{\PGL}{\mathrm{PGL}}
\nc{\G}{\mathrm{G}}
\nc{\W}{\mathrm{W}}
\nc{\Lg}{\mathrm{L}}
\nc{\Pg}{\mathrm{P}}

\nc{\Frob}{\mathrm{Frob}}
\def\PP {\mathcal{P}}

\def\U{{\mathrm{U}}}

\nc{\rN}{\mathrm{N}}
\def\wt{\widetilde}

\def\SS{\mathfrak{S}}

\def\leeq{\leq}
\def\leeqvs{\subseteq}
\def\leeqrad{\subseteq}

\newcommand{\pare}[1]{\left( #1 \right)} % \par a b le pone ( )
\newcommand{\size}[1]{\left| #1 \right|} % \par a b le pone | |
\newcommand{\set}[1]{\left\{ #1 \right\}} % \par a b le pone { }
\newcommand{\dual}[1]{\widehat{ #1 }} %se puede cambiar a ^{*}

\usepackage{longtable}

\DeclareMathOperator{\TD}{\bf {T} } 
\DeclareMathOperator{\DC}{\bf {D} }
\DeclareMathOperator{\id}{id }
\DeclareMathOperator{\stab}{{\bf Stab}}
\DeclareMathOperator{\Stab}{{\bf Stab}}

\nc{\cpt}{{\op{cpt}}} \nc{\Dol}{{\op{Dol}}} \nc{\DR}{{\op{DR}}}
\nc{\B}{{\op{B}}} \nc{\Triv}{\op{Triv}} \nc{\Hod}{{\op{Hod}}}
\nc{\Log}{{\op{Log}}} \nc{\Exp}{{\op{Exp}}} \nc{\Est}{E_{\op{st}}}
\nc{\Hst}{H_{\op{st}}} \nc{\Left}[1]{\hbox{$\left#1\vbox to
		10.5pt{}\right.\nulldelimiterspace=0pt \mathsurround=0pt$}}
\nc{\Right}[1]{\hbox{$\left.\vbox to
		10.5pt{}\right#1\nulldelimiterspace=0pt \mathsurround=0pt$}}
\nc{\LEFT}[1]{\hbox{$\left#1\vbox to
		15.5pt{}\right.\nulldelimiterspace=0pt \mathsurround=0pt$}}
\nc{\RIGHT}[1]{\hbox{$\left.\vbox to
		15.5pt{}\right#1\nulldelimiterspace=0pt \mathsurround=0pt$}}

\nc{\bee}{{\bf E}} \nc{\bphi}{{\bf \Phi}}

\begin{document}
	
	\title{Arithmetic and representation theory of wild character varieties}
	
	\author{ Tam\'as Hausel
		\\ {\it EPF Lausanne} 
		\\{\tt tamas.hausel@epfl.ch} \and Martin Mereb \\ {\it EPF Lausanne -- U de Buenos Aires}\\ {\tt mmereb@gmail.com } \and Michael Lennox Wong \\ {\it Universit\"at Duisburg-Essen}
		\\{\tt michael.wong@uni-due.de}  }
	\pagestyle{myheadings}
	
	\maketitle
	
	\begin{abstract} We count points over a finite field on wild character varieties of Riemann surfaces for singularities with  regular semisimple leading term. The new feature in our counting formulas is the appearance of characters of Yokonuma--Hecke algebras. Our result leads to the conjecture that the mixed Hodge polynomials of these character varieties agree with previously conjectured perverse Hodge polynomials of certain twisted parabolic Higgs moduli spaces, indicating the possibility of a $P=W$ conjecture for a suitable wild Hitchin system. \end{abstract}
	
	\section{Introduction}

\subsection{A conjecture}

Let $C$ be a complex smooth projective curve of genus $g\in \N$, with divisor $$D=p_1+\dots + p_k+rp$$ where $p_1,\dots,p_k, p\in C$ are distinct points,  $p$ having multiplicity $r\in \Z_{\geq 0}$ with $k\geq 0$ and  $k+r\geq 1$.  For $n \in \Z_{\geq 0}$, let $\P_n$ denote the set of partitions of $n$ and set $\P := \bigcup_n \P_n$.  Let $\muhat=(\mu^1,\dots,\mu^k)\in \P_n^k$ denote a $k$-tuple of partitions of $n$, and we write $|\muhat|:=n$.  We denote by $\M_\Dol^{\muhat,r}$  the moduli space of stable parabolic Higgs bundles $(E,\phi)$ with quasi-parabolic structure of type $\mu^i$ at the pole $p_i$, with generic parabolic weights and fixed parabolic degree, and a twisted (meromorphic) Higgs field $$\phi\in H^0\left(C;End(E)\otimes K_C(D)\right)$$ with nilpotent residues compatible with the quasi-parabolic structure at the poles $p_i$ (but no restriction on the residue at $p$).  Then $\M_\Dol^{\muhat,r}$ is a smooth quasi-projective variety of dimension $d_{\muhat,r}$ with a proper Hitchin map $$\chi^{\muhat,r}:\M_\Dol^{\muhat,r}\to \calA^{\muhat,r}$$ defined by taking the characteristic polynomial of the Higgs field $\phi$ and thus taking values in the Hitchin base $$ \calA^{\muhat,r}:=\oplus_{i=1}^n H^0(C;K_C(D)^i).$$ As $\chi^{\muhat,r}$ is proper it induces as in \cite{decataldo-migliorini} 
a perverse filtration $P$ on the rational cohomology $H^*(\M_\Dol^{\muhat,r})$ of the total space. We define the perverse Hodge polynomial as $$PH(\M_\Dol^{\muhat,r};q,t):=\sum \dim\left({\rm Gr}^P_i(H^k(\M_\Dol^{\muhat,r}))\right)q^it^k.$$ 

The recent paper \cite{cddg} by Chuang--Diaconescu--Donagi--Pantev gives a string theoretical derivation of the following mathematical conjecture.

\begin{conjecture} \label{twparconj} We expect
	$$PH(\M_\Dol^{\muhat,r};q,t)=(qt^2)^{d_{\muhat,\r}}\H_{\muhat,r}(q^{-{1}/{2}},-q^{{1}/{2}}t).$$
\end{conjecture}

Here, $\H_{\muhat,r}(z,w)\in \Q(z,w)$ is defined by the generating function formula \beq \label{hmur}\sum_{\muhat\in \calP^k} \frac{(-1)^{r|\muhat|}w^{-d_{\muhat,r}} \H_{\muhat,r}(z,w)}{(1-z^2)(1-w^2)} \prod_{i=1}^k m_{\mu_i}(x_i)=\Log{\left(\sum_{\lambda\in \calP}\calH^{g,r}_\lambda(z,w)\prod_{i=1}^k \tilde{H}_\lambda(z^2,w^2;x_i)\right) }. \eeq
The notation is explained as follows.  For a partition $\lambda\in \calP$ we denote \beq\label{calhgr} \calH_\lambda^{g,r}(z,w)=\prod \frac{(-z^{2a}w^{2l})^r(z^{2a+1}-w^{2l+1})^{2g}}
{(z^{2a+2}-w^{2l})(z^{2a}-w^{2l+2})},\eeq 
where the product is over the boxes in the Young diagram of $\lambda$ and $a$ and $l$ are the arm length and the leg length of the given box. We denote by $m_\lambda(x_i)$ the monomial symmetric functions in the infinitely many variables $x_i:=(x_{i_1},x_{i_2},\dots)$ attached to the puncture $p_i$.  $\tilde{H}_\lambda(q,t;x_i)$ denotes the twisted Macdonald polynomials of Garsia--Haiman \cite{garsia-haiman}, which is a symmetric function in the variables $x_i$ with coefficients from $\Q(q,t)$.  Finally, $\Log$ is the plethystic logarithm, see e.g. \cite[\S 2.3.3.]{hausel-aha1} for a definition.

The paper \cite{cddg} gives several pieces of evidence for Conjecture~\ref{twparconj}. On physical grounds it argues that the left hand side should be the generating function for certain refined BPS invariants of some associated Calabi--Yau 3-orbifold $Y$, which they then relate by a refined Gopakumar--Vafa conjecture to the generating function of the refined Pandharipan\-de--Thomas invariants of $Y$. In turn they  can compute the latter in some cases using the recent approach of Nekra\-sov--Okounkov \cite{nekrasov-okounkov}, finding agreement with Conjecture~\ref{twparconj}. Another approach is to use another duality conjecture---the so-called ``geometric engineering''---which conjecturally relates the left hand side of Conjecture~\ref{twparconj} to generating functions for equivariant indices of some bundles on certain nested Hilbert schemes of points on the affine plane $\C^2$. They compute this using work of Haiman \cite{haiman} and find agreement with the right hand side of Conjecture~\ref{twparconj}.

Purely mathematical evidence for Conjecture~\ref{twparconj} comes through a parabolic version of the $P=W$ conjecture of \cite{thmhcv}, in the case when $r=0$.  In this case, by non-abelian Hodge theory we expect the parabolic Higgs moduli space $\M_\Dol^{\muhat} := \M_\Dol^{\muhat,0}$ to be diffeomorphic with a certain character variety $\M_\B^\muhat$, which we will define more carefully below.  The cohomology of $\M_\B^\muhat$ carries a weight filtration, and we denote by $$WH(\M_\B^\muhat;q,t) := \sum_{i,k} \dim\left({\rm Gr}^W_{2i}(H^k(\M_\B^{\muhat}))\right)q^it^k,$$ the mixed Hodge polynomial of $\M_\B^\muhat$. The $P=W$ conjecture  predicts that the perverse filtration $P$ on $H^*(\M_\Dol^\muhat)$ is identified with the weight filtration $W$ on $H^*(\M_\B^\muhat)$ via non-abelian Hodge theory. In particular, $P=W$ would imply $PH(\M_\Dol^\muhat; q,t) = WH(\M_\B^\muhat; q,t)$, and Conjecture~\ref{twparconj} for $r=0$; $PH(\M_\Dol^\muhat;q,t)$ replaced with $WH(\M_\B^\muhat;q,t)$ was the main conjecture in \cite{hausel-aha1}. 

It is interesting to recall what inspired Conjecture~\ref{twparconj} for $r>0$.  Already in \cite[Section 5]{hausel-villegas}, detailed knowledge of the cohomology ring $H^*(\M_\Dol^{(2),r})$ from \cite{hausel-thaddeus} was needed for the computation of $WH(\M_\B^{(2)};q,t)$.  In fact, it was observed in \cite{thmhcv} that the computation in \cite[Remark 2.5.3]{hausel-villegas} amounted to a formula for $PH(\M_\Dol^\muhat;q,t)$, which is the first non-trivial instance of Conjecture~\ref{twparconj}. This twist by $r$ was first extended for the conjectured $PH(\M_\Dol^{(n),r})$ in \cite{mozgovoy} to match the recursion relation in \cite{diaconescu-etal}; it was then generalized in \cite{cddg} to Conjecture~\ref{twparconj}.  We notice that the twisting by $r$ only slightly changes the definition of $\calH^{g,r}(z,w)$ above and the rest of the right hand side of Conjecture~\ref{twparconj} does not depend on $r$. 

It was also speculated in \cite[Remark 2.5.3]{hausel-villegas} that there is a character variety whose mixed Hodge polynomial would agree with the one conjectured for $PH(\M^{\muhat,r}_\Dol;q,t)$ above. 

\begin{problem} \label{mainp} Is there a character variety whose mixed Hodge polynomial agrees with $PH(\M^{\muhat,r}_\Dol;q,t)$?
\end{problem}

A natural idea to answer this question is to look at the symplectic leaves of the natural Poisson structure on $\M^{\muhat,r}_\Dol$. The symplectic leaves  should correspond to moduli spaces of irregular or wild Higgs moduli spaces. By the wild non-abelian Hodge theorem \cite{biquard-boalch} those will be diffeomorphic with wild character varieties.

\subsection{Main result}
In this paper we will study a class of wild character varieties which will conjecturally provide a partial answer to the problem above. Namely, we will look at wild character varieties allowing irregular singularities with polar part having a diagonal regular leading term. Boalch in \cite{boalch} gives the following construction. 

Let $\G:=\GL_n(\C)$ and let $\T\leeq \G$ be the maximal torus of diagonal matrices.  Let $\B_+\leeq \G$ (resp.\ $\B_-\leeq \G$)  be the Borel subgroup of upper (resp.\ lower) triangular matrices. Let $\U = \U_+\leeq \B_+$ (resp.\ $\U_-\leeq \B_-$) be the respective unipotent radicals, i.e., the group of upper (resp.\ lower) triangular matrices with $1$'s on the main diagonal.  We fix $m\in\N$ and $$\r:=(r_1,\dots,r_m)\in \Z_{> 0}^m$$ an $m$-tuple of positive integers. For a $\muhat\in \calP_n^k$ we also fix a $k$-tuple $(\calC_1, \dots,\calC_k)$ of semisimple conjugacy classes, such that the semisimple conjugacy class $\calC_i\subset \G$ is of type $$\mu^i=(\mu^i_1,\mu^i_2,\dots)\in \calP_n;$$ in other words, $\calC_i$ has eigenvalues with multiplicities $\mu^i_j$. Finally we fix $$(\xi_1,\dots,\xi_m)\in (\T^\reg)^m$$ an $m$-tuple of regular diagonal matrices, such that the $k+m$ tuple $$(\calC_1,\dots,\calC_k,\G\xi_1,\dots,\G\xi_m)$$ of semisimple conjugacy classes is generic in the sense of Definition~\ref{genericconjugacyA}. Then define
\bes \M_\B^{\muhat,\r}:= \{ (A_i,B_i)_{i=1..n} \in (\G\times \G)^{g}, X_j\in \calC_j, C_j\in \G, (S^j_{2i-1},S^j_{2i})_{i=1, \ldots, r_j }\in (\U_- \times \U_+)^{r_j} |\\  (A_1,B_1) \cdots (A_g,B_g) X_1\cdots X_k C_1^{-1}\xi_1S^1_{2r_1}\cdots S^1_1C_1\cdots C_m^{-1}\xi_mS^m_{2r_m}\cdots S^m_1C_m=I_n\}/\!/ \G, \ees
where the affine quotient is by the conjugation action of $\G$ on the matrices $A_i,B_i,X_i,C_i$ and the trivial action on $S^j_i$. Under the genericity condition as above, $\M_\B^{\muhat,\r}$ is a smooth affine variety of dimension $d_{\muhat,\r}$ of \eqref{dimensionformula}. In particular, when $m=0$, we have the character varieties $\M_\B^\muhat=\M_\B^{\muhat,\emptyset}$ 
of \cite{hausel-aha1}.

The main result of this paper is the following:
\begin{theorem} \label{maint}  Let $\muhat\in \calP^k_n$ be a $k$-tuple of partitions of $n$ and $\r$ be an $m$-tuple of positive integers and $\M_\B^{\muhat,\r}$ be the generic wild character variety as defined above. Then we have $$WH(\M_\B^{\muhat,\r};q,-1)=q^{d_{\muhat},\r}\H_{\tilde{\muhat},r}(q^{-1/2},q^{1/2}),$$ where $$\tilde{\muhat}:=(\mu^1,\dots,\mu^k,(1^n),\dots,(1^n))\in \P^{k+m}_n$$ is the type of $$(\calC_1,\dots,\calC_k,\G\xi_1,\dots,\G\xi_m)$$ and $$r:=r_1+ \cdots + r_m.$$
\end{theorem}

The proof of this result follows the route introduced in \cite{hausel-villegas,hausel-aha1,hausel-aha2}.  Using a theorem of Katz \cite[Appendix]{hausel-villegas}, it reduces the problem of the computation of $WH(\M_\B^{\muhat,\r};q,-1)$ to counting $\M_\B^{\muhat,\r}(\F_q)$, i.e., the $\F_q$ points of $\M_\B^{\muhat,\r}$. We count it by a non-abelian Fourier transform.  The novelty here is the determination of the contribution of the wild singularities to the character sum.

The latter problem is solved via the character theory of the Yokonuma--Hecke algebra, which is the convolution algebra on
$$\C[\U(\F_q)\backslash \GL_n(\F_q)/ \U(\F_q)],$$
where $\U$ is as above.  The main computational result, Theorem~\ref{t:scalarvaluecentral}, is an analogue of a theorem of Springer (cf.\ \cite[Theorem 9.2.2]{GeckPfeiffer}) which finds an explicit value for the trace of a certain central element of the Hecke algebra in a given representation.

This theorem, in turn, rests on a somewhat technical result relating the classification of the irreducible characters of the group $\rN = (\F_q^\times)^n \rtimes \SS_n$ to that of certain irreducible characters of $\GL_n(\F_q)$.  To explain briefly, if $\calQ_n$ denotes the set of maps from $\Gamma_1 = \widehat{\F}_q^\times$ (the character group of $\F_q^\times$) to the set of partitions of total size $n$ (see Section~\ref{s:chartables} for definitions and details), then $\calQ_n$ parametrizes both $\Irr \rN$ and a certain subset of $\Irr \GL_n(\F_q)$.  Furthermore, both of these sets are in bijection with the irreducible characters of the Yokonuma--Hecke algebra.  Theorem~\ref{p:expbijection} clarifies this relationship, establishing an analogue of a result proved by Halverson and Ram \cite[Theorem 4.9(b)]{bitraces}, though by different techniques.

%The main contribution  here is the determination of the contribution of the wild singularities to the character sum. This latter is computed via the character theory of the Yokonuma--Hecke algebra, which is the convolution algebra on 

Our main result Theorem~\ref{maint} then leads to the following conjecture.

\begin{conjecture} \label{mainc}   We have $$WH(\M_\B^{\muhat,\r};q,t)=(qt^2)^{d_{\muhat,\r}}\H_{\tilde{\muhat},r}(q^{-1/2},-tq^{-1/2}).$$ \end{conjecture}
This gives a conjectural partial answer to our Problem~\ref{mainp} originally raised in \cite[Remark 2.5.3]{hausel-villegas}. Namely, in the cases when at least one of the partitions $\mu^i=(1^n)$, we can conjecturally find a character variety whose mixed Hodge polynomial agrees with the mixed Hodge polynomial of a twisted parabolic Higgs moduli space. This class does not yet include the example studied in  \cite[Remark 2.5.3]{hausel-villegas}, where there is a single trivial partition $\muhat=((n))$. We expect that those cases could be covered with more complicated, possibly twisted, wild character varieties. 

Finally, we note that a recent conjecture \cite[Conjecture 1.12]{shende-etal} predicts that in the case when $g=0, k=0, m=1$ and $r=r_1\in \Z_{>0}$,  the mixed Hodge polynomial of our (and more general) wild character varieties, are intimately related to refined invariants of links arising from Stokes data. Our formulas in this case should be related to refined invariants of the $(n,rn)$ torus links. We hope that the natural emergence of Hecke algebras in the arithmetic of wild character varieties will shed new light on Jones's approach \cite{jones} to the HOMFLY polynomials via Markov traces on the usual Iwahori--Hecke algebra and the analogous Markov traces on the Yokonuma--Hecke algebra, c.f.\ \cite{juyumaya,cl,jacon-poulain}. 

The structure of the paper is as follows. Section~\ref{general} reviews mixed Hodge structures on the cohomology of algebraic varieties, the theorem of Katz mentioned above, and gives the precise definition of a wild character variety from \cite{boalch}.  In Section~\ref{Halg} we recall the abstract approach to Hecke algebras; the explicit character theory of the Iwahori--Hecke and Yokonuma--Hecke algebras is also reviewed and clarified.  In Section~\ref{countwcv} we recall the arithmetic Fourier transform approach of \cite{hausel-aha1} and  perform the count on the wild character varieties. In Section~\ref{mains} we prove our main Theorem~\ref{maint} and discuss our main Conjecture \ref{mainc}. In Section~\ref{exspec} we compute some specific examples of Theorem~\ref{maint} and Conjecture~\ref{mainc}, when $n=2$, with particular attention paid to the cases when $\M_\B^{\muhat,\r}$ is a surface. 
%\todo{Go over structure of paper at the end}

\noindent {\bf Acknowledgements.}  We thank Philip Boalch, Daniel Bump, Maria Chlouveraki,  Alexander Dimca, Mario Garc\'ia-Fern\'andez, Eugene Gorsky, Emmanuel Letellier, Andr\'as N\'emethi, Lo{\" i}c Poulain d'Andecy,  Vivek Shende, Szil\'ard Szab\'o and Fernando R.\ Villegas for discussions and/or correspondence. We are also indebted to Lusztig's observation \cite[1.3.(a)]{Lusztig} which led us to study representations of Yokonuma--Hecke algebras.  This research was supported by \'Ecole Polytechnique F\'ed\'erale de Lausanne, an Advanced Grant ``Arithmetic and physics of Higgs moduli spaces'' no.\ 320593 of the European Research Council and the NCCR SwissMAP of the Swiss National Foundation.  Additionally, in the final stages of this project, MLW was supported by SFB/TR 45 ``Periods, moduli and arithmetic of algebraic varieties'', subproject M08-10 ``Moduli of vector bundles on higher-dimensional varieties''.

	\section{Generalities}
	\label{general}

\subsection{Mixed Hodge polynomials and counting points}

To motivate the problem of counting points on an algebraic variety, we remind the reader of some facts concerning mixed Hodge polynomials and varieties with polynomial count, more details of which can be found in \cite[\S2.1]{hausel-villegas}.  Let $X$ be a complex algebraic variety.  The general theory of \cite{Deligne-HodgeII, Deligne-HodgeIII} provides for a mixed Hodge structure on the compactly supported cohomology of $X$:  that is, there is an increasing weight filtration $W_\bullet$ on $H_c^j(X, \Q)$ and a decreasing Hodge filtration $F^\bullet$ on $H_c^j(X, \C)$.  The \emph{compactly supported mixed Hodge numbers of $X$} are defined as
\begin{align*}
h_c^{p,q;j}(X) := \dim_\C \Gr_F^p \Gr_{p+q}^W H_c^j(X, \C),
\end{align*}
the \emph{compactly supported mixed Hodge polynomial of $X$} by
\begin{align*}
H_c(X; x, y, t) := \sum h_c^{p,q;j}(X) x^p y^q t^j,
\end{align*}
and the \emph{$E$-polynomial of $X$} by
\begin{align*}
E(X; x,y) := H_c(X; x, y, -1).
\end{align*}
We could also define the mixed weight polynomial 
$$WH(X;q,t)=\sum  \dim_\C  \Gr_{k}^W H_c^j(X, \C) q^{k/2} t^j = H_c(X; q^{1/2}, q^{1/2}, t)$$
which specializes to the weight polynomial 
$$E(q):=WH(X;q,-1)=E(X;q^{1/2}, q^{1/2}).$$
One observes that the compactly supported Poincar\'e polynomial $P_c(X; t)$ is given by
\begin{align*}
P_c(X;t) = H_c(X;1,1,t).
\end{align*}

Suppose that there exists a separated scheme $\mathcal{X}$ over a finitely generated $\Z$-algebra $R$, such that for some embedding $R \hookrightarrow \C$ we have
\begin{align*}
\mathcal{X} \times_R \C \cong X;
\end{align*}
in such a case we say that $\mathcal{X}$ is a \emph{spreading out of} $X$.  If, further, there exists a polynomial $P_X(w) \in \Z[w]$ such that for any homomorphism $R \to \F_q$ (where $\F_q$ is the finite field of $q$ elements), one has
\begin{align*}
| \mathcal{X}(\F_q) | = P_X(q),
\end{align*}
then we say that \emph{$X$ has polynomial count} and $P_X$ is the \emph{counting polynomial of $X$}.  The motivating result is then the following.

\begin{theorem}(N.\ Katz, \cite[Theorem 6.1.2]{hausel-villegas}) \label{katz}
	Suppose that the complex algebraic variety $X$ is of polynomial count with counting polynomial $P_X$.  Then
	\begin{align*}
	E(X; x,y) = P_X(xy).
	\end{align*}
\end{theorem}
\begin{remark} Thus, in the polynomial count case we find that the count polynomial $P_X(q)=E(X;q)$ agrees with the weight polynomial. We also expect our varieties to be Hodge--Tate, i.e., $h_c^{p,q;j}(X)=0$ unless $p=q$, in which case $H_c(X;x,y,t) = WH(X;xy,t)$.  Thus, in these cases we are not losing information by considering $WH(X;xy,t)$ (resp.\ $E(X;q)$) instead of the usual $H_c(X;x,y,t)$ (resp.\ $E(X;x,y)$).
\end{remark}

\subsection{Wild character varieties}\label{ss:WChVar}

The wild character varieties we study in this paper were first mentioned in \cite[\S3 Remark 5]{BoalchQHam}, as a then new example in quasi-Hamiltonian geometry---a ``multiplicative'' variant of the theory of Hamiltonian group actions on symplectic manifolds---with a more thorough (and more general) construction given in \cite[\S8]{boalch}.  We give a direct definition here for which knowledge of quasi-Hamiltonian geometry is not required; however, as we appeal to results of \cite[\S9]{boalch} on smoothness and the dimension of the varieties in question, we use some of the notation of \cite[\S9]{boalch} to justify the applicability of those results.

\subsubsection{Definition}

We now set some notation which will be used throughout the rest of the paper.  %We fix a field $\mathbb{K}$ over which we work, $n \in \Z_{>0}$ and assume $\textnormal{char} \, \mathbb{K} \nmid n$.  
Let $\G := \rm{GL}_n(\mathbb{C})$ and fix the maximal torus $\T \leeq \G$ consisting of diagonal matrices; let $\g := \gl_n(\mathbb{C}), \t := \textnormal{Lie}(\T)$ be the corresponding Lie algebras.  Let $\B_+\leeq \G$ (resp., $\B_-\leeq \G$)  be the Borel subgroup of upper (resp., lower) triangular matrices.  Let $\U = \U_+ \leeq \B_+$ (resp., $\U_- \leeq \B_-$) be the unipotent radical, i.e., the group of upper (resp., lower) triangular matrices with $1$s on the main diagonal; one will note that each of these subgroups is normalized by $\T$.

\begin{definition} \label{wcvdefn}
	We will use the following notation.  For $r \in \Z_{> 0}$, we set
	\begin{align*}
	\calA^r := \G \times (\U_+ \times \U_-)^r \times \T.
	\end{align*}
	An element of $\calA^r$ will typically be written $(C, S, t)$ with $$S = (S_1, \ldots, S_{2r}) \in (\U_+ \times \U_-)^r,$$ where $S_i \in \U_+$ if $i$ is odd and $S_i \in U_-$ if $i$ is even.  The group $\T$ acts on $(\U_+ \times \U_-)^r$ by $$x \cdot S = (x S_1 x^{-1}, \ldots, x S_{2r} x^{-1});$$ the latter tuple will often be written simply as $x S x^{-1}$.
	
	We fix $g, k, m \in \Z_{\geq 0}$ with $k+m \geq 1$.  Fix also a $k$-tuple $${\bm \calC} := (\calC_1, \ldots, \calC_k)$$ of semisimple conjugacy classes $\calC_j \subseteq G$; the multiset of multiplicities of the eigenvalues of each $\calC_j$ determines a partition $\mu^j \in \P_n$. Hence we obtain a $k$-tuple $$\muhat := (\mu^1, \ldots, \mu^k) \in \P_n^k$$ which we call the {\em type} of ${\bm \calC} $.  Fix also $$\r := (r_1, \ldots, r_m) \in \Z_{>0}^m.$$ We will write $r := \sum_{\alpha=1}^m r_\alpha$.  Now consider the product
	\begin{align*}
	\calR^{g, {\bm \calC}, \r} := (\G \times \G)^g \times \calC_1 \times \cdots \times\calC_k \times\calA^{r_1} \times\cdots \times\calA^{r_m}.
	\end{align*}
	The affine variety $\calR^{g, {\bm \calC}, \r}$ admits an action of $\G \times \T^m$ given by
	\begin{align} \label{actionR}
	(y, x_1, \ldots, x_m) \cdot \left( A_i, B_i, X_j, C_\alpha, S^\alpha, t_\alpha \right)= \left( yA_iy^{-1}, yB_i y^{-1}, yX_j y^{-1}, x_\alpha C_\alpha y^{-1}, x_\alpha S_\alpha x_\alpha^{-1}, t_\alpha \right),
	\end{align}
	where the indices run $1 \leq i \leq g, 1 \leq j \leq k, 1 \leq \alpha \leq m$.
	
	Now, fixing an element $\xib = (\xi_1, \ldots, \xi_m) \in \T^m$, we define a closed subvariety of $\calR^{g, {\bm \calC}, \r}$ by 
	\begin{align} \label{explicitdescription}
	\calU^{g, {\bm \calC}, \r, \xib} := \Bigg\{ \left( A_i, B_i, X_j, C_\alpha, S^\alpha, t_\alpha \right) \in \calR^{g, {\bm \calC}, \r} \, : \, \prod_{i=1}^g (A_i, B_i) \prod_{j=1}^k X_j & \prod_{\alpha=1}^m C_\alpha^{-1} \xi_\alpha S_{2r_\alpha}^\alpha \cdots S_1^\alpha C_\alpha = I_n, \nonumber \\
	& t_\alpha = \xi_\alpha, 1 \leq \alpha \leq m \Bigg\},
	\end{align} \normalsize
	where the product means we write the elements in the order of their indices:
	\begin{align*}
	\prod_{i=1}^d y_i = y_1 \cdots y_d.
	\end{align*}
	It is easy to see that $\calU^{g, {\bm \calC}, \r, \xib} $ is invariant under the action of $\G \times \T^m$.  Finally, we define the \emph{(generic) genus $g$ wild character variety with parameters ${\bm \calC}, \r, \xib$} as the affine geometric invariant theory quotient
	\begin{align}\label{wilddef}
	\M_\B^{g, {\bm \calC}, \r, \xib} := \calU^{g, {\bm \calC}, \r, \xib} / (\G \times \T^m) = \Spec \mathbb{C}[ \calU^{g, {\bm \calC}, \r, \xib}]^{\G \times \T^m}.
	\end{align}
	Since $g$ will generally be fixed and understood, we will typically omit it from the notation.  Furthermore, since the invariants we compute depend only on the tuples $\muhat$ and $\r$, rather than the actual conjugacy classes ${\bm \calC}$ and $\xib$, we will usually abbreviate our notation to $\M_\B^{\muhat,  \r}$ and  $\calU^{\muhat, \r}$.
\end{definition}

\begin{remark}
	The space $\calA^r$ defined at the beginning of Definition \ref{wcvdefn} is a ``higher fission space'' in the terminology of \cite[\S3]{boalch}.  These are spaces of local monodromy data for a connection with a higher order pole.  To specify a de Rham space---which are constructed in \cite{biquard-boalch}, along with their Dolbeault count\-er\-parts---at each higher order pole, one specifies a ``formal type'' which is the polar part of an irregular connection which will have diagonal entries under some trivialization; this serves as a ``model'' connection.  The de Rham moduli space then parametrizes holomorphic isomorphism classes of connections which are all formally isomorphic to the specified formal type. Locally these holomorphic isomorphism classes are distinguished by their Stokes data, which live in the factor $(\U_+ \times \U_-)^r$ appearing in $\calA^r$.  The factor of $\T$ appearing is the ``formal monodromy'' which differs from the actual monodromy by the product of the Stokes matrices, as appearing in the last set of factors in the expression (\ref{explicitdescription}).  The interested reader is referred to \cite[\S2]{BoalchGbundles} for details about Stokes data.
\end{remark}

\begin{remark}
	%Michael: do we really need this remark; if not just omit it
	As mentioned above, these wild character varieties were constructed in \cite[\S8]{boalch} as quasi-Hamiltonian quotients.  In quasi-Hamiltonian geometry, one speaks of a space with a group action and a moment map into the group.  In this case, we had an action on $\calR^{\muhat, \r}$ given in (\ref{actionR}) and the corresponding moment map $\Phi : \calR^{\muhat,\r} \to \G \times \T^m$ would be
	\begin{align*}
	\left( A_i, B_i, X_j, C_\alpha, S_\alpha, t_\alpha \right) \mapsto \left( \prod_{i=1}^g [A_i, B_i]  \prod_{j=1}^k X_j \prod_{\alpha=1}^m C_\alpha^{-1} t_\alpha S_{2r_\alpha}^\alpha \cdots S_1^\alpha C_\alpha, t_1^{-1}, \ldots, t_m^{-1} \right).
	\end{align*}
	Then one sees that $\calU^{\muhat, \r} = \Phi^{-1}( (I_n, \xib^{-1}))$ and so $\M_\B^{\muhat, \r} = \Phi^{-1}( (I_n, \xib^{-1}))/ (G \times T^m)$ is a quasi-Hamiltonian quotient.
\end{remark}

\begin{remark}
	By taking determinants in (\ref{explicitdescription}) we observe that a necessary condition for $\calU^{\muhat,  \r} $, and hence $\M_\B^{\muhat,  \r}$, to be non-empty is that
	\begin{align} \label{detcondn}
	\prod_{j=1}^k \det \calC_j \cdot \prod_{\alpha=1}^m \det \xi_\alpha = 1,
	\end{align}
	noting that $\det S_p^\alpha = 1$ for $1 \leq \alpha \leq m, 1 \leq p \leq 2r_\alpha$.
\end{remark}

\subsubsection{Smoothness and dimension computation}

We recall  \cite[Definition 2.1.1]{hausel-aha1}.
\begin{definition}
	\label{genericconjugacyA}
	The $k$-tuple ${\bm \calC}=(\calC_1,\dots,\calC_k)$ is {\em generic} if the following holds.  If $V \leeqvs \C^n$ is a subspace stable by some $X_i\in \calC_i$ for each $i$
	such that
	\beq
	\label{prod-conditionA}
	\prod_{i=1}^k\det \left(X_i|_V\right)= 1
	\eeq
	then either $V=0$ or $V=\C^n$.  When, additionally, $$\xib=(\xi_1, \ldots, \xi_m )\in \T^m$$ we say that $${\bm \calC}\times \xib=(\calC_1,\dots,\calC_k,\xi_1,\dots,\xi_m)$$ is generic if $$(\calC_1,\dots,\calC_k,\G \xi_1,\dots,\G \xi_m)$$ is, where $\G \xi_i$ is the conjugacy class of $\xi_i$ in $\G$.
\end{definition}

\begin{remark} It is straightforward to see that the genericity of $(\calC_1,\dots,\calC_k)$ for a $k$-tuple of semi\-simple conjugacy classes can be formulated in terms of the spectra of the matrices in $\calC_i$ as follows. Let $$\A_i:=\{\alpha^i_1,\dots,\alpha^i_n\}$$ be the multiset of eigenvalues of a matrix in $\calC_i$ for $i=1\dots k$. Then $(\calC_1,\dots,\calC_k)$ is generic if and only if the following non-equalities \eqref{noneq} hold. Write $$[\A]:=\prod_{\alpha\in \A}a$$ for any multiset $\A\subseteq \A_i$. The non-equalities are \beq\label{noneq} [\A'_1]\cdots[\A'_k]\neq 1\eeq for $\A_i'\subseteq \A_i$ of the same cardinality $n'$ with $0<n'<n$.
\end{remark}

\begin{theorem}
	For a generic choice of ${\bm \calC}\times \xib$ (in the sense of Definition \ref{genericconjugacyA}), the wild character variety $\M_B^{\muhat,  \r} $ is smooth.  Furthermore, the $\G \times \T^m$ action on $\calU^{ \muhat,\r}$ is scheme-theoretically free.  Finally, one has
	\begin{align} \label{dimensionformula}
	\dim \M_B^{\muhat,  \r} =  (2g+k-2)n^2 - \|\muhat\|^2 + n(n-1) \left( m + r\right) + 2 =: d_{\muhat, \r},
	\end{align}
	where  $r :=\sum_{\alpha=1}^m r_\alpha$ and $$\|\muhat\|^2 = \sum_{j=1}^k \sum_{p=1}^{\ell_j} (\mu_p^j)^2$$ for $$\mu^j = (\mu_1^j, \ldots, \mu_{\ell_j}^j).$$ 
\end{theorem}

The first statement is a special case of \cite[Corollary 9.9]{boalch}, the second statement follows from the observations following \cite[Lemma 9.10]{boalch}, and the dimension formula comes from \cite[\S9, Equation (41)]{boalch}.  To see that our wild character varieties are indeed special cases of those constructed there, one needs to see that the ``double'' $\mathbf{D} = \G \times \G$ (see \cite[Example 2.3]{boalch}) is a special case of a higher fission variety, as noted at \cite[\S3, Example (1)]{boalch}, and that $\mathbf{D} /\!\!/_{\calC^{-1}} \G \cong \calC$ for a conjugacy class $\calC \subset \G$.  Then one may form the space
\begin{align*}
\calS^{g,k, \r} := \mathbb{D}^{ \circledast_\G g} \circledast_\G \mathbf{D}^{\circledast_\G k} \circledast_\G \calA^{r_1} \circledast_\G \cdots \circledast_\G \calA^{r_m} /\!\!/ \G,
\end{align*}
in the notation of \cite[\S\S2,3]{boalch} and see that $\M_\B^{\muhat, \r}$ is a quasi-Hamiltonian quotient of the above spƒace by the group $\G^k \times \T^m$ at the conjugacy class $({\bm \calC} \times \xib)$, and is hence a wild character variety as defined at \cite[p.342]{boalch}.

To see that the genericity condition given at \cite[\S9, Equations (38), (39)]{boalch} specializes to ours (Definition \ref{genericconjugacyA}), we observe that for $\G=\GL_n (\mathbb{C})$ the Levi subgroup $L$ of a maximal standard parabolic subgroup $P$ corresponds to a subgroup of matrices consisting of two diagonal blocks, and as indicated earlier in the proof of \cite[Corollary 9.7]{boalch}, the map denoted $\textnormal{pr}_L$ takes the determinant of each factor.  In particular, it takes the determinant of the relevant matrices restricted to the subspace preserved by $P$.  But this is the condition in Definition \ref{genericconjugacyA}.

Finally, using \cite[\S9, Equation (41)]{boalch} and the fact that $$\dim \U_+ = \dim \U_- = \binom{n}{2},$$ it is straightforward to compute the dimension as (\ref{dimensionformula}).

	\section{Hecke Algebras} 
	\label{Halg}
In the following, we describe the theory of Hecke algebras that we will need for our main results.  Let us first explain some notation that will be used.  Typically, the object under discussion will be a $\C$-algebra $\mathbf{A}$ which is finite-dimensional over $\C$.  We will denote its set of (isomorphism classes of) representations by $\Rep \mathbf{A}$ and the subset of irreducible representations by $\Irr \mathbf{A}$; since it will often be inconsequential, we will often also freely confuse an irreducible representation with its character.  Of course, if $\mathbf{A} = \C[G]$ is the group algebra of a group $G$, then we often shorten $\Irr \C[G]$ to $\Irr G$.  We will also sometimes need to consider ``deformations'' or ``generalizations'' of these algebras.  If $\mathbf{H}$ is an algebra free of finite rank over $\C[u^{\pm 1}]$  the extension $\C(u) \otimes_{\C[u^{\pm 1}]} \mathbf{H}$ of such an algebra to the quotient field $\C(u)$ of $\C[u^{\pm 1}]$ will be denoted by $\mathbf{H}(u)$. Note that this abbreviates the notation $\C(u) \mathbf{H}$ in \cite{cp} and \cite[\S~7.3]{GeckPfeiffer}.  Now, if $z \in \C^\times$ and $\theta_z : \C[u^{\pm 1}] \to \C$ is the $\C$-algebra homomorphism which takes $u \mapsto z$, then we may consider the ``specialization'' $\C \otimes_{\theta_z} \mathbf{H}$ of $\mathbf{H}$ to $u=z$ which we will denote by $\mathbf{H}(z)$.  

\subsection{Definitions and Conventions} \label{Halgnotn}

Let $G$ be a finite group and $H \leeq G$ a subgroup.  Given a subset $S \subseteq G$, we will denote its indicator function by $\II_S : G \to \N$, that is to say,
\begin{align*}
\II_S(x) = \begin{cases} 1 & x \in S \\ 0 & \textnormal{otherwise.} \end{cases}
\end{align*} 

Let $M$ be the vector space of functions $f : G \to \C$ such that
\begin{align*}
f(hg) = f(g)
\end{align*}
for all $h \in H, g \in G$. Clearly, $M$ can be identified with the space of complex-valued functions on $H \backslash G$ and so has dimension $[ G : H]$.  We may choose a set $V$ of right $H$-coset representatives, so that
\begin{align} \label{GU}
G = \coprod_{v \in V} Hv.
\end{align}
Such a choice gives a basis 
\beq\label{basisnot}\{ f_v := \II_{Hv} \}_{v \in V}\eeq 
of $M$.  Furthermore, we have a $G$-action on $M$ via
\begin{align} \label{GactM}
(g \cdot f)(x) := f(xg).
\end{align}
With this action, $M$ is identified with the induced representation $\Ind_H^G \1_H$ of the trivial representation $\1_H$ on $H$.

The \emph{Hecke algebra associated to $G$ and $H$}, which we denote by $\HHH(G, H)$,  is the vector space of functions $\varphi : G \to \C$ such that 
\begin{align*}
\varphi( h_1 g h_2) = \varphi(g),
\end{align*} 
for $h_1, h_2 \in H, g \in G$.  It has the following convolution product
\begin{align} \label{convprod}
(\varphi_1 \ast \varphi_2)(g) := \frac{1}{|H|} \sum_{a \in G} \varphi_1(ga^{-1}) \varphi_2(a) = \frac{1}{|H|} \sum_{b \in G} \varphi_1(b) \varphi_2(b^{-1} g). 
\end{align}
Furthermore, there is an action of $\HHH(G, H)$ on $M$, where, for $\varphi \in \HHH(G, H), f \in M, g \in G$, one has
\begin{align} \label{Halgaction}
(\varphi . f)(g) & := \frac{1}{|H|} \sum_{x \in G} \varphi(x) f(x^{-1} g). 
\end{align}
One easily checks that this is well-defined (by which we mean that $\varphi . f \in M$).

It is clear that $\HHH(G, H)$ may be identified with $\C$-valued functions on $H \backslash G / H$, and hence it has a basis indexed by the double $H$-cosets in $G$.  Let $\W \subseteq G$ be a set of double coset representatives which contains $e$ (the identity element of $G$), so that
\begin{align} \label{Bruhatdecomp}
G = \coprod_{w \in \W} H w H.
\end{align}
Then for $w \in \W$, we will set
\begin{align*}
T_w := \II_{HwH};
\end{align*}
these form a basis of $\HHH(G, H)$.

\begin{proposition} \cite[Proposition 1.4]{Iwahori1964} \label{Halgcomm}
	Under the convolution product (\ref{convprod}), $\HHH(G,H)$ is an associative algebra with identity $T_e = \II_H$.  The action (\ref{Halgaction}) yields a unital embedding of algebras $\HHH(G, H) \to \End_{\C} M$  whose image is $\End_G M$. Thus we may identify 
	\begin{align*}
	\HHH(G,H) = \End_G \Ind_H^G \1_H.
	\end{align*}
\end{proposition}

%\todo{this needs a proof or reference}

\begin{remark} \label{twostar} (Relation with the group algebra) The group algebra $\C [G]$ may be realized as the space of functions $\sigma : G \to \C$ with the multiplication $$\ast_G : \C[G] \otimes_\C \C[G] \to \C[G]$$ given by
	\begin{align*}
	(\sigma_1 \ast_G \sigma_2)(x) = \sum_{a \in G} \sigma_1(a) \sigma_2(a^{-1} x).
	\end{align*}
	It is clear that we have an embedding of vector spaces $$\iota : \HHH(G, H) \hookrightarrow \C[G]$$ and it is easy to see that if $\varphi_1, \varphi_2 \in \HHH(G, H)$, we have
	\begin{align} \label{gpalgHalg}
	\iota(\varphi_1) \ast_G \iota(\varphi_2) = |H| (\varphi_1 \ast \varphi_2),
	\end{align}
	where the right hand side is the convolution product in $\HHH(G, H)$.  Furthermore, the inclusion takes the identity element $T_e \in \HHH(G, H)$ to $\II_H$, which is not the identity element in $\C[G]$.  Thus, while the relationship between the multiplication in $\HHH(G, H)$ and that in $\C[G]$ will be important for us, we should be careful to note that $\iota$ is not an algebra homomorphism.  When we are dealing with indicator functions for double $H$-cosets, we will write $T_v, v \in \W$ when we consider it as an element of $\HHH(G, H)$, and in contrast, we will write $\II_{H v H}$ when we think of it as an element of $\C[G]$.  We will also be careful to indicate the subscript in $\ast_G$ when we mean multiplication in the group algebra (as opposed to the Hecke algebra).
\end{remark}

We will need some refinements regarding Hecke algebras taken with respect to different subgroups.

\subsubsection{Quotients}

Let $G$ be a group, $H \leq G$ a subgroup and suppose that $H = K \rtimes L$ for some subgroups $K$, $L \leq H$.  Our goal is to show that there is a natural surjection $\HHH(G, K) \twoheadrightarrow \HHH(G, H)$.  We note that since $K \leq H$, there is an obvious inclusion of vector spaces $\HHH(G, H) \leeqvs \HHH(G, K)$, when thought of as bi-invariant $G$-valued functions.  We will write $\ast_K$ and $\ast_H$ to denote the convolution product in the respective Hecke algebras.  From \eqref{gpalgHalg}, we easily see that for $\varphi_1$, $\varphi_2 \in \HHH(G, H)$
\begin{align} \label{convHK}
\varphi_1 \ast_K \varphi_2 = |L| ( \varphi_1 \ast_H \varphi_2).
\end{align}

Let $W = W^K \subseteq G$ be a set of double $K$-coset representatives such that $L \subseteq W$; note that for $\ell \in L$, $K \ell K = \ell K = K \ell$.  If $\ell \in L$ and $w \in W$, then it follows from \cite[Lemma 1.2]{Iwahori1964} that $T_\ell \ast_K T_w = T_{\ell w}$.  From this it is easy to see that
\begin{align*}
E = E_L := \frac{1}{|L|} \sum_{\ell \in L} T_\ell
\end{align*}
is an idempotent in $\HHH(G, K)$.  In fact, one notes that $E = |L|^{-1} \II_H = |L|^{-1} \1_{\HHH(G, H)}$, thought of as bi-invariant functions on $G$.

\begin{lemma} \label{suffcent}
	If $W^K \subseteq N_G(L)$ then $E$ is central in $\HHH(G, K)$.
\end{lemma}

\begin{proof}
	It is enough to show that for any $w \in W$, $E \ast_K T_w = T_w \ast_K E$.  One has \small
	\begin{align*}
	E \ast_K T_w & = \frac{1}{|L| |K|} \sum_{\ell \in L} T_\ell \ast_K T_w = \frac{1}{|H|} \sum_{\ell \in L} T_{\ell w} = \frac{1}{|H|} \sum_{\ell \in L} T_{w (w^{-1} \ell w)} = \frac{1}{|H|} \sum_{m \in L} T_{w m} = T_w \ast_K E. & \qedhere
	\end{align*} \normalsize 
\end{proof}

\begin{proposition} \label{p:Hprojection}
	If $E$ is central in $\HHH(G, K)$ then there exists a surjective algebra homomorphism $\HHH(G, K) \to \HHH(G, H)$ which takes $\1_{\HHH(G,K)}$ to $\1_{\HHH(G,H)}$, given by
	\begin{align*}
	\alpha \mapsto |L| ( E \ast_K \alpha).
	\end{align*}
\end{proposition}

\begin{proof}
	It is easy to check that this map is well-defined, i.e., that if $\alpha$ is $K$-bi-invariant, then $|L| ( E \ast_K \alpha)$ is $H$-bi-invariant.  To see that it preserves the convolution product, one uses the fact that $E$ is a central idempotent and \eqref{convHK} to see that
	\begin{align*}
	|L| \left( E \ast_K ( \alpha \ast_K \beta) \right) = |L| \left( (E \ast_K \alpha) \ast_K (E \ast_K \beta) \right) = |L| (E \ast_K \alpha) \ast_H |L| (E \ast_K \beta).
	\end{align*}
	By the remark preceding Lemma \ref{suffcent}, this map preserves the identity.  Finally, it is surjective, for given $\varphi \in \HHH(G, H)$, as mentioned above, we may think of it as an element of $\HHH(G, K)$ and we find $|L|^{-1} \varphi \mapsto \alpha$.
\end{proof}

\subsubsection{Inclusions} \label{Heckeinclusions}

Suppose now that $G$ is a group $L$, $H \leq G$ are subgroups and let $K := H \cap L$.  Assume $H = K \ltimes U$ for some subgroup $U \leq G$ and that $L \leq N_G(U)$.  We write $\ast_K$ and $\ast_H$ for the convolution products in $\HHH(L, K)$ and $\HHH(G, H)$, respectively.

\begin{lemma} \label{HKcosets}
	Suppose $x$, $y \in L$ are such that $x \in HyH$.  Then $x \in K y K$.
\end{lemma}

\begin{proof}
	We write $x = h_1 y h_2$ for some $h_1$, $h_2 \in H$.  Since $H = K \ltimes U$, we may write $h_1 = k u$ for some $k \in K$, $u \in U$.  Then $x = k y (y^{-1} u y) h_2$, but since $y \in L \leq N_G(U)$, $v := y^{-1} u y \in U \leq H$ and so $vh_2 \in H$.  But also $vh_2 = (ky)^{-1} x \in L$, so $vh_2 \in K = H \cap L$, and hence $x = k y (v h_2) \in K y K$.
\end{proof}

\begin{proposition} \label{p:Hinclusion}
	One has an inclusion of Hecke algebras $\HHH(L, K) \hookrightarrow \HHH(G, H)$, taking $\1_{\HHH(L,K)}$ to $\1_{\HHH(G,H)}$, given by $\varphi \mapsto \varphi^H$, where
	\begin{align*}
	\varphi^H(g) := \frac{1}{|K|} \sum_{x \in L} \varphi(x) \II_H(x^{-1} g).
	\end{align*}
\end{proposition}

\begin{proof}
	It is clear that this is a map of vector spaces.  To show that it preserves multiplication, let $\varphi_1$, $\varphi_2 \in \HHH(L,K)$.  Then \small
	\begin{align*}
	(\varphi_1^H \ast_H \varphi_2^H)(g) & = \frac{1}{|H|} \sum_{a \in G} \varphi_1(a) \varphi_2 ( a^{-1} g) = \frac{1}{|H||K|^2} \sum_{ \substack{ x,y \in L \\ a \in G}} \varphi_1(y) \II_H( y^{-1}a) \varphi_2 ( x) \II_H(x^{-1} a^{-1} g) \\
	& = \frac{1}{|U||K|^3} \sum_{\substack{ x,y \in L \\ a \in yH}} \varphi_1(y) \varphi_2 (x) \II_H(x^{-1} a^{-1} g) = \frac{1}{|U||K|^3} \sum_{ \substack{x,y \in L \\ h \in H}} \varphi_1(y) \varphi_2 (x) \II_H(x^{-1} h^{-1} y^{-1} g) \\
	& = \frac{1}{|U||K|^3} \sum_{ \substack{x,y \in L \\ k \in K,  u \in U}} \varphi_1(y) \varphi_2 (x) \II_H(x^{-1} k^{-1} u^{-1} y^{-1} g) .
	\end{align*} \normalsize
	Making the substitution $x = k^{-1} z$, this becomes \small
	\begin{align*}
	(\varphi_1^H \ast_H \varphi_2^H)(g) & = \frac{1}{|U||K|^3} \sum_{ \substack{y, z \in L \\ k \in K,  u \in U}} \varphi_1(y) \varphi_2 ( k^{-1} z ) \II_H(z^{-1} u^{-1} y^{-1} g) \\
	& = \frac{1}{|U||K|^2} \sum_{ \substack{y, z \in L \\ u \in U}} \varphi_1(y) \varphi_2 ( z ) \II_H \left( (z^{-1} u^{-1} z) z^{-1} y^{-1} g \right)  = \frac{1}{|K|^2} \sum_{ y, z \in L } \varphi_1(y) \varphi_2(z) \II_H \left( z^{-1} y^{-1} g \right)  
	\end{align*} \normalsize
	and now letting $z = y^{-1}x$, 
	\begin{align*}
	(\varphi_1^H \ast_H \varphi_2^H)(g) & = \frac{1}{|K|^2} \sum_{ x,y \in L } \varphi_1(y) \varphi_2( y^{-1} x ) \II_H \left( x^{-1} g \right)  = \frac{1}{|K|} \sum_{ x \in L } (\varphi_1 \ast_K \varphi_2)(x) \II_H \left( x^{-1} g \right)  \\
	& = (\varphi_1 \ast_K \varphi_2)^H(g). 
	\end{align*}
	It is easy to see that $\1_{\HHH(L,K)}^\H = \II_K^H = \II_H = \1_{\HHH(G,H)}$, so we do indeed get a map of algebras.
	
	To see that it is injective, let $V \subseteq L$ be a set of double $K$-coset representatives, so that $\{ T_x^{\HHH(L,K)} \}_{x \in V}$ is a basis of $\HHH(L, K)$.  Lemma\ref{HKcosets} says that if $W$ is a set of double $H$-coset representatives in $G$, then we may take $V \subseteq W$.  We write $\{ T_w^{\HHH(G,H)} \}_{w \in W}$ for the corresponding basis of $\HHH(G, H)$, with the subscripts denoting which Hecke algebra the element lies in.  Then we observe that
	\begin{align*}
	(T_x^{\HHH(L,K)})^H(x) = \frac{1}{|K|} \sum_{y \in L} T_x^{\HHH(L,K)}(y) \II_H(y^{-1}x) = \frac{1}{|K|} \sum_{y \in KxK} \II_H(y^{-1}x) > 0.
	\end{align*}
	This says that the $T_x^{\HHH(G,H)}$-component of $(T_x^{\HHH(L,K)} )^H$ is non-zero, but since $V \subseteq W$, the set $\big\{ T_x^{\HHH(G,H)} \big\}_{x \in V}$ is linearly independent, and hence so is the image $\big\{ (T_x^{\HHH(L,K)})^H \big\}_{x \in V}$ of the basis of $\HHH(L, K)$.
\end{proof}

\subsection{Iwahori--Hecke Algebras of type \texorpdfstring{$ A_{n-1} $}{An-1} }
\label{s:IHAlg}

Let $\G$ be the algebraic group $\GL_n$ defined over the finite field $\F_q$. % with $q$ elements. 
Let $\T \leeq \G$ be the maximal split torus of diagonal matrices.  There will be a corresponding root system with Weyl group $\SS_n$, the symmetric group on $n$ letters, which we will identify with the group of permutation matrices.  Let $\B \leeq \G$ be the Borel subgroup of upper triangular matrices.  	Let the finite dimensional algebra $\HHH(\G, \B)$ be as defined above.  
%We will assume chosen a set of representatives of $\SS_n$, so that 
The Bruhat decomposition for ${\G}$, with respect to $\B$, allows us to think of $\SS_n$ as a set of double $\B$-coset representatives, and hence $\{ T_w \}_{w \in \SS_n}$ gives a basis of $\HHH(\G, \B)$.  Furthermore, the choice of $\B$ determines a set of simple reflections $\{ s_1, \ldots, s_{n-1} \} \subseteq \SS_n$; we will write $T_i := T_{s_i}$.  The main result of \cite[Theorem 3.2]{Iwahori1964} gave the following characterization of $\HHH(\G, \B)$ in terms of generators and relations.
\begin{enumerate}[(a)]
	\item If $w = s_{i_1} \cdots s_{i_k}$ is a reduced expression for $w \in \SS_n$, then $T_w = T_{i_1} \cdots T_{i_k}$.  Hence $\HHH(\G, \B)$ is generated as an algebra by $T_1, \ldots, T_\ell$.
	\item For $1 \leq i \leq \ell$, $T_i^2 = qT_i + (q-1) 1$.
\end{enumerate}

\subsubsection{A generic deformation}
If $u$ is an indeterminate over $\C$, we may consider the $\C[u^{\pm 1}]$-algebra $\HH_n$ generated by elements $\mathsf{T}_1, \ldots, \mathsf{T}_{n-1}$ subject to the relations
\begin{enumerate}[(a)]
	\item $\mathsf{T}_i \mathsf{T}_j = \mathsf{T}_j \mathsf{T}_i$ for all $i,j=1,\ldots, n-1$ such that $|i-j|>1$;
	\item $\mathsf{T}_i \mathsf{T}_{i+1} \mathsf{T}_i = \mathsf{T}_{i+1} \mathsf{T}_i \mathsf{T}_{i+1}$ for all $i=1,\ldots, n-2$;
	\item $\mathsf{T}_i^{2} = u\cdot 1 + (u-1) \mathsf{T}_i$;
\end{enumerate}
$\HH_n$ is called the \emph{generic Iwahori--Hecke algebra of type $A_{n-1}$ with parameter $u$}.  Setting $u = 1$, we see that the generators satisfy those of the generating transpositions for $\SS_n$ and \cite[Theorem 3.2]{Iwahori1964} shows that for $u=q$ these relations give the Iwahori--Hecke algebra above, so (cf.\ \cite[(68.11) Proposition]{CurtisReiner})
\begin{align} \label{Iwahorispec}
\HH_n(1) & \cong \C[\SS_n] & \HH_n(q) & \cong \HHH(\G, \B).
\end{align}

%(I do not think we need this paragraph Tamas 1.1.15) %The current definition of an Iwahori--Hecke algebra is an axiomatization of this result \cite[Definition 4.4.1]{GeckPfeiffer}.  Let $(\W, S)$ be a Coxeter system, let $A$ be a base ring and for $s \in S$, choose $a_s, b_s \in A$ such that $(a_s, b_s) = (a_t, b_t)$ whenever $s, t \in S$ are conjugate in $\W$.  Then the \emph{Iwahori--Hecke algebra of $(\W,S)$ over the ring $A$ and with parameters $\{ a_s, b_s \}_{s \in S}$} is defined as the $A$-algebra with generators $\{ T_s \}_{s \in S}$ and relations
%\begin{enumerate}[(a)]
%\item If $s \neq t \in S$ and $m \in \Z_{\geq 1}$ is minimal with $(st)^m = 1$ in $\W$, then
%\begin{align*}
%T_s T_t T_s \cdots T_s = T_t T_s T_t \cdots T_t,
%\end{align*}
%where each side has $m$ factors.
%\item For $s \in S$, 
%\begin{align*}
%T_s^2 = a_s 1 + b_s T_s.
%\end{align*}
%\end{enumerate}
%It is clear that we recover $\HHH(G, B)$ from the above definition when $(\W, S)$ is the Coxeter system associated to the Weyl group of $(\bG, \bT)$, $A = \C$ and $a_s = q, b_s = q-1$ for all $s \in S$.

\subsection{Yokonuma--Hecke Algebras} \label{s:yokoyoko}

Let $\T \leeq \B \leeq \G$ be as in the previous example, let $\U \leeq \B$ 
be the unipotent radical of $\B$, namely, the group of upper triangular unipotent matrices. 
%and let
%\begin{align*}
%\U := \bU(\F_q)
%\end{align*}
%be its group of $\F_q$-points.  
Then the algebra $\HHH(\G, \U)$, first studied in \cite{yokonuma1}, is called the \emph{Yokonuma--Hecke algebra} associated to $\G, \B, \T$.  Let $N(\T)$ be the normalizer of $\T$ in $\G$; $N(\T)$ is the group of the monomial matrices (i.e., those matrices for which each row and each column has exactly one non-zero entry) and one has $N(\T) =  \T \rtimes \SS_n$, where the Weyl group $\SS_n$ acts by permuting the entries of a diagonal matrix.  We will often write $\rN$ for $N(\T)$.  By the Bruhat decomposition, one may take $\rN \leeq \G$ as a set of double $\U$-coset representatives.  Section \ref{Halgnotn} describes how $\HHH(\G,\U)$ has a basis $\{ T_v, v \in N(\T) \}$.

\subsubsection{A generic deformation} \label{altpres}
%
%In the particular case where $\bG = \GL_n$, $\bT$ is the subgroup of diagonal matrices and $\bB$ is the 
%subgroup of upper triangular invertible matrices, 

The algebra $\HHH(\G,\U)$ has a presentation in terms of generators and relations due to \cite{yokonuma1,yokonuma2}, 
which we now describe and which we will make use of later on.  For this consider the $\C[u^{\pm 1}]$-algebra 
$\YY_{d,n}$ generated by the elements
\begin{align*}
\mathsf{T}_i ,\: i=1,\ldots,n-1 \mbox{ and } \mathsf{h}_j \mbox{ for } j=1,\ldots, n
\end{align*}
subject to the following relations:
\begin{enumerate}[(a)]
	\item $\mathsf{T}_i \mathsf{T}_j = \mathsf{T}_j \mathsf{T}_i$ for all $i,j=1,\ldots, n-1$ such that $|i-j|>1$;
	\item $\mathsf{T}_i \mathsf{T}_{i+1} \mathsf{T}_i = \mathsf{T}_{i+1} \mathsf{T}_i \mathsf{T}_{i+1}$ for all $i=1,\ldots, n-2$;
	\item $\mathsf{h}_i \mathsf{h}_j=\mathsf{h}_j \mathsf{h}_i$ for all $i,j=1,\ldots, n$;
	\item $\mathsf{h}_j \mathsf{T}_i=\mathsf{T}_i \mathsf{h}_{s_i(j)}$ for all $i=1,\ldots, n-1,$ and $j=1,\ldots, n,$ where $s_i := (i,i+1) \in \SS_n$;
	\item $\mathsf{h}_i^{d}=1$ for all $i=1,\ldots, n$;
	\item \label{quadratic} 
	%\begin{align}
	$\mathsf{T}_i^{2}=u\mathsf{f}_i\mathsf{f}_{i+1}+(u-1)  \mathsf{e}_i \mathsf{T}_i$
	%  T_i^{2}=u^{2}+(u^{2}-1)\pare{ \frac{1}{d}\sum_{j=1}^{d} \mathsf{h}_i^{j}\mathsf{h}_{i+1}^{-j} }T_i
	%\end{align}
	for $i=1,\ldots, n-1$, where
	\begin{align}\label{e:esubi}
	\mathsf{e}_i := \frac{1}{d}\sum_{j=1}^{d} \mathsf{h}_i^{j}\mathsf{h}_{i+1}^{-j} 
	\end{align}
	for $ i=1,\ldots,n-1 ,$ and
	\begin{align}\label{e:fcases}
	\mathsf{f}_i := \begin{cases}
	\mathsf{h}_i^{d/2} 	& \textit{ for   $ d $ even, }\\
	1 & \textit{ for  $ d $ odd }
	\end{cases}
	\end{align}
	for $ i=1,\ldots,n .$
	%
	%\begin{align*}\label{e:fcases}
	%\mathsf{f}_i := \mathsf{h}_i^{ {d\choose 2} }
	%\end{align*}
	
\end{enumerate}

%Following the notation of , we will often shorten $\C(u) \otimes_{\C[u^{\pm}]} Y_{d,n}(u)$ to $\C(u) Y_{d,n}(u)$.

%\begin{remark}\label{r:matsu}
%By Matsumoto's Theorem (cf. 1.2.2 in ~\cite{GeckPfeiffer}), there is a well defined
%$T_w := T_{s_{i_1}}T_{s_{i_2}}\ldots T_{s_{i_k}}$ where $w={s_{i_1}}{s_{i_2}}\ldots{s_{i_k}}$ is any reduced 
%expression, and $T_{\tilde{s}_i}:=T_i.$
%\end{remark}

In Theorem~\ref{t:yokonuma} below, we will see that $\HHH(\G, \U)$ arises as the specialization $\YY_{q-1,n}(q)$.

\begin{remark}\label{r:modernyoko}
	The modern definition of $\YY_{d,n} $ in terms of generators and relations takes $\mathsf{f}_i = 1$,
	regardless of the parity of $ d $ (cf. \cite{cp} and \cite{juyumaya3}).  We decided to take that of \cite{yokonuma1} so that the meaning of the generators is more transparent.  Again, this will be clearer from Theorem~\ref{t:yokonuma} and its proof.
	%A direct way to see this when $ d $ is a multiple of $ 4 $ is taking
\end{remark}

\subsection{Some computations in \texorpdfstring{$\HHH(\G, \U)$}{H(G,U)} and \texorpdfstring{$\YY_{d,n}(u)$}{Yd,n(u)}}
%
%We now focus on the algebra $\HHH(\G, \U)$ in the case $\bG = \GL_n$, $\bT$ is the subgroup of diagonal
% matrices, $N(\bT)$ its normalizer, $\bB$ the subgroup of upper triangular matrices, $\bU$ its unipotent
%  radical and $\G = \bG(\F_q), \T = \bT(\F_q), \U = \bU(\F_q), N(\T) = N(\bT)(\F_q)$.  

We continue with the context of the previous subsection.  There is a canonical surjection $p : \rN \to \SS_n$, which allows us to define the length of an element $v \in \rN$ as that of $p(v)$; we will denote this by $\ell(v)$.

\begin{lemma} \label{heckemult}
	If $v_1$, $v_2 \in \rN$ are such that $\ell(v_1 v_2) = \ell(v_1) + \ell(v_2)$, then $T_{v_1} \ast T_{v_2} = T_{v_1 v_2}$.  In particular, if $h \in \T$, $v \in \rN$ then $T_h \ast T_v = T_{hv}$.
\end{lemma}

\begin{proof}
	This follows readily from \cite[Lemma 1.2]{Iwahori1964} using the fact that, in the notation there, for $v \in \rN$, $\textnormal{ind}(v) = q^{\ell(v)}$.
\end{proof}

Our first task is to describe the relationship of $\YY_{d,n}$ to $\HHH(\G, \U)$, as alluded to at the beginning of Section \ref{altpres}.  To do this, we follow the approach in \cite[(7.4),(8.1.6)]{GeckPfeiffer}.  %Given a map $\theta_z: \C[u^{\pm 1}] \to \C$ defined by $u \mapsto z$ for some complex number $z\in \C^\times$, we define the \emph{$u=z$ specialization of $Y_{d,n}(u)$} as $Y_{d,n}(z) := \C \otimes_{\theta_z} Y_{d,n}(u)$.  %We point out that $\C[u^{\pm 1}]$ is integrally closed in $\C(u).$ %ref de split 
For example, the $u=1$ specialization of $\YY_{q-1,n}$ gives 
\beq\label{spec1} \YY_{q-1,n}(1)=\C \otimes_{\theta_1} \YY_{q-1,n}\cong \C[\rN]=\C[(\F_q^\times)^n \rtimes \SS_n] \eeq
the group algebra of the normalizer in $\G$ of the torus $T(\F_q).$

%\begin{theorem} \label{t:yokonuma} Let $q$ be a prime power such that $4|q-1$ so that we can fix a primitive fourth root of unity $\zeta_4\in \F_q^\times$.  For $t\in \F_q^\times$, let $h_{i}(t)\in \T$ be the diagonal matrix obtained by replacing the $i$th diagonal entry of the identity matrix with $t$.  We also let $s_i \in \rN$ denote the permutation matrix corresponding to $(i, i+1)$.  Finally, we set
%	\begin{align*}
%	\omega_i := s_i h_i(\zeta_4^{-1}) h_{i+1}(\zeta_4) \in \rN.
%	\end{align*}
%	Then one has an isomorphism of $\C$-algebras $\YY_{q-1,n}(q) \cong \HHH(\G, \U)$. Under this isomorphism $T_{\omega_i} \in \HHH(\G, \U)$ maps to $\mathsf{T}_i \in \YY_{q-1,n}(q)$. 
%\end{theorem}

\begin{theorem} \label{t:yokonuma} Let $q$ be a prime power and fix a multiplicative generator $t_g\in \F_q^\times$.  For $t\in \F_q^\times$, let $h_{i}(t)\in \T$ be the diagonal matrix obtained by replacing the $i$th diagonal entry of the identity matrix by $t$.  Finally, we let $s_i \in \rN$ denote the permutation matrix corresponding to $(i, i+1)$ and
	\begin{align*}
	\omega_i := s_i h_i({-1}) = h_{i+1}(-1) s_i \in \rN.
	\end{align*}
	Then one has an isomorphism of $\C$-algebras $\YY_{q-1,n}(q) \cong \HHH(\G, \U)$ under which
	\begin{align*}
	\mathsf{T}_i  & \mapsto T_{\omega_i} \in \HHH(\G, \U) & & \text{ and } & 
	\mathsf{h}_i & \mapsto T_{h_{i}(t_g)}\in \HHH(\G, \U). 
	\end{align*}
\end{theorem}

%Note that $\omega_i$ is the matrix obtained by replacing the $(2 \times 2)$-submatrix of the identity matrix formed by the $i$th and $(i+1)$st rows and columns by
%\begin{align*}
%\left[ \begin{array}{cc}
%0 & \zeta_4^{-1} \\ \zeta_4 & 0
%\end{array} \right] = \zeta_4 \left[ \begin{array}{cc}
%0 & -1 \\ 1 & 0
%\end{array} \right].
%\end{align*}

Note that $\omega_i$ is the matrix obtained by replacing the $(2 \times 2)$-submatrix of the identity matrix formed by the $i$th and $(i+1)$st rows and columns by
\begin{align*}	
\left[ \begin{array}{cc}
0 & 1 \\ -1 & 0
\end{array} \right] .
\end{align*}

\begin{proof}
	It is sufficient to show that $T_{\omega_i}$, $T_{h_j(t_g)}$ satisfy the relations prescribed for $\mathsf{T}_i, \mathsf{h}_j$ in Section \ref{altpres}.  Lemma \ref{heckemult} makes most of these straightforward.  
	%For example, with the Lemma, checking relation (c) reduces to checking that 
	%$\omega_i \omega_{i+1} \omega_i = \omega_{i+1} \omega_i \omega_{i+1}$.  
	The computation in \cite[Th\'eor\`eme 2.4${}^\circ$]{yokonuma1} gives
	\begin{align*}
	T_{\omega_i}^2=qT_{h_i(-1) h_{i+1}({-1})} + \sum_{t\in \F_q^\times} T_{h_i(t) h_{i+1}(t^{-1})} T_{\omega_i},
	\end{align*}
	%Thanks to the Lemma it becomes
	%\begin{align*}
	%T_{s_i}^2=qT_e +\left( \sum_{t\in \F_q^\times} T_{h_i(t) h_{i+1}(t^{-1})} \right) T_{h_{i}(-1)} T_{s_i},
	%\end{align*}
	which is relation \eqref{quadratic} . %(f).
	%Noting that $T_{\tilde{s}_i}^2=T_{\omega_i}^2 $ we see that 
	%$$T_{\tilde{s}_i}^2=q+\sum_{t\in \F_q^\times} T_{h_i(\zeta_4^{-1}t)h_{i+1}(\zeta_4t^{-1})} T_{\tilde{s}_i}=q+\sum_{t\in \F_q^\times} T_{h_i(t)h_{i+1}(t^{-1})} T_{\tilde{s}_i}.$$ 
	%then % and consider $$h_i:=\prod_j^n T_{h_{j}(t_g)}\in \HHH(G,U).$$
	% \cite[Theorem 2]{yokonuma1} and \cite[Theorem 4.8]{yokonuma2} implies our version (cf.\ also \cite[Theorem 2]{juyumaya}) with $T_{\omega_i}\in \HHH(G,U)$ corresponding to  $\mathsf{T}_i\in  \YY_{q-1,n}(q)$ and $T_{h_i(t_g)}\in \HHH(G,U)$  corresponding to $\mathsf{h}_i \in \YY_{q-1,n}(q)$.
\end{proof}

The longest element $w_0 \in \W = \SS_n$ is the permutation $\prod_{i=1}^{ \lfloor n/2 \rfloor} (i, n+1-i)$ (the order of the factors is immaterial since this is a product of disjoint transpositions) and is of length $\binom{n}{2}$.  We may choose a reduced expression 
\begin{align} \label{longestelt}
w_0 = s_{i_1} \cdots s_{i_{\binom{n}{2}}}.
%\omega_0 = s_{i_1} \cdots s_{i_{\binom{n}{2}}}.
\end{align}
%
%Let $\sigma_i \in \rN$ be the permutation matrix corresponding to $s_i$, that is, the matrix obtained by replacing the $(2 \times 2)$-submatrix of the identity matrix formed by the $i$th and $(i+1)$st rows and columns by
%\begin{align*}
%\left[ \begin{array}{cc}
%0 & 1 \\ 1 & 0
%\end{array} \right].
%\end{align*}
With the same indices as in \eqref{longestelt}, we define
\begin{align} \label{specialelements}
\omega_0 & := \omega_{i_1} \cdots \omega_{i_{\binom{n}{2}}} \in \rN, & & \textnormal{ and } &
%\sigma_0 & := \sigma_{i_1} \cdots \sigma_{i_{\binom{n}{2}}} \in \rN & & \textnormal{ and } &
\mathsf{T}_0 & := \mathsf{T}_{i_1} \cdots \mathsf{T}_{i_{\binom{n}{2}}} \in \YY_{d,n}.%(u).
\end{align}
Using the braid relations (a) and (b) and arguing as for Matsumoto's Theorem \cite[Theorem 1.2.2]{GeckPfeiffer}, one sees that $\mathsf{T}_0$ is independent of the choice of reduced expression \eqref{longestelt}.  Now, Lemma \ref{heckemult} shows that
\begin{align*}
T_{\omega_0} = T_{\omega_{i_1}} \cdots T_{\omega_{i_{\binom{n}{2}}}} \in \HHH(\G, \U)
%T_{\omega_0} = T_{s_{i_1}} \cdots T_{s_{i_{\binom{n}{2}}}} \in \HHH(\G, \U)
\end{align*}
and Theorem \ref{t:yokonuma} shows that this corresponds to $\mathsf{T}_0 \in \YY_{q-1,n}(q)$.

%\begin{lemma} Under the isomorphism $\HHH(G,U)\cong \YY_{q-1,n}(q)$ of Theorem~\ref{t:yokonuma} the operator $T_w\in \YY_{q-1,n}(q)$ defined above matches up with the Hecke operator $T_{\tilde{w}}\in \HHH(G,U)$ corresponding to $\tilde{w}={\tilde{s}_{i_1}}{\tilde{s}_{i_2}}\ldots {\tilde{s}_{i_k}}\in N(T)$ in other words   $$T_{\tilde{w}} = T_{\tilde{s}_{i_1}}T_{\tilde{s}_{i_2}}\ldots T_{\tilde{s}_{i_k}}$$ where $w={s_{i_1}}{s_{i_2}}\ldots{s_{i_k}}\in \W$ is a reduced expression for $w$. 
%\end{lemma}
%\begin{proof}
%\end{proof}

%When $d$ is divisible by $4$ we have 
%introduced operators $T_i$ in Remark~\ref{r:cpquadratic}. With the same construction as above we can define $T_{\tilde{w}_0}$.

%\begin{lemma} \label{tildesame} When $4|d$  we have $T_{\tilde{w}_0}^2=T_{\tilde{w}_0}^2$.
%\end{lemma}
%\proof First let $w_k:=s_1s_2\cdots s_{k}$. Then by definition \eqref{tildet} we have
%\bes T_{w_k}=T_1T_2\cdots T_{k}=\\ =T_1h_1^{-\binom{d}{2}/2}h_{2}^{\binom{d}{2}/2}T_2h_2^{-\binom{d}{2}/2}h_{3}^{\binom{d}{2}/2}\cdots T_{k}h_{k}^{-\binom{d}{2}/2}h_{k}^{\binom{d}{2}/2}=\\=T_1T_2\cdots T_{k}h_1^{-\binom{d}{2}/2}h_2^{-\binom{d}{2}/2}\cdots h_{k}^{-\binom{d}{2}/2}h_k^{\binom{d}{2}(k/2)}. \ees
%We note that $T_w h_i=h_{w(i)} T_w$ for $w\in \W$ and $w_0=w_{n-1}w_{n-2}\cdots w_1$ is a reduced expression thus $$T_{\tilde{w}_0}=T_{\tilde{w}_0} h_1^{\binom{d}{2}(-n/2)} h_2^{\binom{d}{2}(1-n/2)}\cdots h_n^{\binom{d}{2}(n/2)}$$ and in turn the result follows.  
%\endproof

\begin{lemma}\label{l:t2central}
	The element $\mathsf{T}_0^2$ is central in $\YY_{d,n}$.  It follows, by specialization, that $T_{\omega_0}^2$ is central in $\HHH(\G, \U)$.
\end{lemma}

\proof
Proceeding as in~\cite[\S~4.1]{GeckPfeiffer}, we define a monoid $\mathtt{B}^+$ generated by 
\begin{align*}
\mathtt{T}_1, \ldots, \mathtt{T}_{n-1}, \mathtt{h}_1, \ldots, \mathtt{h}_n
\end{align*}
and subject to the relations
\begin{enumerate}[(a)]
	\item $\mathtt{T}_i \mathtt{T}_j = \mathtt{T}_j \mathtt{T}_i$ for $1 \leq i,j \leq n-1$ with $|i-j|>1$;
	\item $\mathtt{T}_i \mathtt{T}_{i+1} \mathtt{T}_i = \mathtt{T}_{i+1} \mathtt{T}_i \mathtt{T}_{i+1}$ for $1 \leq i \leq n-2$;
	\item $\mathtt{h}_j \mathtt{T}_i = \mathtt{T}_i \mathtt{h}_{s_i(j)}$ for $1 \leq i \leq n-1, 1 \leq j \leq n$, where $s_i := (i,i+1) \in \SS_n$.
\end{enumerate}
Observe that these are simply relations (a), (b) and (d) of those 
given for $\YY_{d,n}$ in Section~\ref{altpres}; one can define the 
monoid algebra $\C[u^{\pm 1}][\mathtt{B}^+]$ of which $\YY_{d,n}$ will be a quotient via the mapping
\begin{align*}
\mathtt{T}_i & \mapsto \mathsf{T}_i & \mathtt{h}_j & \mapsto \mathsf{h}_j,
\end{align*}
for $1 \leq i \leq n-1, 1 \leq j \leq n$.

Letting
\begin{align*}
\mathtt{T}_0 := \mathtt{T}_{i_1} \cdots \mathtt{T}_{i_{\binom{n}{2}}} \in \C[u^{\pm 1}][\mathtt{B}^+],
\end{align*}
it is enough to show that $\mathtt{T}_0$ is central in $\C[u^{\pm 1}][\mathtt{B}^+]$.  Arguing as in the proof of~\cite[Lemma 4.1.9]{GeckPfeiffer}, one sees that $\mathsf{T}_0^2$ commutes with each $\mathsf{T}_i, 1 \leq i \leq n-1$.  Furthermore, relation (c) and \eqref{l:secondcount} give
\begin{align*}
\mathtt{h}_j \mathtt{T}_0^2 = \mathtt{T}_0^2 \mathtt{h}_{w_0^2(j)} = \mathtt{T}_0^2 \mathtt{h}_j,
\end{align*}
noting that $w_0^2 = 1$ implies $w_0 = w_0^{-1} = s_{i_{\binom{n}{2}}} \cdots s_{i_1}$.
%we recognize $Y_{d,n}(u)$ as a quotient of the monoid algebra of the semidirect product of the torus and the braid monoid $B^{+}=B^{+}(\SS_n, \set{s_i }_{1\leq i\leq n-1})$ with respect to the natural permutation action.  Since $\tilde{w}_0^2$ acts trivially on the torus, and is central in $B^{+},$ the lemma follows.
\endproof

%\begin{remark}\label{Ts2centraltoo}
%With the same notation as above, 
%Define
%\begin{align*}
%T_{s_0}: = T_{s_{i_1}} \cdots T_{s_{i_{\binom{n}{2}}}} \in \HHH(\G, \U).
%\end{align*}
%The same argument proves that $ T_{s_0}^2 $ is central in $\HHH(\G , \U) .$
%\end{remark}

The following will be useful when we look at representations.

\begin{lemma} \label{eiTicomm}
	The element $\mathsf{e}_i$ defined in \eqref{e:esubi} commutes with $\mathsf{T}_i$ in $\YY_{d,n}(u)$.
\end{lemma}

\proof
By relation (d), we see that $\mathsf{T}_i( \mathsf{h}_i^j \mathsf{h}_{i+1}^{-j} ) =  (\mathsf{h}_i^{-j} \mathsf{h}_{i+1}^j) \mathsf{T}_i$.  The Lemma follows by averaging over $j = 1, \ldots, d$ and observing that both $\mathsf{h}_i$ and $\mathsf{h}_{i+1}$ have order $d$.
\endproof

\begin{lemma} \label{Yconj}
	The elements $\mathsf{T}_1, \ldots, \mathsf{T}_{n-1} \in \YY_{d,n}(u)$ are all conjugate.
\end{lemma}

\begin{proof}
	From relation \eqref{quadratic} in Section \ref{altpres}, we see that each $\mathsf{T}_i$ is invertible, with inverse $$\mathsf{T}_i^{-1} = u^{-1}( \mathsf{T}_i - (u-1) \mathsf{e}_i )\mathsf{f}_i\mathsf{f}_{i+1}.$$  Then
	\begin{align*}
	(\mathsf{T}_i \mathsf{T}_{i+1} \mathsf{T}_i) \mathsf{T}_{i+1} (\mathsf{T}_i \mathsf{T}_{i+1} \mathsf{T}_i)^{-1} = (\mathsf{T}_i \mathsf{T}_{i+1} \mathsf{T}_i) \mathsf{T}_{i+1} (\mathsf{T}_{i+1} \mathsf{T}_i \mathsf{T}_{i+1})^{-1} = \mathsf{T}_i,
	\end{align*}
	and the statement follows by transitivity of the conjugacy relation.
\end{proof}

\subsection{The double centralizer theorem}

The following is taken from \cite[\S~3.2]{KraftProcesi}.  Let $\K$ be an arbitrary field, $A$ a finite-dimensional algebra over $\K$ and let $W$ be a finite-dimensional (left) $A$-module.  Recall that $W$ is said to be \emph{semisimple} if it decomposes as a direct sum of irreducible submodules.  If $A$ is semisimple as a module over itself then it is called a \emph{semisimple algebra}; the Artin--Wedderburn theorem then states that any such $A$ is a product of matrix algebras over (finite-dimensional) division $\K$-algebras.  If $U$ is a finite-dimensional simple $A$-module, then the \emph{isotypic component of $W$ of type $U$} is the direct sum of all submodules of $W$ isomorphic to $U$.  The isotypic components are then direct summands of $W$ and their sum gives a decomposition of $W$ precisely when the latter is semisimple; in this case, it is called the \emph{isotypic decomposition of $W$}.  We recall the following, which is often called the ``double centralizer theorem.''

\begin{theorem} \label{doublecentralizer}
	Let $W$ a finite-dimensional vector space over $\K$, let $A \subseteq \End_\K M$ be a semisimple subalgebra and let
	\begin{align*}
	A' = \End_A W := \{ b \in \End_\K W \, : \, ab = ba \ \forall a \in A \},
	\end{align*}
	be its centralizer subalgebra.  Then $A'$ is also semisimple and there is a direct sum decomposition
	\begin{align*}
	W = \bigoplus_{i=1}^r W_i
	\end{align*}
	which is the isotypic decomposition of $W$ as either an $A$-module or an $A'$-module.  In fact, for $1 \leq i \leq r$, there is an irreducible $A$-module $U_i$ and an irreducible $A'$-module $U_i'$ such that if $D_i := \End_A U_i$ (this is a division $\K$-algebra by Schur's lemma), then $\End_{A'} U_i' \cong D^\opp$ and 
	\begin{align*}
	W_i \cong U_i \otimes_{D_i} U_i'.
	\end{align*}
\end{theorem}

\begin{remark}
	In the case where $\K$ is algebraically closed, then there are no non-trivial finite-dimensional division algebras over $\K$, and so in the statement above, the tensor product is over $\K$.
\end{remark}

We are interested in the case where $\K = \C$ (so that we are within the scope of the Remark), $H$ and $G$ are as in Section \ref{Halgnotn}, $W = \Ind_H^G \1_H$ is the induction of the trivial representation of a subgroup $H \leq G$ to $G$ and $A$ is the image of the group algebra $\C[G]$ in $\End_\C W$.  Then $A$ is semisimple and via Proposition~\ref{Halgcomm}, we know $A' = \HHH(G, H)$.  We can then conclude the following.

\begin{corollary}
	The Hecke algebra $\HHH(G, H)$ is semisimple.
\end{corollary}

Furthermore, one observes that the kernel of the induced representation $\Ind_H^G \1_H$ is given by $\bigcap_{g \in G} gHg^{-1}$.  Thus, applying Theorem \ref{doublecentralizer} with $A = \HHH(G, H)$, since its commuting algebra is the image of the $G$-action, we may state the following.

\begin{corollary} \label{groupalgHecke}
	If $\bigcap_{g \in G} gHg^{-1}$ is trivial, one has 
	\begin{align*}
	\C[G] \cong \End_{\HHH(G, H)} \Ind_H^G \1_H.
	\end{align*}
\end{corollary}

\subsection{Representations of Hecke algebras} \label{s:RepsHA}

We return to the abstract situation of Section~\ref{Halgnotn}.  By a \emph{representation of $\HHH(G,H)$} we will mean a pair $(W, \rho)$ consisting of a finite-dimensional complex vector space $V$ and an identity-preserving homomorphism $\rho : \HHH(G, H) \to \End_\C W$.  Let $(V, \pi)$ be a representation of $G$ and let $V^H \subseteq V$ be the subspace fixed by $H$.  Then $V^H$ is a representation of the Hecke algebra $\HHH(G,H)$ via the action
\begin{align} \label{Heckerepdef}
\varphi . v := \frac{1}{|H|} \sum_{a \in G} \varphi(a) \pi(a) \cdot v
\end{align}
for $\varphi \in \HHH(G,H), v \in V^H$.  It is easy to check that $\varphi . v \in V^H$ so that this is well-defined.  

Note that upon choosing a basis vector for the trivial representation $\1_H$ of $H$, we may identify
\begin{align*}
\Hom_H(\1_H, \Res_H^G V) \cong V^H
\end{align*}
by taking an $H$-morphism to the image of the basis vector.  Thus, we have defined a map $\DC_H : \Rep G \to \Rep \HHH(G, H)$
\begin{align*}
(V, \pi) \mapsto \Hom_H(\1_H, \Res_H^G V) \cong V^H.
\end{align*}
Let us now set
\begin{align} \label{IrrGH}
\Irr(G : H) := \left\{ \zeta \in \Irr G \, : ( \zeta, \1_H^G) > 0 \right\},
\end{align}
where $( \, , )$ is the pairing on characters; the condition is equivalent to $\Hom_H( \1_H, \Res_H^G \zeta) \neq 0$.  We can now give the following characterisation of irreducible representations of $\HHH(G, H)$, which, in the more general case of locally compact groups, is \cite[Proposition 2.10]{bernstein-zelevinsky}.

\begin{proposition} \label{Heckerepgen}
	If $(V, \pi)$ is an irreducible representation of $G$, then $V^H$ is an irreducible representation of $\HHH(G,H)$, and every irreducible representation of $\HHH(G, H)$ arises in this way, that is, $\DC_H$ restricts to a bijection $\DC_H : \Irr(G : H) \xrightarrow{\sim} \Irr \HHH(G, H)$.
\end{proposition}

Since $\C[G]$ and  $\HHH(G,H)$ are semisimple, we can apply Theorem \ref{doublecentralizer} to $W = \Ind_H^G \1_H$.  If we denote the set of irreducible representations of $G$ by $\Irr G$, we find that 
\begin{align} \label{Inddecomp}
\Ind_H^G \1_H = \bigoplus_{ \substack{V \in \Irr G  \\ V^H \neq \{ 0 \}} } V \otimes V^H,
\end{align}
with elements of $G$ acting on the left side of the tensor product and those of $\HHH(G,H)$ acting on the right.  One has a consistency check here in that for an irreducible representation $V$ of $G$, the multiplicity of $V$ in $\Ind_H^G \1_H$ is given by
\begin{align*} 
\dim \Hom_G \left(\Ind_H^G \1_H, V \right) = \dim \Hom_H \left(\1_H, \Res_H^G V \right) = \dim V^H,
\end{align*}
which the decomposition in \eqref{Inddecomp} confirms.

\subsubsection{Traces} 

Let $M \cong \Ind_H^G \1_H$ be as in Section~\ref{Halgnotn}.  Observe that if $X \in \End_\C M$, then using the basis \eqref{basisnot}, its trace is computed as 
\begin{align*}
\tr_M X = \sum_{v \in V} (X.f_v)(v).
\end{align*}

\begin{lemma} \label{traceconj} 
	Let $g \in G$ and $\varphi \in \HHH(G, H)$ and consider $g \varphi = \varphi g \in \End_\C \Ind_H^G \1_H$ (where $g$ is thought of as an element of $\End_\C \Ind_H^G \1_H$ via (\ref{GactM})).  Then
	\begin{align*}
	\tr ( g \varphi) = \frac{1}{|H|} \sum_{x \in G} \varphi( x g x^{-1})=\sum_{ V \in \Irr(G:H) } \chi_V(g)  \chi_{\DC_H(V)}( \varphi),
	\end{align*}
	where $\chi_V$ is the character of the $G$-module $V$, and $\chi_{\DC_H(V)}$ is that of the $\HHH(G, H)$-module $\DC_H(V)$. 
\end{lemma}

\begin{proof}
	In the notation of \eqref{basisnot},
	\begin{align*}
	|H| \tr(g \varphi) & = |H| \sum_{v \in V} \big( (g \varphi). f_v \big)(v) = |H| \sum_{v \in V} (\varphi. f_v) (vg) = \sum_{v \in V} \sum_{y \in G} \varphi(y) f_v(y^{-1} vg) \\
	& = \sum_{v \in V} \sum_{y^{-1} vg \in Hv} \varphi(y) = \sum_{v \in V} \sum_{y \in vgv^{-1} H} \varphi(y) = \sum_{v \in V} \sum_{h \in H} \varphi(vgv^{-1} h^{-1}) \\
	& = \sum_{v \in V} \sum_{h \in H} \varphi\big( (hv)g (hv)^{-1}\big) = \sum_{x \in G} \varphi( xg x^{-1} ),
	\end{align*}
	where we use \eqref{GU} at the last line.  On the other hand, if $g$, $\varphi$ are as in the Lemma, then applying $g \varphi = \varphi g$ to the decomposition (\ref{Inddecomp}), we get
	\begin{align*} %\label{trexp}
	\tr(g \varphi) = \sum_{ V \in \Irr (G:H) } \chi_V(g)  \chi_{\DC_H(V)}( \varphi). 
	& \qedhere
	\end{align*} 
\end{proof} 

\subsubsection{Induced representations}

Let $G$, $H$, $L$, $K$ and $U$ be as in Section \ref{Heckeinclusions}.  Then one sees that $L \cap U$ is trivial and hence we may define $P := L \ltimes U$.  The inclusion $\HHH(L, K) \hookrightarrow \HHH(G, H)$ given by Proposition \ref{p:Hinclusion} allows us to induce representations from $\HHH(L, K)$ to $\HHH(G, H)$:  if $V$ is an $\HHH(L, K)$-representation, then 
\begin{align*}
\Ind_{\HHH(L, K)}^{\HHH(G, H)} V := \HHH(G, H) \otimes_{\HHH(L, K)} V,
\end{align*}
yielding a map $\Rep \HHH(L,K) \to \Rep \HHH(G, H)$.

We also have a ``parabolic induction'' functor:  for $V \in \Rep L$, we define
\begin{align} \label{parinddef}
R_L^G V := \Ind_P^G \Infl_L^P V,
\end{align}
which gives a map $\Rep L \to \Rep G$.  Let $\Rep(G : H)$ denote the set of isomorphism classes of (finite-dimensional) representations of $G$ at least one of whose irreducible components lies in $\Irr(G : H)$.  Then we claim that if $V \in \Irr(L:K)$, then $R_L^G V \in \Rep(G : H)$ and hence $R_L^G$ yields a map
\begin{align*}
R_L^G : \Irr(L:K) \to \Rep(G: H).
\end{align*}
We know that $V \in \Irr(L : K)$ if and only if $\Hom_K(\1_K, \Res_K^L V) \neq 0$.  The latter implies that 
\begin{align*}
0 \neq \Hom_H (\1_H, \Res_H^P \Infl_L^P V) = \Hom_P( \Ind_H^P \1_H, \Infl_L^P V).
\end{align*}
Now, inducing to $G$ in both factors\footnote{If $B \in \Irr P$, then $\1_{\Ind B} \in \End_G(\Ind_P^G B, \Ind_P^G B)$, so $\Hom_P (B, \Res_P^G \Ind_P^G B) \neq 0$, and thus $B$ is an irreducible component of $\Res_P^G \Ind_P^G B$.  It follows that if $A$ and $B$ are any $P$-representations with $\Hom_P(A, B) \neq 0$ then $\Hom_G( \Ind_P^G A, \Hom_P^G B) \neq 0$.} 
\begin{align*}
0 \neq \Hom_G( \Ind_P^G \Ind_H^P \1_H, \Ind_P^G \Infl_L^P V) = \Hom_G( \Ind_H^G \1_H, R_L^G V) = \Hom_H(\1_H, \Res_H^G R_L^G V),
\end{align*} 
which proves the claim.

Our goal here is to show that the $\DC$-operators are compatible with these induction operations.  Here is the precise statement.

\begin{proposition} \label{p:DCind}
	Assume that $\bigcap_{\ell \in L} \ell K \ell^{-1}$ is trivial.  Then the following diagram commutes:
	\begin{align*}
	\xymatrix{
		\Irr(L:K) \ar[d]_-{ R_L^G } \ar[r]^-{ \DC_K} & \Irr \HHH(L,K) \ar[d]_-{ \Ind  } \\
		\Rep(G:H)   \ar[r]^-{ \DC_H } & \Rep \HHH(G, H) .
	}
	\end{align*}
\end{proposition}

\begin{proof}
	By the assumption, Lemma \ref{groupalgHecke} gives us
	\begin{align*}
	\C[L] \cong \End_{\HHH(L,K)} \Ind_K^L \1_K = \Ind_K^L \1_K \otimes_{\HHH(L,K)} \left( \Ind_K^L \1_K \right)^*.
	\end{align*}
	Therefore, given $V \in \Irr(L : K)$, we may rewrite this as
	\begin{align} \label{DCindproof1}
	V \cong \C[L] \otimes_{\C[L]} V = \Ind_K^L \1_K \otimes_{\HHH(L,K)} \left( \Ind_K^L \1_K \right)^* \otimes_{\C[L]} V,
	\end{align}
	where now we are taking the $L$-action on the first factor $\Ind_K^L \1_K$.  Using
	\begin{align*}
	\left( \Ind_K^L \1_K \right)^*\otimes_{\C[L]} V = \Hom_L \left( \Ind_K^L\1_K, V \right) = \Hom_K \left( \1_K, \Res_K^L V \right) = \DC_K (V),
	\end{align*}
	\eqref{DCindproof1} gives
	\begin{align*}
	V \cong \Ind_K^L \1_K \otimes_{\HHH(L,K)} \DC_K(V).
	\end{align*}
	Now, applying $R_L^G$ to both sides, one gets
	\begin{align} \label{DCindproof2}
	R_L^G V \cong R_L^G \left( \Ind_K^L \1_K \otimes_{\HHH(L,K)} \DC_K(V) \right) = R_L^G \left( \Ind_K^L \1_K \right) \otimes_{\HHH(L,K)} \DC_K(V),
	\end{align}
	as we had said that the $L$-action is on the first factor.  Now, using the natural isomorphisms of functors
	\begin{align*}
	\Infl_L^P \Ind_K^L  & = \Ind_H^P  \Infl_K^H & \Ind_P^{G} \Ind_H^P & = \Ind_H^{G} 
	\end{align*}
	and the fact that $\Infl_K^H\1_K = \1_H$, we can simplify
	\begin{align*}
	R_L^G \left( \Ind_K^L \1_K \right) = \Ind_P^G \Infl_L^P \Ind_K^L \1_K = \Ind_P^G \Ind_H^P  \Infl_K^H  \1_K = \Ind_H^G \1_H
	\end{align*}
	and so \eqref{DCindproof2} becomes
	\begin{align*}
	R_L^G V \cong \Ind_H^G \1_H \otimes_{\HHH(L,K)} \DC_K(V).
	\end{align*}
	
	Applying now the functor 
	\begin{align*}  
	\DC_H = \Hom_H \big(\1_H,\Res_H^{G}(-) \big) 
	\end{align*}
	to both sides, and again noting that the $G$-action in the right hand side is on the first factor, we get
	\begin{align*} 
	\DC_H( R_L^G V ) & \cong \Hom_H\big(\1_H,\Res_H^{G}\Ind_H^{G}\1_H\big) \otimes_{\HHH(L,K)} \DC_K(V) \\
	& = \End_G \left( \Ind_H^G \1_H, \Ind_H^G \1_H \right) \otimes_{\HHH(L,K)} \DC_K(V) \\
	& = \HHH(G, H) \otimes_{\HHH(L,K)} \DC_K(V) = \Ind_{\HHH(L,K)}^{\HHH(G,H)} \DC_K(V). & \qedhere
	\end{align*}
\end{proof}

\subsection{Character tables} \label{s:chartables}

We now review some facts about the character tables of some finite groups which will be used in our counting arguments later.  As a matter of notation, in this section and later, if $A$ is an abelian group, we often denote its group of characters by $\widehat{A} := \Hom(A, \C^\times)$.

\subsubsection{Character table of \texorpdfstring{$\GL_n(\F_q)$}{GLn(Fq)}}\label{s:charGL}

We follow the presentation of ~\cite[Chapter IV]{MacD}.  Fix a prime power $q$.  Let $\Gamma_n := \widehat{\F}_{q^n}^\times$ be the dual group of $\F_{q^n}^\times$.  For $n|m$, the norm maps $\textnormal{Nm}_{n,m} : \F_{q^m}^\times \to \F_{q^n}^\times$ yield an inverse system, and hence the $\Gamma_n$ form a direct system whose colimit we denote by
\[ \Gamma := \varinjlim \Gamma_n. \] 
There is a natural action of the Frobenius $\Frob_q : \overline{\F}_q^\times \to \overline{\F}_q^\times$, given by $\gamma \mapsto \gamma^q$, restricting to each $\F_{q^n}^\times$ and hence inducing an on action each $\Gamma_n$ and hence on $\Gamma$; we identify $\Gamma_n$ with $\Gamma^{\Frob_q^{n}}$.  Let $\Theta$ denote the set of $\Frob_q$-orbits in $\Gamma$.

The \emph{weighted size} of a partition $\lambda = (\lambda_1,\lambda_2,\ldots) \in \P$ is
\[
n(\lambda):=\sum_{i\geq 1}(i-1)\lambda_i=\sum_{j \geq 1} \binom{\lambda_{j}'}{2}
\]
where, as usual, $\lambda' = (\lambda'_1,\lambda'_2,\ldots)$ is the conjugate partition of $\lambda,$ i.e., $\lambda'_i$ is the number of $\lambda_j$'s not smaller than $i$.  The \emph{hook polynomial} of $\lambda$ is defined as%Macdonalds book?
\begin{align}\label{e:hookpolyl}
H_\lambda(q):=\prod_{\square\in\lambda}(q^{h(\square)}-1)
\end{align}	
where the product is taken over the boxes in the Ferrers' diagram (cf.~\cite[1.3]{Stan}) of $\lambda,$ and $h(\square)$ is the \emph{hook length} of the box $\square$ in position $(i,j)$ defined as $$h(\square):= \lambda_i + \lambda'_j - i - j+1.$$ 

By~\cite[IV (6.8)]{MacD} there is a bijection between the irreducible characters of $\GL_n(q)$ and the set of functions $\Lambda : \Theta \to \PP$ which are stable under the Frobenius action and having \emph{total size}
\[
|\Lambda|:=\sum_{\gamma \in \Theta} |\gamma| |\Lambda(\gamma)|
\]
equal to $n$.  Under this correspondence, the character $\chi_\Lambda^{\G}$ corresponding to $\Lambda$ has degree 
\begin{align} \label{e:degree}
\frac{\prod_{i=1}^n(q^i-1)}{ q^{-n(\Lambda')} H_{\Lambda}(q)}
\end{align}
where
\begin{align} \label{e:hookpolyL}
H_\Lambda(q):=\prod_{ \gamma \in \Theta} H_{\Lambda(\gamma)}(q^{|\gamma |})
\end{align}
and
\beq \label{nLambda} 
n(\Lambda):=\sum_{ \gamma \in \Theta} |\gamma|n(\Lambda(\gamma)).\eeq

\begin{remark}\label{r:unipart}
	There is a class of irreducible characters of $\GL_n(\F_q)$, known as the unipotent characters, which are also parametrized by $\P_n$.  Given $\lambda \in \P_n$, the associated unipotent character $\chi_\lambda^\G$ is the one corresponding to, in the description above, the function $\Lambda_\lambda : \Theta \to \P$, where $\Lambda_\lambda$ takes the (singleton) orbit of the trivial character in $\Gamma_1$ to $\lambda \in \P_n$ and all other orbits to the empty partition.
	
	In fact, any $\psi \in \Gamma_1 = \widehat{\F}_q^\times$ is a singleton orbit of  $\Frob_q$ and so we may view $\Gamma_1$ as a subset of $\Theta$.  Thus, if we let $\calQ_n$ denote the set of maps $\Lambda : \Gamma_1 \to \P$ of total size $n$, i.e.,
	\begin{align*}
	|\Lambda| = \sum_{\psi \in \Gamma_1} | \Lambda(\psi)| = n,
	\end{align*}
	then $\calQ_n$ is a subset of the maps $\Theta \to \P$ of size $n$.  The set of characters of $\GL_n(\F_q)$ corresponding to the maps in $\calQ_n$ will also be important for us later.
	
	The following description of the characters corresponding to $\calQ_n$ can be found in the work of Green \cite{Green55}.  Suppose $n_1$, $n_2$ are such that $n = n_1 + n_2$.  %For $i = 1$, $2$, let $\G_i := \GL_{n_i}(\F_q)$, and let $\B_i$, $\U_i$, and $\rN_i$ be the respective subgroups of $\G_i$ as described at the beginning of Section \ref{s:yokoyoko}.  
	Let $\Lg = \GL_{n_1}(\F_q) \times \GL_{n_2}(\F_q)$ and view it as the subgroup of $\GL_n(\F_q)$ of block diagonal matrices, $\U_{12} \leeq \G$ the subgroup of upper block unipotent matrices, and $\mathrm{P} := \Lg \ltimes \U_{12}$ the parabolic subgroup of block upper triangular matrices.  %If $\B_\Lg := \Lg \cap \B$, then its unipotent radical is $\U_\Lg = \Lg \cap \U = \U_1 \times \U_2$.  One may also identify $\rN_1 \times \rN_2$ as a (proper) subgroup of $\rN$.
	The $ \circ $-product $- \circ - : \Irr \GL_{n_1}(\F_q) \times \Irr \GL_{n_2}(\F_q) \to \Rep \GL_n(\F_q)$ is defined as
	\begin{align}\label{e:circprod}
	\chi_1 \circ \chi_2 =  R_\Lg^\G (\chi_1 \otimes \chi_2) = \Ind_{\mathrm{P}}^{\G}\Infl_{\Lg}^{\mathrm{P}}(\chi_1 \otimes \chi_2)  .
	\end{align} 
	Now, given $\Lambda \in \calQ_n$, we will often write $\psi_1$, $\ldots$, $\psi_r \in \Gamma_1$ for the distinct elements for which $n_i := |\Lambda(\psi_i)| > 0$ (note that $\sum_i n_i = n$) and $\lambda_i := \Lambda(\psi_i)$.  For each $1 \leq i \leq r$, we have the unipotent representation $\chi_{\lambda_i}^\G$ of $\GL_{n_i}(\F_q)$, described above, as well as the character $\psi_i^\G := \psi_i \circ \det : \GL_{n_i}(\F_q) \to \F_q^\times \to \C^\times$.  Then the irreducible character $\chi_\Lambda^\G$ of $\GL_n(\F_q)$ associated to $\Lambda \in \calQ_n$ is
	\begin{align} \label{e:chiLambdaG}
	\chi_\Lambda^\G := \left( \chi_{\Lambda(\psi_1)}^\G \otimes \psi_1^\G \right) \circ \cdots \circ \left( \chi_{\Lambda(\psi_r)}^\G \otimes \psi_r^\G \right).
	\end{align}
	The fact that it is irreducible is attributable to \cite{Green55}.  One may also think of the tuple of characters of the $\GL_{n_i}(\F_q)$ as yielding one on the product, which may be viewed as the Levi of some parabolic subgroup of $\GL_n(\F_q)$.  Then the above $\circ$-product is the parabolic induction of the character on the Levi.
\end{remark}

\subsubsection{Character table of \texorpdfstring{$\rN$}{N(T)}}\label{s:charnormtor}

Recall that we have isomorphisms $\rN \cong \T \rtimes \SS_n = (\F_q^\times)^{n} \rtimes \SS_n$.  Our aim is to describe $\Irr \rN$, but let us begin with a description of the irreducible representations of each of its factors.  One has $\Irr \T = \widehat{\T}$, the dual group.  Furthermore, it is well known that $\Irr \SS_n$ is in natural bijection with the set $\P_n$ of partitions of $n$:  to $\lambda \in \P_n$ one associates the (left) submodule of $\C[\SS_n]$ spanned by its ``Young symmetrizer'' \cite[\S~4.1]{fulton-harris}; we will denote the resulting character by $\chi_\lambda^\SS$.

To describe $\Irr \rN$ explicitly, we appeal to~\cite[\S~8.2, Proposition 25]{Serre}, which treats the general situation of a semidirect product with abelian normal factor.   If $\psi \in \widehat{\T}$ then $\psi$ extends to a ($1$-dimensional) character of $\T \rtimes \stab \psi$ (noting that if we identify $\widehat{\T} = ( \widehat{\F}_q^\times)^n$, then $\SS_n$ acts by permutations) trivial on $\stab \psi$; so now, given $\chi \in \Irr( \stab \psi)$, we get $\psi \otimes \chi \in \Irr( \T \rtimes \stab \psi)$.  The result cited above says that $\Ind_{\T \rtimes \stab \psi}^\rN \psi \otimes \chi$ is irreducible and in fact all irreducible representations of $\rN$ arise this way (with the proviso that we get isomorphic representations if we start with characters in the $\SS_n$-orbit of $\psi$).

If we write $\psi = (\psi_1, \ldots, \psi_n) \in ( \widehat{\F}_q^\times)^n$, then $\stab \psi \cong \SS_{n_1} \times \cdots \times \SS_{n_r}$, where the $n_i$ are the multiplicities with which the $\psi_j$ appear.  Further, $\chi \in \Irr(\stab \psi) = \Irr \SS_{n_1} \times \cdots \times \Irr \SS_{n_r}$ so is an exterior tensor product $\chi = \chi_{\lambda_1}^\SS \otimes \cdots \otimes \chi_{\lambda_r}^\SS$ with $\lambda_i \in \P_{n_i}$.  Thus, from $\Ind \psi \otimes \chi$, we may define a map $\Lambda : \Gamma_1 \to \P$ by setting $\Lambda(\psi_j) = \lambda_i$, where $j$ is among the indices permuted by $\SS_{n_i}$ and $\Lambda(\varphi)$ to be the empty partition if $\varphi \in \widehat{\F}_q^\times$ does not appear in $\psi$.  In this way, we get a map $\Lambda : \Gamma_1 \to \P$ of total size $n$, i.e., an element of $\calQ_n$ defined in Remark \ref{r:unipart}.

Conversely, given $\Lambda \in \calQ_n$, let $\psi_i \in \Gamma_1$, $n_i$ and $\lambda_i \in \P_{n_i}$ be as in the paragraph preceding \eqref{e:chiLambdaG}.  Let $\T_i := (\F_q^\times)^{n_i}$ and set $\rN_i := \T_i \rtimes \SS_{n_i}$.  Observe that if $\psi \in \widehat{\T}$ is defined by taking the $\psi_i$ with multiplicity $n_i$, then $\prod_{i=1}^r \rN_i = \T \rtimes \stab \psi$.  Now, $\psi_i$ defines a character of $\rN_i$ by
\begin{align*}
(t_1, \ldots, t_{n_i}, \sigma) \mapsto \prod_{j=1}^{n_i} \psi_i(t_j),
\end{align*}
and $\chi_{\lambda_i}^\SS$ defines an irreducible representation of $\SS_{n_i}$ and hence of $\rN_i$.  Hence we get
\begin{align*}
\chi_{\lambda_i,\psi_i}^{\rN_i} := \chi_{\lambda_i}^\SS \otimes \psi_i \in \Irr \rN_i.
\end{align*}

Taking their exterior tensor product and then inducing to $\rN $ gives the irreducible representation
\begin{align} \label{e:chiLambdaNdef}
\chi _{\Lambda}^{\rN} := \Ind ^{\rN}_{\prod \rN_i} \bigotimes_{i} \chi_{\lambda_i,\psi_i}^{\rN_i} \in \Irr \rN .
\end{align}
It is in this way that we will realize the bijection $\calQ_n \xrightarrow{\sim} \Irr \rN$.

It will be convenient to define for $\Lambda \in \calQ_n $ the function $\wt{\Lambda} \in \calQ_n$ as
\begin{align}\label{e:Lamtildef}
\wt{\Lambda}\left(  \psi \right) := 
\begin{cases}
{\Lambda}\left(  \psi \right)'  & \mbox{ for }\psi\mbox{ odd,}\\
{\Lambda}\left(  \psi \right)  & \mbox{ for }\psi\mbox{ even}  
\end{cases}
\end{align}
where $\psi $ is said to be \emph{even} if $\psi(-1_{\F_q}) = 1_\C$ and \emph{odd} otherwise.

%We prove in Theorem~\ref{p:expbijection} the existence of a $\mathcal{X}_\Lambda^Y \in \Irr \YY_{d,n}(u)$ whose $u=q $ specialization defines a $\chi_\Lambda^\HHH \in \HHH(\G, \U)$, and its $ u=1 $ specialization is $\chi^{\rN}_{\wt{\Lambda}} \in \Irr \rN$.

% For $\Lambda \in \calQ_n$, we will write $\chi_\Lambda^{\rN}$ for the corresponding irreducible 
% character of $\rN$, $\mathcal{X}_\Lambda^Y$ for that of $\YY_{d,n}(u)$ and $\calX_\Lambda^\HHH$ 
% for that of $\HHH(\G, \U)$.  
% 
%By Proposition \ref{Heckerepgen}, each $\calX_\Lambda^\HHH$ arises 
%from a unique irreducible character of $\G$, which we will denote by $\chi_\Lambda^\G$.

%\subsection{Parameter sets for irreducible characters}
%
%\subsubsection{The Iwahori--Hecke algebra in type \texorpdfstring{$A_{n-1}$}{Lg}}\label{s:IwHalg}

\subsection{Parameter sets for \texorpdfstring{$\Irr \HH_n$}{Irr Hn} }\label{s:IwHalg}

Here, we take up again the notation introduced at the beginning of Section \ref{s:IHAlg}.  Given a partition $\lambda \in \P_n$, we will be able to associate to it three different characters:  the unipotent character $\chi_\lambda^\G$ of $\G$ described in Remark \ref{r:unipart}, which we will see below is, in fact, an element of $\Irr(\G: \B)$; a character $\chi_\lambda^\HHH \in \Irr \HHH(\G, \B)$; and the irreducible character $\chi_\lambda^\SS$ of the symmetric group $\SS_n$.  The discussions in Sections \ref{s:IHAlg} and \ref{s:RepsHA} suggest that there are relationships amongst these and the purpose of this section is indeed to clarify this.

To a partition $\nu \in \P_n$, say $\nu = (\nu_1, \ldots, \nu_\ell)$, one can associate a subgroup $\SS_\nu := \SS_{\nu_1} \times \cdots \times \SS_{\nu_\ell} \leq \SS_n$ and then consider the character $\tau_\nu^\SS$ of the induced representation $\Ind_{\SS_\nu}^{\SS_n} \1_{\SS_\nu}$.  Then these characters are related to those of the irreducible representations $\chi_\lambda^\SS$ (see Section \ref{s:charGL}) by the Kostka numbers $K_{\lambda \nu}$ \cite[Corollary 4.39]{fulton-harris}, via the relationship
\begin{align} \label{Kostkasym}
\tau_\nu^\SS = \sum_{\lambda \in \P_n} K_{\lambda \nu} \chi_\lambda^\SS.
\end{align}
%On the other hand, given $\lambda \in \P_n$, there is a unipotent (irreducible) character $\chi_\lambda^G$ of $GL_n(\F_q)$.  In fact, these are the precisely the irreducible characters appearing in $\Ind_B^G \1_B$ \eqref{Inddecomp} and hence are those giving rise to the irreducible characters of the Iwahori--Hecke algebra; let $\mathfrak{X}_\lambda^\HHH$ be the 

Also to $\nu \in \P_n$ one can associate a standard parabolic subgroup $P_\nu \leq \G$ whose Levi factor $L_\nu$ is isomorphic to $\GL_{\nu_1}(\F_1) \times \cdots \times \GL_{\nu_\ell}(\F_q)$.  Then one may consider the character $\tau_\nu^\G := \Ind_{P_\nu}^\G \1_{P_\nu}$ of the parabolic induction of the trivial representation.\footnote{By \cite[Proposition 6.1]{DigneMichel}, $\tau_\nu^\G$ depends only on the isomorphism class of $L_\nu$ (rather than the parabolic $P_\nu$) which in turn depends only on $\nu$.}  
%Then, given $\lambda \in \P_n$, one defines a character $\chi_\lambda^\G$ of $\G$ by
%\begin{align*}
%\chi_\lambda^\G := \sum_{w \in \SS_n} \sgn(w) \tau_{\lambda + \delta - w(\delta)}^\G,
%\end{align*}
%where $\delta := (n-1, n-2, \ldots, 1,0)$, the ``sum'' of partitions is by regarding them as $n$-tuples, possibly with several entries $0$ and not necessarily decreasing, and we take the convention that $\tau_\nu = 0$ if $\nu$ has a negative entry; these are, in fact, irreducible \cite[Theorem 2.1]{Steinberg}.  
Then if $\lambda \in \P_n$ and $\chi_\lambda^\G$ denotes the corresponding unipotent characters (as described in Remark \ref{r:unipart}), the following remarkable parallel with representations of the symmetric group was already observed at \cite[Corollary 1]{Steinberg}:
\begin{align} \label{KostkaG}
\tau_\nu^\G = \sum_{\lambda \in \P_n} K_{\lambda \nu} \chi_\lambda^\G.
\end{align}
In particular, as $K_{\lambda \lambda} = 1$ for all $\lambda \in \P_n$, we see that $\chi_\lambda^\G$ is a component of $\tau_\lambda^\G$ and hence of $\tau_{(1^n)}^\G$, which is the character of $\Ind_\B^\G \1_\B$.  This shows that
\begin{align} \label{unisubset}
\left\{ \chi_\lambda^\G \, : \, \lambda \in \P_n \right\} \subseteq \Irr(\G: \B).
\end{align}
%Proposition \ref{Heckerepgen} then shows that one obtains a non-zero irreducible character $\chi_\lambda^\HHH := \DC_\B( \chi_\lambda^\G)$ of $\HHH(\G, \B)$ from $\chi_\lambda^\G$.

Let us now consider the specializations \eqref{Iwahorispec} of $\HH_n$ corresponding to $\theta_q, \theta_1 : \C[u^{\pm 1}] \to \C$.  Since $\HH_n(u)$ is split semisimple \cite[(68.12) Corollary]{CurtisReiner}, Tits's deformation theorem (\cite[(68.20) Corollary]{CurtisReiner}, \cite[7.4.6 Theorem]{GeckPfeiffer}) applies to give bijections 
\begin{align*}
d_{\theta_q} : \Irr \HH_n (u) & \xrightarrow{\sim} \Irr \HH_n(q) = \Irr \HHH(\G,\B) & d_{\theta_1} : \Irr \HH_n (u) & \xrightarrow{\sim} \Irr \HH_n(1) = \Irr \SS_n,
\end{align*}
where a character $\mathcal{X} : \HH_n \to \C[u^{\pm 1}]$ (the characters of $\HH_n(u)$ are in fact defined over $\C[u^{\pm 1}]$ by \cite[Proposition 7.3.8]{GeckPfeiffer}) is taken to its specialization $\mathcal{X}_z : \HH_n \to \C \otimes_{\theta_z} \C[u^{\pm 1}] = \C$, for $z= 1$ or $z = q$.  We can thus define the composition
\begin{align} \label{TDBdef}
\TD_\B := d_{\theta_q} \circ d_{\theta_1}^{-1} : \Irr \SS_n \xrightarrow{\sim} \Irr \HHH(\G, \B).
\end{align}
Now, we have bijections (using Proposition \ref{Heckerepgen} for the one on the left)
\begin{align} \label{DBTB}
\vcenter{ \xymatrix{ \Irr(\G : \B) \ar[dr]^{\DC_\B} & & \Irr \SS_n \ar[dl]_{\TD_\B} \\
		& \Irr \HHH(\G, \B) & } }
\end{align}
and since $| \Irr \SS_n | = | \P_n|$ all of the sets are of this size, so it follows that the inclusion in \eqref{unisubset} is in fact an equality
\begin{align*}
\Irr(\G: \B) = \left\{ \chi_\lambda^\G \, : \, \lambda \in \P_n \right\} .
\end{align*}
Furthermore, the following holds.

\begin{proposition}\cite[Theorem 4.9(b)]{bitraces}\label{p:bitraces}
	For $\lambda \in \P_n$, one has
	\begin{align*}
	\DC_\B( \chi_\lambda^\G) = \TD_\B(\chi_\lambda^\SS).
	\end{align*}
\end{proposition}

\begin{proof}
	\cite[(68.24) Theorem]{CurtisReiner} states that the bijection $\DC_\B^{-1} \circ \TD_\B : \Irr \SS_n \xrightarrow{\sim} \Irr(\G : \B)$ satisfies
	\begin{align*}
	\left( \left( \DC_\B^{-1} \circ \TD_\B \right) ( \chi_\lambda^\SS), \tau_\nu^\G \right) = \left(\chi_\lambda^\SS, \tau_\nu^\SS \right)
	\end{align*}
	for all $\nu \in \P_n$.  But the right hand side is, from \eqref{Kostkasym}, $K_{\lambda \nu}$, but then from \eqref{KostkaG}, we must have $\DC_\B^{-1} \circ \TD_\B( \chi_\lambda^\SS) = \chi_\lambda^\G$.
\end{proof}

This allows us to unambiguously define, for each $\lambda \in \P_n$, an irreducible representation $\chi_\lambda^\HHH \in \Irr \HHH(\G, \B)$ by
\begin{align*}
\chi_\lambda^\HHH := \DC_\B( \chi_\lambda^\G) = \TD_\B(\chi_\lambda^\SS).
\end{align*}

\subsection{Parameter sets for \texorpdfstring{$\Irr \YY_{d,n}$}{Irr Yd,n} }\label{s:parY}% The Yokonuma--Hecke algebra}

In Section \ref{s:chartables}, we saw that the set $\calQ_n$ was used to parametrize both the irreducible representations of $\Irr \rN$ (Section \ref{s:charnormtor}) as well as a subset of those of $\G = \GL_n(\F_q)$ (Remark \ref{r:unipart}).  We will see (in Remark \ref{r:IrrGUtwists}) that this latter subset is in fact $\Irr(\G : \U)$, which by Proposition \ref{Heckerepgen} yields the irreducible representations of $\HHH(\G, \U)$, and furthermore, that Tits's deformation theorem again applies to the generic Yokonuma--Hecke algebra, which gives a bijection of these with $\Irr \rN$.  The purpose of this section is to establish, as in Section \ref{s:IwHalg}, the precise correspondence between the relevant irreducible representations in terms of the elements of the parameter set $\calQ_n$.

Let us proceed with the argument involving Tits's deformation theorem.  Recall from Theorem~\ref{t:yokonuma} that we have isomorphisms
\begin{align*}
\HHH(\G, \U) & \cong \YY_{q-1,n}(q) & & \textnormal{and} & \C[\rN] & \cong \YY_{q-1,n}(1)
\end{align*}
by specialising $\YY_{q-1,n}$ at $u=q$ and $u=1$, respectively.  Thus, both $\HHH(\G, \U)\cong \YY_{q-1,n}(q)$ and $\C[\rN]\cong \YY_{q-1,n}(1)$ are split semisimple by \cite[Proposition 9]{cp} and $\YY_{d,n}(u)$ is also split by \cite[5.2]{cp}.\footnote{Strictly speaking \cite{cp} considers $\YY_{d,n}(v)$ where $v^2=u$.  However, one can define all irreducible representations in \cite[Proposition 5]{cp} of  $\YY_{d,n}(v)$ already over $\YY_{d,n}(u)$ by a slight change in the defining formulas of \cite[Proposition 5]{cp} see \cite[Theorem 3.7]{hm}.} %We thank Maria Chlouveraki and Lo{\" i}c Poulain d'Andecy explaining this to us.}  
Thus, the deformation theorem (\cite[(68.20) Corollary]{CurtisReiner}, \cite[7.4.6 Theorem]{GeckPfeiffer}) again applies to yield bijections \small
\begin{align} \label{titsiso}
d_{\theta_q} : \Irr \YY_{q-1,n}(u) & \xrightarrow{\sim} \Irr \YY_{q-1,n}(q) = \Irr \HHH(\G,\U) & d_{\theta_1} : \Irr \YY_{q-1,n}(u) & \xrightarrow{\sim} \Irr \YY_{q-1,n}(1) = \Irr \rN,
\end{align} \normalsize
where we denote by $\theta_q:\C[u^{\pm 1}]\to \C$ the $\C$-algebra homomorphism sending $u$ to $q$.  Again, \cite[Proposition 7.3.8]{GeckPfeiffer}  applies to say that if $\calX \in \Irr \YY_{d,n}(u)$, then in fact $\calX : \YY_{d,n}\to \C[u^{\pm 1}]$, and the bijections in \eqref{titsiso} are in fact the specializations of $\calX$.  We may put these together to obtain a bijection
\begin{align} \label{TDUdef}
\TD_\U := d_{\theta_q} \circ d_{\theta_1}^{-1} : \Irr \rN \to \Irr \HHH(\G, \U).
\end{align}

Furthermore, \cite[Remark 7.4.4]{GeckPfeiffer} tells us that
\begin{align} \label{speccharval}
d_{\theta_q}(\calX) & = \theta_q(\calX) & & \textnormal{ and } & d_{\theta_1}(\calX) & =\theta_1(\calX).
\end{align} 
In particular, the dimensions of the irreducible representations of $\HHH(\G, \U)$ and $\C[\rN]$ agree.  Thus, we may conclude from Wedderburn's theorem that the Hecke algebra $\HHH(\G, \U)$ and the group algebra $\C[\rN]$ are isomorphic as abstract $\C$-algebras.
%\end{remark}

% We write $\calX_1 \in \Irr \rN$ and $\calX_q \in \Irr \HHH(\G, \U) $ for the characters obtained from $\calX \in \Irr \YY_{q-1,n}(u)$ by the specializations $u\mapsto 1$ and $u \mapsto q$, respectively.

%We need an analogue of Proposition \ref{p:bitraces} for Yokonuma--Hecke algebras.  %Let us define
%\begin{align*}
%\Irr\big( \G:\U  \big) := \{ \zeta \in \Irr (\G) \, : \, (\zeta, \1_\U^\G) > 0 \} 
%\end{align*}

Defining $\DC_\U : \Irr ( \G : \U ) \to \Irr \HHH(\G,\U)$ as in Proposition \ref{Heckerepgen}, we get a diagram like that at \eqref{DBTB}:
\begin{align} \label{DUTU}
\vcenter{ \xymatrix{ \Irr(\G : \U) \ar[dr]^{\DC_\U} & & \Irr \rN \ar[dl]_{\TD_\U} \\
		& \Irr \HHH(\G, \U). & } }
\end{align}

%The purpose of this section is to clarify the relationship between the bijections $\TD_\U$ given in \eqref{TDUdef} and 
%\begin{align}\label{e:TdD}
% \TD_{\U}: \Irr \left( (\F_q^\times)^{n}  \rtimes \SS_n \right) \xrightarrow{\sim} \Irr \HHH(\G,\U)   
%\end{align} 
%\begin{align}\label{e:DcC}
%\DC_{\U} : \Irr \big( \G:\U  \big)\xrightarrow{\sim} \Irr \HHH(\G,\U)
%\end{align} 
%given by Proposition \ref{Heckerepgen}.

%As in Remark \ref{r:TdD}, 
For every $\Lambda \in \calQ_n$ we will define a character $\mathcal{X}^{\YY}_\Lambda \in \Irr \YY_{d,n}(u)$ (see \eqref{e:defdefchar} in Section \ref{yokotable} below) such that %the character
\begin{align*}
\mathcal{X}^{\YY}_\Lambda \otimes_{\theta_1} \C = d_{\theta_1} \left( \mathcal{X}^{\YY}_\Lambda \right) = \chi_{\wt{\Lambda}}^\rN \in \Irr \rN
\end{align*}
is the one described in Section \ref{s:charnormtor} for the modified $\wt{\Lambda} $.  This, together with the $ u=q $ specialization
\begin{align*}
\mathcal{X}^{\YY}_\Lambda \otimes_{\theta_q} \C = d_{\theta_q} (\mathcal{X}^{\YY}_\Lambda )= \chi^{ \HHH }_\Lambda  \in \Irr \HHH(\G,\U )
\end{align*}
will satisfy
\begin{align*}
\TD_{\U} \left( \chi^{ \rN }_{\wt{\Lambda}} \right) = \chi^{ \HHH }_\Lambda.
\end{align*}

As in Section \ref{s:RepsHA}, for a character $\chi^\G_\Lambda \in \Irr(\G : \U)$ one defines
\begin{align*}
\DC_{\U} \left(	\chi^{\G}_\Lambda \right) := \Hom_{\U} \left( \1_{\U}, \Res^{\G}_{\U} \chi^{\G}_{\Lambda} \right),
\end{align*}
namely, the subspace of $\U $-invariants, as in Proposition \ref{Heckerepgen}.

Our definition of $\mathcal{X}^{\YY}_\Lambda $ will be in such a way that the $\chi_\Lambda^{\G} $ from Section \ref{s:charGL} is mapped by $\DC_{\U} $ to the same $\chi^{ \HHH }_\Lambda $, proving thus the main result of this section. %\ref{s:parY}:

\begin{theorem}\label{p:expbijection}
	Let the characters $\chi_\Lambda^\G$ and $\chi_{\wt{\Lambda}}^\rN$ be those described in Sections~\ref{s:charGL} and~\ref{s:charnormtor}, respectively.  Then, the set $\calQ_n $ parametrizes the pairs of irreducible representations $ (V,V^{\U}) $ from Proposition~\ref{Heckerepgen} in such a way that the characters
	\begin{align*}
	\chi_{\Lambda}^{\G} & \in \Irr(\G:\U) &
	& \text{ and } & \chi_{\wt{\Lambda}}^{\rN} & \in \Irr \rN
	\end{align*}
	satisfy
	\begin{align*}
	\DC_{\U} \left(\chi_{\Lambda}^{\G} \right)= \TD_{\U} \left( \chi_{\wt{\Lambda}}^{\rN} \right) \in \Irr \HHH \left(\G,\U \right).
	\end{align*}
	%where $\DC_\U $ and $\TD_\U$ are the maps from \eqref{e:DcC} and \eqref{TDUdef}, respectively.
	
	%\begin{align*}
	%\chi_{\Lambda}^{\G} \in \Irr(\G:\U) := 
	%\{ \zeta \in \Irr (\G) \, : \, (\zeta, \1_\U^\G) > 0 \} 
	%\end{align*}
	%is sent by the $\DC_{\lala} $ map from \eqref{e:DcC} to
	%\begin{align*}
	%\chi_{\Lambda}^{\HHH} \in \Irr(\HHH(\G,\U)),
	%\end{align*}
	%which is also the image of
	%\begin{align*}
	% \chi_{\Lambda}^{N} \in \Irr( N )
	%\end{align*}
	%by the $\TD_{\lala} $ map from \eqref{e:TdD}.
	%%by the bijections in~\eqref{titsiso}.
\end{theorem}

Inspired by the construction of $\Irr \rN$ in Section \ref{s:charnormtor}, we establish a parallel with the technique of parabolic induction %(Harish-Chandra)?
to build the character table of $\G $ using the unipotent characters as building blocks.

The rest of this section is devoted to studying more carefully the bijections $\DC_{\U},\TD_{\U}, \DC_{\B}$ and $\TD_{\B}$.  In Section~\ref{p:vanishBtoU} we prove the compatibility between them.  In Section~\ref{p:sutwist} we analyze their behaviour with a twist by a degree one character.  In Section~\ref{p:paraind} we check their interplay with parabolic induction and exterior tensor products.  Finally, in Section~\ref{yokotable} we construct the character table of the generic Yokonuma--Hecke algebra, as a common lift of both $\Irr \rN$ and $\Irr(\G : \U)$.

\subsubsection{\texorpdfstring{From $\B$ to $\U$}{From B to U} } \label{p:vanishBtoU}

Let us start by checking the correspondence for the unipotent characters $\chi_{\lambda}^{\G}$ in the $\HHH(\G,\U) $-case agrees with the one in the $\HHH(\G,\B)$-case (cf. Remark \ref{r:unipart}).  Taking $G = \G$, $H = \B$, $K = \U$ and $L = \T$ in Proposition \ref{p:Hprojection}, noting that Lemma \ref{suffcent} applies as the set $\rN$ of double $\U$-coset representatives normalizes $\T$ , we are provided with a surjective homomorphism $\HHH(\G, \U) \twoheadrightarrow \HHH(\G, \B)$, and hence we obtain an inflation map $\Infl : \Rep \HHH(\G, \B) \to \Rep \HHH(\G, \U)$, taking $\Irr \HHH(\G, \B)$ to $\Irr \HHH(\G, \U)$.  We can then make the following precise statement.

\begin{proposition} \label{p:BtoU}
	%The natural inclusions
	%$$  \Irr(\G:\B) \subseteq \Irr(\G:\U) , $$
	%$$  \Irr(\HHH(\G,\B)) \subseteq \Irr(\HHH(\G,\U))  $$ and
	%$$  \Irr(\SS_n) \subseteq \Irr(N) , $$ 
	%commute with the
	%$\DC_{\lala} $ and $\TD_{\lala} $ maps from \eqref{e:DcC},\eqref{e:IHDcC} and
	%\eqref{e:TdD},\eqref{e:IHTdD} respectively.
	The following diagram is commutative
	\begin{align}\label{e:diagrm2sqr}
	\vcenter{ \xymatrix{
			\Irr(\G:\B) \ar@{_{(}->}[d]^-{  } \ar[r]^-{ \DC_{\B} } & 
			\Irr \HHH(\G,\B)\ar@{_{(}->}[d]^-{ \Infl } & 
			\Irr \SS_n \ar[l]_-{\TD_{\B}} \ar@{_{(}->}[d]^-{ \Infl }\\
			\Irr(\G:\U)   \ar[r]^-{ \DC_{\U} } &    \Irr \HHH(\G,\U) & \Irr \rN   \ar[l]_-{\TD_{\U}} 
		} }
		\end{align}
		where the top horizontal arrows are the bijections in \eqref{DBTB}, the bottom horizontal arrows those in \eqref{DUTU}, the leftmost vertical arrow is the natural inclusion and the other two are the inflation maps.
	\end{proposition}
	
	\proof
	We know from Proposition \ref{Heckerepgen} that $\DC_{\B} $ maps $\chi_{\lambda}^{\G} $ to the character $\chi_{\lambda}^{\HHH} \in \Irr \HHH(\G,\B)$ given by
	\begin{align*}
	\chi_{\lambda}^{\HHH} %(G,B)} 
	=\Hom_{\B}( \1_\B  ,  \Res^{\G}_{\B}\chi_{\lambda}^{\G}) = \Hom_{\G}(  \Ind^{\G}_{\B}\1_\B,  \chi_{\lambda}^{\G} ).
	\end{align*}
	We want to replace $  \Ind^{\G}_{\B}\1_\B $ by  $\Ind^{\G}_{\U} \1_\U$.  Let us take a closer look at the latter.  It is canonically isomorphic to $\Ind^{\G}_{\B} \Ind^{\B}_{\U} \1_\U$, but since $B = \U \rtimes \T$, $\Ind_\U^\B \1_\U = \Infl_\T^\B \C[\T]$, where $\C[\T] $ is the regular representation of $\T$.
	%\begin{align*}
	%\Ind^{\G}_{\B} \Ind^{\B}_{\U} \Infl_{1}^{\U} \1 
	%\end{align*}
	%since $\1_\U = \Infl_{1}^{\U} \1   .$
	%Taking into account that
	%\begin{align*}
	%\Ind^{\B}_{\U} \Infl_{1}^{\U}  =  \Infl_{\T}^{\B} \Ind^{\T}_{1}  
	%\end{align*}
	%since $\U \unlhd \B,$ then  $\Ind^{\G}_{\U} \1_\U $   %$\DC_{\U}(\chi_{\lambda}^{\G}) $
	%becomes
	%\begin{align*}
	%\Ind^{\G}_{\B} \Infl_{\T}^{\B} \Ind^{\T}_{1}  \1  =
	%\Ind^{\G}_{\B} \Infl^{\B}_{\T} \C[\T]   
	%\end{align*}
	%where $\C[\T] $ is the regular representation of $\T. $
	Since $\C[\T] $ is the direct sum of the degree one characters $\phi:\T \to \C^{\times}$, the group of which we denote by $\dual{\T}$, one gets $\Infl^{\B}_{\T} \C[\T]$ is the sum of the inflations $\phi: \B\to \C^{\times}$,
	hence
	\begin{align*}
	\Ind^{\G}_{\U} \1_\U = \bigoplus_{\phi \in \dual{\T}} \Ind_{\B}^{\G} \phi,
	\end{align*}
	and
	\begin{align}\label{e:decHomInflReg}
	\Hom_{\G} \left(\Ind^{\G}_{\U} \1_\U, \chi_{\lambda}^{\G} \right) = \bigoplus_{\phi \in \dual{\T}} \Hom_{\G} \left( \Ind_{\B}^{\G} \phi,  \chi_{\lambda}^{\G} \right).
	%\Hom_{\G}(  \chi_{\lambda}^{\G} ,\Ind^{\G}_{\U}(\1_\U) ) =
	%\bigoplus_{\phi \in \dual{\T}} \Hom_{\G}(  \chi_{\lambda}^{\G} , \Ind_{\B}^{\G}( \phi)).
	\end{align}
	
	For any $\phi \in \dual{\T} $ one has
	\begin{align*}
	\Hom_{\G} \left(  \Ind_{\B}^{\G} \phi,  \Ind_{\B}^{\G} \1_{\B}  \right) =
	\Hom_{\B} \left(  \phi, \Res_{\B}^{\G}  \Ind_{\B}^{\G} \1_{\B} \right)   
	%\Hom_{\G}(  \Ind_{\B}^{\G}(\1_{\B})  , \Ind_{\B}^{\G} \phi) =
	%\Hom_{\B}( \Res_{\B}^{\G}  \Ind_{\B}^{\G}(\1_{\B})  ,  \phi ) 
	\end{align*}
	which by the Mackey decomposition becomes
	\begin{align*} %\label{e:HomMackey}
	\sum_{\sigma \in \B \setminus \G /\B} 
	\Hom_{\B} \left(\phi ,   \Ind_{\B\cap \sigma ^{-1}\B\sigma}^{\B}( \1_{\B\cap \sigma ^{-1}\B \sigma } )  \right) =
	\sum_{\sigma \in \B \setminus \G /\B} 
	\Hom_{\B\cap \sigma ^{-1}\B \sigma} \left(    \Res_{\B\cap \sigma ^{-1}\B \sigma }^{\B} \phi , \1_{\B\cap \sigma ^{-1}\B \sigma } \right)
	%\sum_{\sigma \in \B \setminus \G /\B} 
	%\Hom_{\B}(   \Ind_{\B\cap \sigma ^{-1}\B\sigma}^{\B}( \1_{\B\cap \sigma ^{-1}\B \sigma } )  ,  \phi ) =
	%\sum_{\sigma \in \B \setminus \G /\B} 
	%\Hom_{\B\cap \sigma ^{-1}\B \sigma}(   \1_{\B\cap \sigma ^{-1}\B \sigma }  ,  \Res_{\B\cap \sigma ^{-1}\B \sigma }^{\B}( \phi ) )
	\end{align*}
	where $\sigma $ runs over a full set of $\B $-double coset representatives.  Since $\T \subseteq  \B\cap \sigma ^{-1}\B \sigma$ acts non-trivially on $\Res_{\B\cap \sigma ^{-1}\B \sigma }^{\B}( \phi )$ for $\phi$ nontrivial, the only non-vanishing term in this last sum is the one with $\phi  = \1_{\B}$.
	
	Since $\chi_{\lambda}^{\G} $ is a constituent of $\Ind_{\B}^{\G} \1_{\B}$, only one summand in the right hand side of \eqref{e:decHomInflReg} does not vanish and we end up with
	\begin{align*}
	\Hom_{\G}(   \Ind_{\U}^{\G} \1_{\U} , \chi_{\lambda}^{\G} ) = \Hom_{\G}(   \Ind_{\B}^{\G} \1_{\B} , \chi_{\lambda}^{\G} ).
	\end{align*}
	
	The irreducible $\HHH(\G,\U) $-representation associated to $\chi_{\lambda}^{\G} $ is
	\begin{align*}
	\Hom_{\U}\left( \1_\U  , \Res^{\G}_{\U} \chi_{\lambda}^{\G} \right) =  \Hom_{\G}(  \Ind^{\G}_{\U} \1_\U , \chi_{\lambda}^{\G}  ) 
	%% \chi_{\lambda}^{\HHH} %(\G,\B)}
	%\Hom_{\G}( \chi_{\lambda}^{\G} , \Ind^{\G}_{\U}(\1_\U)  ) 
	%=\Hom_{\B}( \Res^{\G}_{\B}(\chi_{\lambda}^{\G}) , \1_\B  ) .
	%
	\end{align*}
	which is thus isomorphic to $\Hom_{\G}( \Ind^{\G}_{\B}\1_\B,  \chi_{\lambda}^{\G} )$ with the $\HHH(\G,\U) $-module structure induced by the surjection $\HHH(\G,\U) \to \HHH(\G,\B)$ coming from Proposition \ref{p:Hprojection}.  This proves the commutativity of the left square in \eqref{e:diagrm2sqr}.  
	
	The right one is also commutative since all the non-horizontal arrows in
	$$\xymatrix{
		& \Irr \HH_n(u) \ar@{_{(}->}'[d][dd]^-{ \Infl } \ar[dl]_-{ d_{\theta_q} } \ar[dr]^-{ d_{\theta_1} } & \\
		\Irr \HHH(\G,\B ) \ar@{_{(}->}[dd]^-{ \Infl } &  & \Irr \SS_n \ar[ll] \ar@{}[l]_-{ \TD_{\B} } \ar@{_{(}->}[dd]^-{ \Infl } \\
		& \Irr \YY_n(u) \ar[dl]_-{ d_{\theta_q} } \ar[dr]^-{ d_{\theta_1} } & \\
		\Irr \HHH(\G,\U) &  & \Irr \rN \ar[ll]_-{ \TD_{\U} }
	}$$
	arise from taking tensor products.
	\endproof
	
	\subsubsection{Twisting by characters of \texorpdfstring{$\F_q^\times$}{Lg}} \label{p:sutwist}
	
	In order to deal with character twists, we extend the Iwahori--Hecke algebra as follows.  For $ d,n \geq 1 $ we introduce the $\C[u^{\pm 1}]$-algebra $\HH_{d,n} $ 
	%%generated by the elements
	%%\begin{align*}
	%%\mathsf{h},\mathsf{T}_i ,\: i=1,\ldots,n-1 ;
	%%\end{align*}
	%%subject to the following relations:
	%%\begin{enumerate}[(a)]
	%% \item $\mathsf{T}_i \mathsf{T}_j = \mathsf{T}_j \mathsf{T}_i$ for all $i,j=1,\ldots, n-1$ such that $|i-j|>1$;
	%% \item $\mathsf{T}_i \mathsf{T}_{i+1} \mathsf{T}_i = \mathsf{T}_{i+1} \mathsf{T}_i \mathsf{T}_{i+1}$ for all $i=1,\ldots, n-2$;
	%% \item $\mathsf{h} \mathsf{T}_i=\mathsf{T}_i \mathsf{h}$ for all $i=1,\ldots, n-1$;
	%% \item $\mathsf{h}^{d}=1$;
	%% \item 
	%% $\mathsf{T}_i^{2}=u+(u-1)\mathsf{T}_i$
	%%for $i=1,\ldots, n-1$.
	%\end{enumerate}
	defined by 
	\begin{align*}
	\HH_{d,n} := \HH_n[h]/(h^d-1) = \HH_n \otimes_{\C} \C[C_d],
	\end{align*}
	where $ C_d $ is the cyclic group of order $ d $ and $ h\in C_d $ a generator.
	As usual, denote by $\HH_{d,n}(u) $ the corresponding $\C(u)$-algebra $\C(u)\otimes_{\C[u^{\pm 1}]} \HH_{d,n}$.
	
	Since $\HH_{d,n}(u) $ is a tensor product of semisimple algebras it is also semisimple and its set of characters is the cartesian product of those of its factors. %the latter (parametrized by $\P_n $) and the group dual to the cyclic group of order $d$.  
	Concretely, for $\lambda \in \P_n$ and $\psi \in \Irr \C[C_d]$ we define $\mathcal{X}^{\HH_{d, n}}_{\lambda,\psi}$ to be the exterior tensor product:\begin{align}\label{e:cartprodtable}
	\mathcal{X}^{\HH_{d, n}}_{\lambda,\psi} :=\mathcal{X}^{\HH_{ n}}_{\lambda} \otimes \psi \in \Irr \HH_{d,n}(u) = \Irr \HH_{n}(u) \times \Irr \C[C_d] 
	\end{align}
	These are all the irreducible representations of $\HH_{d,n}(u)$, and they all come from localizing certain finitely generated representations of $\HH_{d,n}$ that we also denote by $\mathcal{X}^{\HH_{d, n}}_{\lambda,\psi}$.
	
	%Let us define the following central idempotent of $  Y_{d,n}(u) $ as
	%\begin{align*}
	%e := \prod_{i=1}^{n-1}e_i = \frac{1}{d^{n-1}}\sum_{ (\alpha_i)_{i=1}^n } \prod_{i=1}^{n}h_i^{\alpha_i} 
	%\end{align*}
	%where the sum in the last formula ranges over all $ n$-tuples $ (\alpha_i)_{i=1}^{n} \in (\Z/d\Z)^{n}$
	%adding up to $ 0, $ and the $ e_i $ are those (not necessarily central) idempotents of~\eqref{e:esubi}.

	There is a natural quotient
	\begin{align}\label{e:natquot}
	\YY_{d,n} \twoheadrightarrow \HH_{d,n}
	\end{align}
	that sends the $\mathsf{h_i} \in \YY_{d,n}$ to $\mathsf{h} \in \HH_{d,n} $ and the $\mathsf{T_i} \in \YY_{d,n}$ to the $\mathsf{T_i} $ from $\HH_{n} .$ 
	
	%mapping every $ f \in \HH_{n}(u) $ to $ efe $ and $ h \in \HH_{n}(u)[h]  $ is mapped to $e h_1 $ which is easily seen to agree with any other $ e h_i  .$
	
	The $ u=1 $ specialization gives 
	\begin{align*}
	\C \otimes_{\theta_1} \HH_{d,n}  \simeq \C[\SS_n] \otimes_{\C}\C[C_d] = \C[\SS_n\times C_d]
	\end{align*} 
	and the  $ u= q $ specialization gives 
	\begin{align*}
	\C \otimes_{\theta_q} \HH_{d,n}  \simeq  \HHH(\G,\B) \otimes_{\C}\C[C_d] 
	\end{align*}
	which in the $ d=q-1 $ case gives 
	\begin{align*}%\label{e:isomb1}
	\HHH(\G,\B) [\F_q^{\times}] \simeq  \HHH(\G,\B_1),
	\end{align*}
	where $\B_1 = \B \cap \SL_n(\F_q)$.  Here one has 
	\begin{align*}
	\Irr \HHH(\G,\B_1) = \Irr \HHH(\G,\B) \times \Irr \F_q^{\times}
	\end{align*}
	and the $ u=q $ specialization of the natural map \eqref{e:natquot}  becomes
	\begin{align}\label{e:specialsurj}
	%\HHH(\G,\U) \twoheadrightarrow   \HHH(\G,\B_1) \simeq \HHH(\G,\B)[\F_q^\times]
	\HHH(\G,\U) \twoheadrightarrow   \HHH(\G,\B)[\F_q^\times]
	\end{align}
	where, for $\sigma \in \SS_n$ and $t \in \T$, the corresponding basis elements are mapped as follows:
	\begin{align*}
	T_\sigma \in \HHH(\G, \U) & \mapsto T_\sigma \in \HHH(\G, \B) \\
	T_t \in \HHH(\G, \U) & \mapsto \det t \in \F_q^\times \subseteq \HHH(\G, \B)[\F_q^\times].
	\end{align*}
	
	\begin{remark}\label{r:infltilde}
		The $\mathsf{T}_i $ from $\YY_{q-1,n} $ corresponds to $ {\omega_i} \in  \T \rtimes \SS_n$, %\rN,$
		whereas the $\mathsf{T}_i $ from $\HH_{q-1,n} $ corresponds to $\sigma_i \in \SS_n \subseteq \F_q^{\times}\times \SS_n. $
		Therefore, the $ u=1 $ specialization of \eqref{e:natquot} gives the surjection
		\begin{align}\label{e:surjtilde}
		\T \rtimes \SS_n & \twoheadrightarrow \F_q^{\times} \times \SS_n &  (t, \sigma) \in \T \rtimes \SS_n & \mapsto \left( (\det t) (\sgn \sigma) , \sigma \right) \in \F_q^\times \times \SS_n,
		\end{align}
		where $\sgn \sigma = (-1)^{\ell(\sigma)}\in \F_q^{\times}$.  For this reason we define
		\begin{align*}
		\wt{\Infl}:	\Irr \left( \SS_n \times \F_q^{\times} \right) \to \Irr \rN
		\end{align*}
		as composition with the map \eqref{e:surjtilde}.  Thus, for $\chi^{\F_q^{\times} \SS}_{\lambda ,\psi } := \chi^{\SS}_\lambda \otimes \psi \in \Irr(\F_q^{\times} \times \SS_n ) $ and $ (t, \sigma )\in \rN$, we have
		\begin{align}\label{e:infltilde}
		\wt{\Infl}\left(\chi^{\F_q^{\times} \SS}_{\lambda ,\psi } \right)(t,\sigma ) = 
		\psi\left( (\sgn \sigma )( \det t) \right) \chi^{\SS}_{\lambda } (\sigma ) . 
		\end{align}
	\end{remark}
	
	\begin{remark}\label{r:subltediffinfl}
		In general, we reserve the notation $\Infl $ for inflation by the natural quotient. In this case it is
		\begin{align*}
		\T \rtimes \SS_n  \twoheadrightarrow \F_q^{\times} \times \SS_n
		\end{align*}
		mapping $ (t, \sigma) \in \T \rtimes \SS_n $ to $ ( \det t , \sigma ) .$
		Therefore, by \eqref{e:infltilde}
		\begin{align*}
		\wt{\Infl}\left(\chi^{\F_q^{\times} \SS}_{\lambda ,\psi } \right)(t,\sigma ) = \psi(\sgn \sigma) \Infl \left(\chi^{\F_q^{\times} \SS}_{\lambda ,\psi } \right)(t,\sigma ) .
		\end{align*}
		
		Since $\psi \circ \sgn \in \Irr \SS_n$ is the sign representation when $\psi$ is odd (i.e., a non-square) character, and is trivial when $\psi$ is even (i.e., the square of a character), and tensoring with the sign representation amounts to taking the transpose partition $\lambda'$ %ref?
		we see that
		\begin{align*}
		\wt{\Infl}\left(  \chi^{}_{\lambda,\psi}    \right) = 
		\begin{cases}
		{\Infl}\left(  \chi^{}_{\lambda',\psi}  \right) & \mbox{ for }\psi\mbox{ odd,}\\
		{\Infl}\left(  \chi^{}_{\lambda,\psi}   \right) & \mbox{ for }\psi\mbox{ even,}  
		\end{cases}
		\end{align*}
		where the superscripts $ {\F_q^{\times} \SS} $ were omitted.
		%
		%For this reason it will be convenient to define for $\Lambda \in \calQ_n $ the function $\wt{\Lambda} \in \calQ_n$ as
		%\begin{align*}
		%\wt{\Lambda}\left(  \psi \right) := 
		%\begin{cases}
		%{\Lambda'}\left(  \psi \right)  & \mbox{ for }\psi\mbox{ odd,}\\
		%{\Lambda}\left(  \psi \right)  & \mbox{ for }\psi\mbox{ even.}  
		%\end{cases}
		%\end{align*}
	\end{remark}
	
	\begin{remark}\label{r:interestingvalues}
		In any case we have
		\begin{align*}
		\wt{\Infl}\left(\chi^{\F_q^{\times} \SS}_{\lambda ,\psi } \right)( t ) =  \psi(\det t ) \chi^{ \SS}_{\lambda } (1) = {\Infl}\left(\chi^{\F_q^{\times} \SS}_{\lambda ,\psi } \right)( t )  
		\end{align*}
		for  $ t \in \T \subseteq \rN$, and
		\begin{align*}
		\wt{\Infl}\left(\chi^{\F_q^{\times} \SS}_{\lambda ,\psi } \right)( \omega_i ) =   \chi^{ \SS}_{\lambda } (s_i) = \Infl \left(\chi^{\F_q^{\times} \SS}_{\lambda ,\psi } \right)( s_i )  
		\end{align*}
		where $\omega_i, s_i \in \rN$ as in Theorem~\ref{t:yokonuma}.  Thus, by definition of the characters $\chi_\Lambda^\rN \in \Irr \N$ for $\Lambda \in \calQ_n$ in Section \ref{s:charnormtor} and \eqref{e:Lamtildef}
		\begin{align} %\label{e:interestingvalues}
		\chi^{\rN}_{\wt{\Lambda}}(t) & = \chi^{\rN}_{\Lambda}(t) & & \textnormal{ and }  & \chi^{\rN}_{\wt{\Lambda}}(\omega_i) & = \chi^{\rN}_{\Lambda}(s_i) 
		\end{align}
		for $\Lambda \in \calQ_n$, $t \in \T$, and $\omega_i$, $s_i \in \rN$ as before.
	\end{remark}
	
	Let $\chi_\lambda^\G \in \Irr(\G : \B) \subseteq \Irr(\G : \U)$ be a unipotent character and let $\psi \in \Gamma_1 = \widehat{\F}_q^\times$.  As in Remark \ref{r:unipart}, we have $\psi^\G := \psi \circ \det : \G \to \C^\times$ and we consider the representation $\chi_\lambda^\G \otimes \psi^\G$.  Since $ \U \subseteq \ker \psi^\G$, $\dim (\chi_\lambda^\G \otimes \psi^\G)^\U = \dim (\chi_\lambda^\G)^\U > 0$, so $\chi_\lambda^\G \otimes \psi^\G \in \Irr(\G:\U)$.  Thus, taking tensor products gives a map $\Irr(\G : \B) \times \widehat{\F}_q^\times \to \Irr(\G : \U)$.
	
	\begin{proposition}\label{p:twist}
		Let $\psi \in \Gamma_1$ be as above.  Then the following diagram is commutative
		$$\xymatrix{
			\Irr(\G:\B)\ar[r]^-{ \DC_{\B} }  \ar@{_{(}->}[d]^-{ - \otimes \psi} & \Irr \HHH(\G,\B) \ar@{_{(}->}[d]^-{ - \otimes \psi} & \Irr \SS_n \ar[l]_-{ \TD_{\B} } \ar@{_{(}->}[d]^-{ - \otimes \psi} \\
			\Irr(\G:\B) \times \widehat{\F}_q^\times \ar[r]^-{ (\DC_{\B},\id) }  \ar@{_{(}->}[d]^-{ - \otimes - } & \Irr\big(\HHH(\G,\B)[\F_q^\times] \big) \ar@{_{(}->}[d]^-{ \Infl } & \Irr\big(\SS_n\times \F_q^\times \big) \ar[l]_-{ (\TD_{\B},\id) } 
			\ar@{_{(}->}[d]^-{ \wt{\Infl} } \\
			\Irr(\G:\U)\ar[r]^-{ \DC_{\U} } & \Irr \HHH(\G,\U) &  \Irr \rN. \ar[l]_-{ \TD_{\U} }
		}$$
	\end{proposition}
	
	\proof
	The commutativity of the squares on the left largely comes from the formula \eqref{Heckerepdef} for the action of the Hecke algebra, as well as noting that the $\HHH(\G, \U)$-action factors through the surjection onto $\HHH(\G, \B)[\F_q^\times] = \HHH(\G, \B_1)$, which comes from either \eqref{e:specialsurj} or Proposition \ref{p:Hprojection}.
	%Let us define $ \T_1 := \T \cap \SL_n(\F_q)$, and the central idempotent
	%\begin{align*}
	%e_1 := \frac{1}{|\T_1|} \sum_{t\in \T_1 } T_t \in \HHH(\G,\U)
	%\end{align*}
	%which acts trivially on $ (V,\pi\otimes\psi)^\U =(V,\pi)^\U .$ 
	%Thus the $ \HHH(\G,\U) $-action on $ (V,\pi\otimes\psi)^\U$ factors through the map
	%\begin{align*}
	%\varphi \longmapsto \varphi*e_1
	%\end{align*}
	%which we identify with the surjection .  This proves the commutativity of the squares on the left.
	
	%Let us consider a $ \G $-representation $ (V ,\pi )$ affording $\chi^\G_\lambda \in \Irr(\G:\B)$ and $\psi \in \Irr \F_q^\times$.  The $\G $-representation $( V ,\pi \otimes \psi )$ affords %, by inflating the second factor,
	%\begin{align*}
	%\chi^\G_\lambda \otimes \psi \in \Irr(\G:\U)
	%\end{align*}
	%which has the same $ \U $-invariants as $ (V ,\pi )$.
	
	The squares on the right are also commutative since all the non-horizontal arrows in
	$$\xymatrix{
		& \Irr \HH_n(u) \ar@{_{(}->}'[d][dd]^-{ \_\, \otimes \psi} \ar[dl]_-{ d_{\theta_q} } \ar[dr]^-{ d_{\theta_1} } & \\
		\Irr \HHH(\G,\B) \ar@{_{(}->}[dd]^-{ \_\, \otimes \psi} & & \Irr \SS_n \ar[ll] \ar@{}[l]_-{ \TD_{\B} } \ar@{_{(}->}[dd]^-{ \_\, \otimes \psi} \\
		& \Irr\left(\HH_n(u)[\F_q^\times] \right) \ar@{_{(}->}'[d][dd]^-{ \Infl } \ar[dl]_-{ d_{\theta_q} } \ar[dr]^-{ d_{\theta_1} } & \\
		\Irr \left(\HHH(\G,\B)[\F_q^\times] \right) \ar@{_{(}->}[dd]^-{ \Infl }  &  & \Irr\big(\SS_n\times \F_q^\times \big) \ar[ll] \ar@{}[l]_-{ (\TD_{\B},\id) } \ar@{_{(}->}[dd]^-{ \wt{\Infl} } \\
		& \Irr \YY_n(u) \ar[dl]_-{ d_{\theta_q} } \ar[dr]^-{ d_{\theta_1} } & \\
		\Irr \HHH(\G,\U) &  & \Irr \rN \ar[ll]_-{ \TD_{\U} } 
	}$$
	arise from tensor products.
	\endproof
	
	\subsubsection{Parabolic induction }\label{p:paraind}
	
	We now check that the parabolic induction of a product of irreducible representations is compatible with the operations $\DC_\U$ (coming from Section \ref{s:RepsHA}) and $\TD_\U$ (defined at \eqref{TDUdef}).
	
	%To this end, let us recall the natural isomorphisms for $R$-$Q$-bimodules
	%\begin{align}\label{e:tenhom}
	%\Hom_S({}_S^{}X_T^{},{}_S^{}Y_R^{}) \otimes_R^{} {}_R^{}Z_Q^{} \simeq
	%\Hom_S({}_S^{}X_T^{} , {}_S^{}Y_R^{}  \otimes_R^{} {}_R^{}Z_Q^{}) ,
	%\end{align}
	%valid for finitely generated and projective modules, where $ {}_S^{}{X}_T^{} $ is short for the $S$-$T$-bimodule $ X $.
	
	%, gives us
	%\begin{align}\label{e:homtenten}
	%\Hom_R({}_R^{}X_S^{},{}_R^{}Y_S^{}) \otimes_S^{} \Hom_S({}_S^{}X'_S,{}_S^{}Y'_T) \simeq
	%\Hom_R({}_R^{}X_S^{} \otimes_S^{} {}_S^{}X'_S ,{}_R^{}Y_S^{} \otimes_S^{} {}_S^{}Y'_T) .
	%\end{align}
	
	Suppose $n_1$, $n_2$ are such that $n = n_1 + n_2$.  For $i = 1$, $2$, let $\G_i := \GL_{n_i}(\F_q)$, and let $\B_i$, $\U_i$, and $\rN_i$ be the respective subgroups of $\G_i$ as described at the beginning of Section \ref{s:yokoyoko}.  Let $\Lg = \G_1 \times \G_2$ and we will view it as the subgroup of $\G = \GL_n(\F_q)$ of block diagonal matrices; let $\U_{12} \leeq \G$ be the subgroup of upper block unipotent matrices and $\mathrm{P} := \Lg \U_{12}$ the corresponding parabolic subgroup.  If $\B_\Lg := \Lg \cap \B$, then its unipotent radical is $\U_\Lg = \Lg \cap \U = \U_1 \times \U_2$.  One may also identify $\rN_1 \times \rN_2$ as a (proper) subgroup of $\rN$.
	
	\begin{remark} \label{r:parspeccase}
		Observe that $\U = \U_\Lg \ltimes \U_{12}$ and that $\Lg$ normalizes $\U_{12}$, so that we, with $G = \G$, $H = \U$, $L = \Lg$, $K = \U_\Lg$ and $U = \U_{12}$, we are in the situation of Section \ref{Heckeinclusions}, so that Proposition \ref{p:Hinclusion} gives an inclusion $\HHH(\Lg, \U_\Lg) \hookrightarrow \HHH(\G, \U)$.  Furthermore, the functor $R_\Lg^\G : \Rep \Lg \to \Rep \G$ defined at \eqref{parinddef} is the usual parabolic induction functor.
	\end{remark}
	
	\begin{remark} \label{r:IrrGUtwists}
		For $i = 1$, $2$, let $\psi_1$, $\psi_2 \in \widehat{\F}_q^\times$ be distinct and let $\lambda_i \in \P_{n_i}$.  Then by the discussion preceding Proposition \ref{p:twist}, $\chi_{\lambda_i}^\G \otimes \psi_i^{\G_i} \in \Irr(\G_i : \U_i)$.  Thus, Remark \ref{r:parspeccase} makes the discussion preceding Proposition \ref{p:DCind} relevant:  it says that $\left( \chi_{\lambda_1}^\G \otimes \psi_1^{\G_1} \right) \circ \left(\chi_{\lambda_i}^\G \otimes \psi_i^{\G_i} \right) \in \Rep(\G : \U)$.  However, this $\circ$-product is irreducible, as mentioned in Remark \ref{r:unipart}, so in fact it lies in $\Irr(\G: \U)$.  By a straightforward inductive argument, we see that for every $\Lambda \in \calQ_n$, one has $\chi_\Lambda^\G \in \Irr(\G : \U)$ (as defined in Remark \ref{r:unipart}).  We repeat the argument of Section \ref{s:IwHalg} to show that $\Irr(\B : \G)$ consists of precisely the unipotent representations:  one has an inclusion
		\begin{align*}
		\left\{ \chi_\Lambda^\G \, : \, \Lambda \in \calQ_n \right\} \subseteq \Irr(\G : \U),
		\end{align*}
		but since $|\Irr(\G : \U)| = |\Irr \HHH(\G, \U)| = |\Irr \rN| = |\calQ_n|$ via the bijections $\DC_\U$ and $\TD_\U$, we must have equality.
	\end{remark}

	\begin{remark}
		Since
		\begin{align*}
		\Hom_{\U_\Lg} \left( \1_{\U_{\Lg}},  \Res_{\U_{\Lg}}^{\Lg} \left( \chi_1^{\G_1}\otimes\chi_2^{\G_2} \right) \right) \cong \Hom_{\U_1} \left( \1_{\U_1}, \Res_{\U_1}^{\G_1}\chi_1^{\G_1} \right) \otimes_{\C}^{} \Hom_{\U_2} \left( \1_{\U_2}, \Res_{\U_2}^{\G_2}\chi_2^{\G_2} \right)
		\end{align*}
		as modules over the algebra
		\begin{align*}
		\HHH(\Lg,\U_{\Lg}) \cong \HHH(\G_1,\U_1) \otimes %^{}_{\C}
		\HHH(\G_2,\U_2),
		\end{align*}
		one has
		\begin{align}\label{e:canisotensor}
		\DC_{\U_{\Lg}}\big(\chi_1^{\G_1}\otimes\chi_2^{\G_2}\big) \cong \DC_{\U_1}\big(\chi_1^{\G_1} \big)\otimes %_{\C}^{}
		\DC_{\U_2}\big(\chi_2^{\G_2} \big).
		\end{align}
		
		%the expression~\eqref{e:parainduc5} becomes
		%\begin{align*}
		%\Ind_{\HHH(\Lg,\U_{\Lg})}^{\HHH(\G,\U)}
		% \chi_1^{\HHH_1}\otimes\chi_2^{\HHH_2}
		%\end{align*}
		%
		%Equating both expressions for ~\eqref{e:parainduc5} we get
		%\begin{align*}
		%\DC_{\lala}( \chi_1^{\G_1} \circ \chi_2^{\G_2} )=
		%\TD_{\lala}\left(\Ind^{\rN}_{\rN_1 \times \rN_2} 
		%\big( \chi_1^{\rN_1}  \chi_2^{\rN_2} \big)\right)
		%\end{align*}
		%
	\end{remark}
	
	%\begin{lemma}
	%Let $ (\Ind_{\U}^{\G}\1_\U )^* $ be the contragradient representation of $ \Ind_\U^\G \1_\U$.  Then
	%\begin{align} \label{e:lemmaDcC}
	%\End_{\HHH(\G,\U)} \big( (\Ind_{\U}^{\G}\1_\U )^* \big) \cong 
	%\C[\G] %poner opp aca?
	%\end{align}
	%\end{lemma}
	
	%\proof
	%From Proposition \ref{doublecentralizer} we see that the left hand side of \eqref{e:lemmaDcC} is the image of $ \C[\G]  $ acting on $\C[\G/\U]$.  We only need to check that the action of $ \G $ on $ \C[\G/\U] $ is faithful.  Let $ h \in \G $ be such that $ hg\U =g\U $ for all $ g\in \G$.  This implies $ hg \in g\U $ or equivalently, $ h \in g\U g^{-1} $ for all $ g \in \G$.  But the intersection $ \displaystyle \bigcap _{g \in \G} g\U g^{-1} $ is trivial.
	%\endproof
	
	%The same argument works for $ \Lg = \G_1 \times \G_2 $ instead of $ \G$,
	%\begin{align} \label{e:lemmaUL}
	%\End_{\HHH(\Lg,\U_{\Lg})}\big(\Ind_{\U_{\Lg}}^{\Lg}(\1_{\U_{\Lg}})^* \big) \cong \C[\Lg] 
	%\end{align}
	%where $ \U_{\Lg} $ is the unipotent radical of $ \Lg . $
	
	\begin{proposition} \label{p:parainduc}
		The following diagram is commutative:
		$$\xymatrix{
			\Irr(\G_1:\U_1)\times \Irr(\G_2:\U_2) \ar[d]_-{ - \circ - } \ar[r]^-{ \DC_{\U_\Lg} } & \Irr \HHH(\Lg,\U_\Lg) \ar[d]_-{ \Ind  } & \Irr(\rN_1\times \rN_2) \ar[l]_-{\TD_{\U_\Lg}} \ar[d]_-{ \Ind }\\
			\Rep(\G:\U)   \ar[r]^-{ \DC_{\U} } & \Rep \HHH(\G,\U) & \Rep \rN. \ar[l]_-{\TD_{\U}} 
		} $$
	\end{proposition}
	
	\proof
	Commutativity of the square on the left comes from Proposition \ref{p:DCind} which applies by Remark \ref{r:parspeccase} and observing that $\bigcap_{\ell \in \Lg} \ell \U_\Lg \ell^{-1}$ is trivial.
	
	Let us take representations $ \mathcal{X}_1^{\YY}, \mathcal{X}_2^{\YY}$ of $\YY_{q-1,n_1}(u)$ and $\YY_{q-1,n_2}(u)$ and denote their corresponding $u=1$ specializations by $\chi_1^{\rN_1}$ and $\chi_2^{\rN_2}$, and their $u=q$ specializations by $ \chi_1^{\HHH_1}$ and $\chi_2^{\HHH_2}$, respectively, so that
	\begin{align*}
	\TD_{\U_i}(\chi_i^{\rN_i}) = \chi_i^{\HHH_i} .
	\end{align*}
	
	For $ i=1,2 $ we write $ \chi_i^{\G_i} $ for the corresponding $ \G_i $ representation, in such a way that $\chi_i^{\HHH_i}$ agrees with
	\begin{align*}
	\Hom_{\U_i}\big( \1_{\U_i}, \Res^{\G_i}_{\U_i}\chi_i^{\G_i} \big)
	\end{align*}
	namely, the $ \HHH(\G_i,\U_i)$-representation $\DC_{\U_i}(\chi_i^{\G_i})$.
	
	By \eqref{e:canisotensor}, the commutativity of the rightmost square follows since the vertical and diagonal arrows in 
	$$\xymatrix{
		& \Irr \big(\YY_{n_1}(u) \otimes^{}_{\C(u)} \YY_{n_2}(u) \big) \ar'[d][dd]^-{ \Ind } \ar[dl]_-{ d_{\theta_q} } \ar[dr]^-{ d_{\theta_1} } & \\
		\Irr \HHH(\Lg,\U_\Lg) \ar[dd]^-{  \Ind } &  
		& \Irr(\rN_1\times \rN_2) \ar[ll] \ar@{}[l]_-{ \TD_{\U_{\Lg}} } \ar[dd]^-{  \Ind  } \\
		& \Rep \YY_n(u) \ar[dl]_-{ d_{\theta_q} } \ar[dr]^-{ d_{\theta_1} } & \\
		\Rep \HHH(\G,\U) &  & \Rep \rN \ar[ll]_-{ \TD_{\U} }
	}$$ 
	come from taking tensor products. Namely
	\begin{align*}
	\Ind^{\HHH(\G,\U)}_{\HHH(\Lg,\U_{\Lg})}
	\TD_{\U}\big( \chi_1^{\rN_1} \otimes \chi_2^{\rN_2}\big) =
	\HHH(\G,\U)\otimes_{\HHH(\Lg,\U_{\Lg})}^{}
	(\chi_1^{\HHH_1} \otimes \chi_2^{\HHH_2})
	\end{align*}
	happens to be the $ u=q $ specialization of
	\begin{align*}
	%\YY_n\otimes_{ \YY_{n_1} \otimes^{}_{\C[u^{\pm 1}]} \YY_{n_2}}^{}
	\YY_n\otimes_{ \YY_{n_1} \otimes \YY_{n_2}}^{}
	\big( \mathcal{X}^{\YY}_1 \otimes^{}_{} \mathcal{X}^{\YY}_2 \big).
	\end{align*}
	Taking the $u=1$ specialization gives
	\begin{align*}
	\C [\rN] \otimes^{}_{ \C [\rN_1 \times \rN_2]} \big( \chi_1^{\rN_1}\otimes   \chi_2^{\rN_2} \big) = \Ind^{\rN}_{\rN_1 \times \rN_2} \big( \chi_1^{\rN_1} \otimes  \chi_2^{\rN_2} \big)
	\end{align*}
	proving thus 
	\begin{align*}
	\Ind^{\HHH(\G,\U)}_{\HHH(\Lg,\U_{\Lg})} \TD_{\U_{\Lg}}\big( \chi_1^{\rN_1} \otimes \chi_2^{\rN_2}\big) = \TD_{\U}\left(\Ind^{\rN}_{\rN_1 \times \rN_2} \big( \chi_1^{\rN_1} \otimes \chi_2^{\rN_2} \big)\right). & \qedhere
	\end{align*}
	\endproof
	
	\subsubsection{Proof of Proposition~\ref{p:expbijection}}\label{yokotable}
	
	We define a family of representations of $ \YY_{q-1,n} $ parametrized by $ \calQ_n $ whose $u=1$ specialization matches that of Section~\ref{s:charnormtor}, and whose $u=q$ specialization gives the $ \DC_{\U} $ image of the characters described in Remark~\ref{r:unipart}.  
	
	Given a partition $\lambda \vdash n$ and $\psi : \F_q^\times \to \C^\times$ a character, let us define a representation $\mathcal{X}^{\YY_{q-1, n}}_{\lambda,\psi}$ of $\YY_{q-1, n}$ by inflating the $ \HH_{q-1,n} $-representation $\mathcal{X}^{\HH_n}_{\lambda} \otimes \psi$ from \eqref{e:cartprodtable} via the natural quotient \eqref{e:natquot}.
	
	Now, for $\Lambda \in \calQ_n$, let $\psi_i$, $\lambda_i$, $n_i$ be as in the paragraph preceding \eqref{e:chiLambdaG} and consider the $\C[u ^{\pm 1}]$-algebra 
	\begin{align*}
	\YY_{(n_i)}:=\bigotimes_{i} \YY_{q-1,n_i}
	\end{align*}
	which may be embedded in $\YY_{q-1,n }$ by a choice of ordering of the indices $i$.  This has the representation 
	\begin{align*}
	\mathcal{X}^{\YY_{(n_i)}}_{\Lambda} := \bigotimes_{i}\mathcal{X}^{\YY_{q-1, n_i}}_{\lambda_i,\psi_i},
	\end{align*}
	where the tensor is over $ \C[u^{\pm 1}] $, and we may define $ \mathcal{X}^{\YY}_{\Lambda} $ as the induced character
	\begin{align}\label{e:defdefchar}
	%\mathcal{X}^{Y}_{\Lambda}:= \Ind^{\C(u)\YY_{q-1,n}}_{\C(u) Y_{(n_i)}} \mathcal{X}^{Y_{(n_i)}}_{\Lambda}.
	\mathcal{X}^{\YY}_{\Lambda}:= \Ind^{\YY_{q-1,n}}_{\YY_{(n_i)}} \mathcal{X}^{\YY_{(n_i)}}_{\Lambda}.
	\end{align}
	
	%\begin{remark}
	%The $ \mathcal{X}^{Y}_{\Lambda} $ gives us the irreducible representations of $ \C(u)Y_{d,n}. $
	%\end{remark}
	
	We are now in position to prove Proposition~\ref{p:expbijection}.
	
	\proof[Proof of Proposition~\ref{p:expbijection}]
	Take a $ \Lambda \in \calQ_n$ and the corresponding pairs $ (\lambda_i, \psi_i) $ with $\Lambda(\psi_i) = \lambda_i$, and define $ n_i = |\lambda_i |$.  Let us write  $\rN_i = (\F_q^\times)^{n_i} \rtimes \SS_{n_i}$, $\G_i=\GL_{n_i}(\F_q)$, $\Lg=\prod_{i} \G_i $ viewed as the subgroup of block diagonal matrices of $ \G$,  and $\U_{\Lg}$ the product of the upper triangular unipotent subgroups.
	
	Consider the $\YY_{q-1,n}$-representation $\mathcal{X}^{\YY}_\Lambda$ defined in \eqref{e:defdefchar}.
	Its $ u=1 $ and $ u=q $ specializations give, by the commutativity of the squares on the right in Proposition \ref{p:twist},
	\begin{align*}
	\chi_{\wt{\Lambda}}^{\rN} & := \Ind^{\rN}_{\prod \rN_i} \bigotimes_{i}  \wt{\Infl}\left( \chi^{\SS_{n_i}}_{\lambda_i} \otimes \psi_i \right) & 
	& \textnormal{ and } &
	\chi_{\Lambda}^{\HHH} & := \Ind^{\HHH(\G,\U)}_{ \HHH(\Lg,\U_{\Lg}) } \bigotimes_{i} {\Infl}\left( \chi^{\HH_{n_i}}_{\lambda_i} \otimes \psi_i \right)
	\end{align*} 
	regarding $ \prod_{i} \rN_i $ as a subgroup of $\rN$ embedded by the same choice of indices as $\YY_{(n_i)}$ is in $\YY_{q-1,n}$.  Therefore, 
	\begin{align}\label{e:RHSexpBij}
	\TD_{\U}(\chi_{\wt{\Lambda}}^{\rN} ) = \chi_{\Lambda}^{\HHH}.
	\end{align}
	
	The character $\chi_\Lambda^\G \in \Irr(\G : \U)$ was defined at~\eqref{e:chiLambdaG}.  Applying $\DC_\U$, one gets, by Proposition~\ref{p:parainduc} and an inductive argument,
	\begin{align*}
	\DC_{\U}(\chi_{\Lambda}^{\G}) & = \Ind^{\HHH(\G,\U)}_{ \HHH(\Lg,\U_{\Lg}) } \bigotimes^{}_i \: \DC_{\U_i}(\chi^{\G_i}_{\lambda_i} \otimes \psi_i ) = \Ind^{\HHH(\G,\U)}_{ \HHH(\Lg,\U_{\Lg}) } \bigotimes^{}_i \: \DC_{\B_i}(\chi^{\G_i}_{\lambda_i}) \otimes \psi_i \\
	& = \Ind^{\HHH(\G,\U)}_{ \HHH(\Lg,\U_{\Lg}) } \bigotimes^{}_i \: \chi^{\HH_{n_i}}_{\lambda_i} \otimes \psi_i
	\end{align*}
	where we have used Proposition~\ref{p:twist} and Proposition~\ref{p:BtoU} for the following two equalities.  But this is exactly $ \chi_\Lambda^{\HHH}$.  So we may conclude by taking a second look at \eqref{e:RHSexpBij}.
	\endproof
	
	\begin{remark}\label{r:rationality}
		%The localization 
		%$ \YY_{(n_i)}(u):= \C(u)\otimes_{\C[u^{\pm 1}]}^{} \YY_{(n_i)} $
		%is semisimple and the localized $ \mathcal{X}^{\YY_{(n_i)}}_{\Lambda} $ 
		%are irreducible.
		One can use this construction to give another proof of the splitting of $ \YY_{d,n}(u)$, together with a list of their irreducible representations parametrized by the set
		\begin{align*}
		\calQ_{d,n} := \left\{
		\Lambda : \dual{C}_d \to \PP\; : \; \sum_{\phi \in \widehat{C}_d} |\Lambda(\phi)|=n  
		\right\}  ,
		\end{align*}
		where $ C_d $ is the cyclic group with $d$ elements.
		
		%We include here the argument for the sake of self-containedness.
		
		Given a function $\Lambda \in \calQ_{d,n}$ we consider its set of pairs $(\psi_i, \lambda_i)$ with $ \Lambda(\psi_i) = \lambda_i$ and $n_i = |\lambda_i|$ adding up to $n$.  Define
		\begin{align*}
		\mathcal{X}^{\YY_{(n_i)}}_{\Lambda} := \bigotimes_{i}\mathcal{X}^{\YY_{d, n_i}}_{\lambda_i,\psi_i} 
		\end{align*}
		for the $ \C[u^{\pm 1}]$-algebra $ \YY_{(n_i)}:= \bigotimes_{i} \YY_{d,n_i}$,
		where the $ \mathcal{X}^{\YY_{d, n}}_{\lambda,\psi}  $
		stands for the inflation of the $ \mathcal{X}^{\HH_{d, n}}_{\lambda,\psi} $
		from \eqref{e:cartprodtable}.
		
		The induced characters $ \mathcal{X}^{\YY}_{\Lambda}:= \Ind^{\YY_{d,n}(u)}_{\YY_{(n_i)}(u)} \mathcal{X}^{\YY_{(n_i)}}_{\Lambda} $ are defined over $ \C(u) $ since the $ \mathcal{X}^{\HH_n}_{\lambda} $ are (cf.\ \cite[Theorem 2.9]{Benson-Curtis}).  We extend scalars to some finite Galois extension $ \mathbb{K}/\C(u) $ so that $ \YY_{d,n} $ becomes split.  One can also extend the $ u=1 $ specialization as in \cite[8.1.6]{GeckPfeiffer} and $ \mathbb{K}\YY_{d,n} $ becomes isomorphic to $ \mathbb{K}[ C_d^{n}\rtimes\SS_n]$, being a deformation of $ \C[ C_d^{n}\rtimes \SS_n]$.  The specializations 
		$$  \chi_{\Lambda}^{ C_d^{n}\rtimes \SS_n} : = d_{\theta_1} (\mathcal{X}^{\YY}_{\Lambda}) $$ 
		are all the irreducible characters of $ C_d^{n}\rtimes \SS_n $ as in Section \ref{s:charnormtor} (invoking again \cite[\S~8.2, Proposition 25]{Serre}).  Therefore, $ \YY_{d,n}(u) $ is split semisimple and the $ \mathcal{X}^{\YY}_{\Lambda} $ is the full list of irreducibles.
	\end{remark}
	
	\section{Counting on wild character varieties}
	\label{countwcv}

\subsection{Counting on quasi-Hamiltonian fusion products}

Here we describe the technique we use to count the points on the wild character varieties, which was already implicitly used in \cite{hausel-villegas, hausel-aha1}.  The idea is to use the construction of the wild character variety as a quotient of a fusion product and reduce the point-counting problem to one on each of the factors.  Then the counting function on the entire variety will be the convolution product of those on each of the factors.  This can be handled by a type of Fourier transform as in the references above.

\subsubsection{Arithmetic harmonic analysis}\label{aha}

In carrying out our computations, we will employ the technique of ``arithmetic harmonic analysis,'' which is an of analogue of the Fourier transform for non-abelian finite groups such as  $\GL_n(\F_q)$.  This is described in \cite[\S3]{hausel-aha1}, a part of which we reproduce here for the convenience of the reader.

Let $G$ be a finite group, $\Irr G$ the set of irreducible character of $G$, $C(G_\bullet)$ the vector space of class functions (i.e., functions which are constant on conjugacy classes) on $G$ and $C(G^\bullet)$ the space of functions on the set $\Irr G$.  We define isomorphisms
\begin{align*}
& \FF_\bullet : C(G_\bullet) \to C(G^\bullet) & & \FF^\bullet : C(G^\bullet) \to C(G_\bullet) 
\end{align*}
by
\begin{align*}
\FF_\bullet(f)(\chi) & := \sum_{x \in G} f(x) \frac{\chi(x)}{\chi(e)} & \FF^\bullet(F)(x) & := \sum_{\chi \in \Irr G} F(\chi) \chi(e) \overline{\chi}(x).
\end{align*}
Note that these are not quite mutually inverse, but will be up to a scalar; precisely,
\begin{align*}
\FF^\bullet \circ \FF_\bullet & = |G| \cdot 1_{C(G_\bullet)} & \FF_\bullet \circ \FF^\bullet & = |G| \cdot 1_{C(G^\bullet)}.
\end{align*}

It is clear that $C(G_\bullet)$ is a subspace of the group algebra $\C[G]$; it is not difficult to verify that it is in fact a subalgebra for the convolution product $\ast_\G$.  We can define a product on $C(G^\bullet)$ simply by pointwise multiplication:
\begin{align*}
(F_1 \cdot F_2) (\chi) := F_1(\chi) F_2(\chi).
\end{align*}
Then $\FF_\bullet$ and $\FF^\bullet$ have the important properties that
\begin{align} \label{FourierMult}
\FF_\bullet(f_1 \ast_\G f_2) & = \FF_\bullet(f_1) \cdot \FF_\bullet(f_2) & |G| \cdot \FF^\bullet(F_1 \cdot F_2) & = \FF^\bullet(F_1) \ast_\G \FF^\bullet(F_2)
\end{align}
for $f_1, f_2 \in C(G_\bullet), F_1, F_2 \in C(G^\bullet)$.

\subsubsection{Set-theoretic fusion}\label{setteofusion}

Let $G$ be a finite group, $M$ a (left) $G$-set and $\mu : M \to G$ an equivariant map of sets (where $G$ acts on itself by conjugation).  We may define a function $N : G \to \Z_{\geq 0}$ by
\begin{align*}
N(x) := \left| \mu^{-1}(x) \right|.
\end{align*}
The equivariance condition implies that $m \mapsto a \cdot m$ gives a bijection $\mu^{-1}(x) \leftrightarrow \mu^{-1}(axa^{-1})$ for $a, x \in G$, and hence it is easy to see that $N \in C(G_\bullet)$.

Suppose $M_1$ and $M_2$ are two $G$-sets and $\mu_1 : M_1 \to G, \mu_2 : M_2 \to G$ are equivariant maps, and let $M := M_1 \times M_2$ and define $\mu : M \to G$ by
\begin{align*}
\mu(m_1, m_2) = \mu_1(m_1) \mu_2(m_2).
\end{align*}
Then since for $x \in G$,
\begin{align*}
\mu^{-1}(x) = \coprod_{a \in G} \mu_1^{-1}(a) \times \mu_2^{-1}(a^{-1} x),
\end{align*}
a straightforward computation gives
\begin{align}\label{f:convofusion}
N_M = N_{M_1} \ast_G N_{M_2}.
\end{align}
%\begin{corollary} In the notation of Lemma~\ref{traceconj} let $\phi_1,\phi_2\in \HHH(G, H)$ and define $N_{\phi_i}(g):=\tr(g\phi)$. Then we have 
%$$N_{\phi_i}\ast_\G N_{\phi_i}(g)=\sum_{ \substack{V \in \Irr \, G \\ V^H \neq \{ 0 \}} } \frac{\chi_V(g)}{\chi_V(1)}  \mathcal{X}_{V^H}( \varphi_1)\mathcal{X}_{V^H}( \varphi_2) $$
%\end{corollary} 

\subsection{Counting via Hecke algebras}

%\paragraph{Double Cosets}

%Let $H, \overline{H} \leq G$ be subgroups.  Let $V \subseteq G$ be a set of double $(\overline{H}, H)$-coset representatives, i.e., so that
%\begin{align*}
%G = \coprod_{v \in V} \overline{H} v H.
%\end{align*}
%It is then a simple observation that $V^{-1} := \{ v^{-1} \, : \, v \in V \}$ is a set of double $(H, \overline{H})$-coset representatives.

%The double cosets are precisely the orbits of the $\overline{H} \times H$-action on $G$ given by
%\begin{align*}
%(\bar{h}, h) \cdot g := \bar{h}g h^{-1}.
%\end{align*}
%Given $v \in V$, set
%\begin{align*}
%\overline{H}_{v,H} := \{ \bar{h} \in \overline{H} \, : \, v^{-1} \bar{h} v \in H \};
%\end{align*}
%it is easy to see that this is a subgroup of $\overline{H}$, and the following is straightforward to verify.

%\begin{lemma}
%If $(\overline{H} \times H)_v$ is the isotropy group of $v$ for the above action, then $(\overline{H} \times H)_v \cong \overline{H}_{v,H}$ via the projection map.
%\end{lemma}

%From this, we find
%\begin{align*}
%\overline{H} v H = ( \overline{H} \times H) \cdot v \cong (\overline{H} \times H)/ (\overline{H} \times H)_v,
%\end{align*}
%and hence
%\begin{align} \label{stabcosetorder}
%\left| \overline{H}_{v,H} \right| = \left|(\overline{H} \times H)_v \right| = \frac{ \left| \overline{H} \right| \left|H \right|}{ \left| \overline{H} v H \right|}.
%\end{align}

%\paragraph{Counting Functions}

Recall the notation of Section \ref{Halgnotn}.  Let $V \subseteq G$ be a set of double $H$-coset representatives as in \eqref{Bruhatdecomp}.  Let $k\in \Z_{>0}$, $x \in G$ and  $$\v=(v_1,w_1,\dots,v_{k},w_k)\in V^{2k}.$$  For $h \in H$ we set

\begin{align*}
\NN_h(x, \v) := \left\{ (a, a_1, \ldots, a_k) \in G \times v_1 H w_1 H \times \cdots \times v_k H w_k H  \, : \, a x a^{-1} = ha_1\cdots a_k \right\}.
\end{align*} Often we will abbreviate
$\NN(x,\v):=\NN_1(x,\v)$.  We are interested in the function \beq  \begin{array}{cccc}  N : & G \times V^{2k} & \to & \N \\
	&(x,\v)& \mapsto& |\NN(x,\v)|. \end{array} \label{gencount} \eeq

\begin{proposition}\label{count}  
	One has
	\begin{align*}
	N(x, \v) = \frac{\left| H \right|^{2k}}{\left| H v_1 H \right|\dots \left| H v_k H \right| } \tr \left(x  T_{v_1} \ast T_{w_1}\ast \dots \ast T_{v_k} \ast T_{w_k}\right).
	\end{align*}
\end{proposition}

\begin{proof}
	For $h \in H$, we have bijections
	\begin{align*}
	\NN(x,\v) & \leftrightarrow \NN_h(x, \v) & (a,a_1,\dots,a_{k-1},a_k) &\leftrightarrow (ha,a_1,\dots,a_{k-1},a_kh^{-1}) \end{align*}
	and hence $|\NN(x,\v)| = |\NN_h(x, \v)|$.  From this, we get
	\begin{align*}
	N(x, \v) & = \frac{1}{|H|} \sum_{h \in H} | \NN_h(x, \v) |.
	\end{align*}
	
	On the other hand, one sees that 
	\begin{align*}
	& \sum_{a \in G} \left( \II_H \ast_G \II_{ v_1 H w_1 H}\ast_G \dots \ast_G \II_{ v_k H w_k H}\right) (a x a^{-1}) \\
	= & \sum_{a \in G} \sum_{h \in G} \II_H(h) \left( \II_{ v_1 H w_1 H}\ast_G \dots \ast_G \II_{ v_k H w_k H} \right) ( h^{-1} a x a^{-1}) \\
	= & \sum_{h \in H} \sum_{a \in G} \left( \II_{ v_1 H w_1 H}\ast_G \dots \ast_G \II_{ v_k H w_k H} \right) (h^{-1} a x a^{-1}) = \sum_{h \in H} | \NN_h(x, \v) | .
	\end{align*}
	Hence
	\begin{align} \label{firstcount}
	N(x, \v) = \frac{1}{|H|} \sum_{a \in G} \left( \II_H \ast_G \II_{ v_1 H w_1 H}\ast_G \dots \ast_G \II_{ v_k H w_k H} \right) (a x a^{-1}).
	\end{align}
	Therefore, if we set
	\begin{align*}
	\varphi_{\v} := \II_H \ast_G \II_{ v_1 H w_1 H}\ast_G \dots \ast_G \II_{ v_k H w_k H} 
	\end{align*}
	then Lemma \ref{traceconj} applied to \eqref{firstcount} gives us
	\begin{align} \label{counttr}
	N(x, \v) = \tr (x \varphi_{\v} ),
	\end{align}
	
	To compute $\varphi_{\v}$ in $ \HHH(G,H) $ we need 
	
	\begin{lemma}\label{l:secondcount}
		For $v, v_1, w_1 \in V$, one has 
		\begin{enumerate}[(a)]
			\item $|H \cap v^{-1} H v|=\frac{|H|^2}{|HvH|}$;
			\item 
			$
			\II_{H} \ast_G \II_{v H}  = \frac{\left| H \right|^2 }{\left| H v H \right|} \II_{H v H}=  \II_{Hv} \ast_G \II_H ;
			$
			\item
			$
			\II_H \ast_G \II_{ v_1 H w_1 H} = \II_{Hv_1} \ast_G \II_{Hw_1 H};
			$
			\item $\II_{ v_1 H w_1 H}  \ast_G \II_H= {|H|} \II_{ v_1 H w_1 H}$.
		\end{enumerate}
	\end{lemma}
	
	\begin{proof}
		One finds in \cite[\S1]{Iwahori1964} a bijection $(H \cap v^{-1} H v) \backslash H {\leftrightarrow} H \backslash H v H$ given by
		\begin{align*}
		(H \cap v^{-1} H v) h \mapsto H vh.
		\end{align*}
		This quickly yields (a). 
		
		Now we let $x \in G$ and evaluate
		\begin{align*}
		\left( \II_{H} \ast_G \II_{vH} \right) (x) = \sum_{h \in G} \II_{H}(h) \II_{vH}( h^{-1} x) = \sum_{h \in H} \II_{vH}( h^{-1} x).
		\end{align*}
		If $x \not\in HvH$, then $h^{-1} x \not\in vH$ for any $h \in H$, and so the above is zero.  On the other hand, if $x = h_0 v h$ for some $h_0 \in H, h \in H$, then the above evaluates to
		\begin{align*}
		\left| \left\{ h \in H \, : \, h^{-1} h_0 v h \in v H \right\} \right| = \left| \left\{ h \in H \, : \, v^{-1} hv \in H \right\} \right| = \left|H \cap v^{-1} H v  \right| = \frac{ \left| H \right|^2 }{ \left| H v H \right|},
		\end{align*}
		using (a) for the last step.  The second equation in (b) is proved similarly.  
		
		For (c), we have
		\begin{align*}
		\left( \II_{Hv_1} \ast_G \II_{H w_1 H} \right) (x) & = \sum_{a \in G} \II_{H v_1}(a) \II_{H w_1 H} (a^{-1} x) = \sum_{a \in H} \II_{H v_1} (hv_1) \II_{H w_1 H}( v_1^{-1} h^{-1} x) \\
		& = \sum_{h \in H} \II_H(h) \II_{v_1 H w_1 H}( h^{-1} x) = \left( \II_H \ast_G \II_{v_1 H w_1 H} \right) (x). 
		\end{align*}
		
		Finally, for (d) we compute
		\begin{align*}
		\left( \II_{ v_1 H w_1 H}  \ast_G \II_H \right) (x) & = \sum_{a \in G}  \II_{v_1H w_1 H} (xa^{-1})\II_{H }(a) = \sum_{h \in H}  \II_{v_1H w_1 H} (xh^{-1}) = \sum_{h \in H}  \II_{v_1H w_1 H} (x)   \\
		& = {|H|} \II_{ v_1 H w_1 H}.  & \qedhere
		\end{align*}
	\end{proof}
	
	We can conclude the proof of Proposition~\ref{count} by noting that
	\begin{align} \label{phik1}
	\II_H \ast_G \II_{v_1 H w_1 H} = \II_{H v_1 } \ast_G \II_{H w_1 H} & = \frac{\left| H w_1 H \right|}{\left| H \right|^2 } \II_{Hv_1} \ast_G \II_{H} \ast_G \II_{w_1 H} \nonumber \\
	& = \frac{\left| H w_1 H \right|}{\left| H \right|^3 } \II_{Hv_1} \ast_G \II_{H} \ast_G \II_{H} \ast_G \II_{w_1 H} = \frac{\left| H \right|}{\left| H v_1 H \right|} \II_{H v_1 H} \ast_G \II_{ H w_1 H } \nonumber \\
	& = \frac{\left| H \right|^2}{\left| H v_1 H \right|} T_{v_1} \ast T_{w_1},
	\end{align}
	using $\II_H=\frac{\II_H\ast_G \II_H}{|H|}$ and \eqref{gpalgHalg} noting the relation between the two different products $\ast_G$ and $\ast$ explained in Remark~\ref{twostar}.  Thus, by Lemma~\ref{l:secondcount}(d) and \eqref{phik1} 
	\begin{align*}
	\varphi_{\v} & = \frac{1}{|H|^k} \II_H \ast_G \II_{v_1 H w_1 H} \ast_G  \II_H \ast_G \II_{v_2 H w_2 H} \ast_G \dots \ast_G \II_H \ast_G \II_{v_k H w_k H} \\
	& = \frac{\left| H \right|^{2k}}{\left| H v_1 H \right|\dots \left| H v_k H \right| }   T_{v_1} \ast T_{w_1}\ast \dots \ast T_{v_k} \ast T_{w_k}.
	\end{align*}
	This and \eqref{counttr} imply Proposition~\ref{count}.
\end{proof}

\begin{remark} When $k=1$ and $v_1=1$, Proposition~\ref{count} gives a character formula for the cardinality of the intersection of conjugacy classes and Bruhat strata and appears at \cite[1.3.(a)]{Lusztig}. In fact, the computation there is what led us to Proposition~\ref{count}. 
\end{remark}

\subsection{Character values at the longest element}

From now on we will let $\G$, $\T$, $\B$, $\U$ and $\rN$ be as in Section \ref{s:yokoyoko}.  We need to compute certain values of the characters of $\HHH(\G,\U)$.

\begin{remark}\label{r:idempotent} Let $(V,\pi)$ be a representation of $\mathsf{Y}_{d,n}(u),$ and fix $i \in [1, n-1]$.  The element $\mathsf{e}_i$ is the idempotent projector to the subspace $V_i$ of $V$ where $\mathsf{h}_i\mathsf{h}_{i+1}^{-1}$ acts trivially and there is a direct sum decomposition $V=V_i\bigoplus W_i$, where $W_i := \ker \mathsf{e}_i$. % Since $h_ih_{i+1}^{-1}$ is diagonalizable, so is $e_i$ and 
	Lemma \ref{eiTicomm} shows that this decomposition is preserved by $\mathsf{T}_i$.  Over $V_i,$ the endomorphism $\mathsf{T}_i$ satisfies the following quadratic relation
	\begin{align}\label{e:quadrelTV}
	\mathsf{T}_i|_{V_i}^{2} = u\cdot 1  + (u-1)\mathsf{T}_i|_{V_i}
	\end{align}
	whereas over $W_i$ it satisfies
	%\begin{align}\label{e:quadrelTW}
	% \mathsf{T}_i|_{W_i}^{2} = u\cdot 1.
	%\end{align}
	
	\begin{align}\label{e:quadrelTWsq}
	\mathsf{T}_i|_{W_i}^{4} = u^2 \cdot 1.
	\end{align}
	
\end{remark}

\begin{theorem}\label{t:scalarvaluecentral}
	For $\calX \in \Irr \mathsf{Y}_{d,n}(u)$ the element $\mathsf{T}_0^2$ acts by scalar multiplication by
	\begin{align}
	z_\calX = u^{f_\calX} 
	\end{align}
	where $f_\calX := \binom{n}{2}  \pare{  1 +  \frac { \calX_1(\omega) } {\calX_1(1)}  }$ and 
	%$\sigma \in \rN$ is the permutation matrix of any transposition (all such are conjugate). 
	$\omega \in \rN$ is any of the $\omega_i$ from Theorem~\ref{t:yokonuma} 
	(all such are conjugate).  In particular, specializing to $u = q$, the central element $T_{\omega_0}^2 \in \HHH(\G, \U)$ acts by the scalar $q^{f_\calX}$.
\end{theorem}

\begin{proof}
	Let $(V,\pi)$ be an irreducible representation affording $\calX$.  In the notation of Remark~\ref{r:idempotent}, 
	\eqref{e:quadrelTV} shows that the possible 
	eigenvalues of $\mathsf{T}_i|_{V_i}$ are $u$ and $-1$, and \eqref{e:quadrelTWsq} shows that those of 
	$\mathsf{T}_i|_{W_i}$ are $\pm \sqrt{u}$, $ \pm i \sqrt{u} $.  Thus, 
	\begin{align*}
	\calX(\mathsf{T}_i) = 
	m_1^{+} \sqrt{u} - m_1^{-} \sqrt{u}
	+ m_1^{\oplus} i\sqrt{u} - m_1^{\ominus}i \sqrt{u} - m_0 + m_2 u,
	\end{align*}
	where $ m_1^+, m_1^- , m_1^\oplus, m_1^{\ominus} , m_0$ and $m_2$ are the respective multiplicities of $\sqrt{u} , -\sqrt{u} ,i\sqrt{u} , -i\sqrt{u} , -1$ and $u$ as eigenvalues of $\mathsf{T}_i$.  
	
	Since %$ \C[u^{\pm 1}] $ is a PID, by \cite[7.3.8]{GeckPfeiffer} 
	$\calX(\mathsf{T}_i) \in \C[u^{\pm 1}]$, 
	we know $m_1^+ = m_1^- =: m_1$ and $m_1^\oplus = m_1^\ominus =: m_1^{\circ}$, and so we get
	\begin{align}
	\calX(\mathsf{T}_i) & = -m_0 + m_2 u \label{chiTi} \\
	\dim V = \calX(1) & = 2m_1 + 2m_1^{\circ} + m_0 + m_2 \label{chi1}\\
	\det \pi(\mathsf{T}_i) & = (-1)^{m_0} (\sqrt{u})^{m_1} (-\sqrt{u})^{m_1}(i\sqrt{u})^{m_1^{\circ}} (-i\sqrt{u})^{m_1^{\circ}} u^{m_2} \nonumber \\
	& = (-1)^{m_0 + m_1} u^{m_1 +m_1^{\circ}+ m_2}. \label{detTi}
	\end{align}
	Since $\mathsf{T}_0^{2}$ is central, Schur's Lemma implies that it acts by scalar multiplication by some $z_\calX \in \C[u^{\pm 1}]$.  Let $i_1, \ldots, i_{\binom{n}{2}}$ be as in \eqref{longestelt}. Taking determinants we find
	\begin{align} \label{zchi1}
	z_\calX^{\dim V} = \det \pi(\mathsf{T}_0^2) = \det \left( \pi\left( \mathsf{T}_{i_1} \cdots \mathsf{T}_{i_{\binom{n}{2}}} \right) \right)^2 = \det \pi( \mathsf{T}_i)^{n(n-1)}
	\end{align}
	for any $1 \leq i \leq n-1$ since the $\mathsf{T}_i$ are all conjugate (Lemma \ref{Yconj}).
	
	Now, under the specialization $u=1$, $\mathsf{T}_i$ maps to $\omega_i$ and so \eqref{chiTi} gives
	\begin{align} \label{chis1}
	\calX_1(\omega_i) = - m_0 + m_2.
	\end{align}
	Substituting \eqref{detTi} into \eqref{zchi1} gives
	\begin{align*}
	z_\calX^{\dim V} = u^{ \binom{n}{2} (2m_1 + 2m_1^{\circ} + 2m_2)} = 
	u^{ \binom{n}{2} (2m_1 + 2m_1^{\circ} + m_0 + m_2 - m_0 + m_2)} = u^{ \binom{n}{2} (\calX(1) + \calX_1(\omega_i))},
	\end{align*}
	where we use \eqref{chi1} and \eqref{chis1} for the last equality.  Taking $\calX(1)$th roots, we find
	\begin{align*}
	z_\calX = \zeta u^{f_\calX}
	\end{align*}
	where $\zeta$ is some root of unity and $f_\calX$ is as in the statement, recalling that $\calX_1(1) = \calX(1)$. 
	
	It remains to show that $\zeta = 1$.  We do this by specialising $u=1$.  We note that in this specialisation $\theta_1(\mathsf{T}_0^2)= \sigma_0^2 = 1\in \C[\rN]$. Then \bes \zeta \calX(1) = \theta_1 \big( \calX(\mathsf{T}_0^2) \big) = \calX_1 \big( \theta_1(\mathsf{T}_0^2) \big) = \calX_1(1)\ees
	by \eqref{speccharval}, and thus $\zeta = 1$.
\end{proof}

\begin{remark}\label{r:memechose}
	For $\calX = \calX_\Lambda$, $\Lambda \in \calQ_n$, we have by Remark~\ref{r:interestingvalues} and Theorem~\ref{p:expbijection}
	\begin{align}\label{e:flamdef}
	f_\Lambda := f_{\calX}	= {n \choose 2}\left( 1+ \frac{\calX_1(\omega)}{\calX(1)} \right)
	= {n \choose 2}
	\left( 1+ \frac{\chi^{\rN}_{\wt{\Lambda}}(\omega)}{\chi^{\rN}_{\wt{\Lambda}}(1)} \right)
	= {n \choose 2}
	\left( 1+ \frac{\chi^{\rN}_{{\Lambda}}(s)}{\chi^{\rN}_{{\Lambda}}(1)} \right)
	\end{align}
	with $ s \in \rN $ any transposition.
\end{remark}

%
%\begin{remark}\label{r:memechose2}
%	Let $ T_{s_0}:= T_{s_{i_1}} \ldots T_{s_{i_{ n \choose 2}}} \in \HHH(\G, \U)$ for any factorization $w_0 = s_{i_1} \ldots s_{i_{n \choose 2}}$ of the longest element of $ \SS_n ,$ as in \eqref{longestelt}.  By Lemma~\ref{l:t2central} we have that $ T_{s_0}^{2} $ is also central.  Since it only differs from $ T_{\omega_0}^{2}$ by a multiplicative factor of finite order (i.e., by $ T_t $ for some $t \in \T$), the argument in Theorem~\ref{t:scalarvaluecentral} also proves that $ T_{s_0}^{2} $ acts by scalar multiplication by the same value as $ T_{\omega_0}^{2} $ in any irreducible representation and hence must be the same element in $\HHH(\G,\U)$.
%\end{remark}

\begin{remark}\label{r:memechose2}
	With the modern definition of $ \YY_{d,n} $ (as in Remark~\ref{r:modernyoko})
	one can also define the element
	$ \mathsf{T}_{s_0}:= \mathsf{T}_{s_{i_1}} \ldots \mathsf{T}_{s_{i_{ n \choose 2}}} $ for any 
	factorization $w_0 = s_{i_1} \ldots s_{i_{n \choose 2}}$ of the longest element of $ \SS_n ,$ as in \eqref{longestelt}.  
	By the argument in Lemma~\ref{l:t2central} we have that $ \mathsf{T}_{s_0}^{2} $ is also central.  
	For a $ \Lambda \in \calQ_{d,n} $ (as in Remark \ref{r:rationality}) and its associated
	irreducible representation of $ \YY_{d,n}(u) ,$ the same argument from Theorem~\ref{t:scalarvaluecentral} 
	proves that $ \mathsf{T}_{s_0}^{2} $ acts by scalar multiplication by $ u^{f_\Lambda} $ 
	(as in \eqref{e:flamdef}).
	
	Therefore, this $ \mathsf{T}_{s_0}^{2} $ corresponds to our $ \mathsf{T}_{0}^{2}. $
\end{remark}

\begin{lemma}\label{r:flambdavalue} 
	Let $ \Lambda \in \calQ_n $ and $ f_\Lambda = f_{\calX_\Lambda}$ as in~\eqref{e:flamdef}.
	Then we have the formula
	%Let $\calX_\Lambda \in \Irr(Y_{q-1,n}(u))$ be the irreducible character corresponding, via
	%Remark~\ref{r:tits}, to the irreducible character of $\rN$ attached to $\Lambda:\Gamma_1\to \calP$
	%as explained in Section \ref{s:charnormtor}.  Then we have for $f_\Lambda:=f_{\calX_\Lambda}$ the 
	%formula
	
	\begin{align}\label{e:formula}
	f_{\Lambda} = \binom{n}{2} +  n(\Lambda') - n(\Lambda),
	\end{align} with the notation of \eqref{nLambda} from Section~\ref{s:charGL}.
\end{lemma}

\begin{proof}
	If $\lambda \in \P_n$ and $\chi_\lambda^\SS \in \Irr \SS_n$ is the corresponding irreducible character of $\SS_n$ and $s\in \SS_n$ is a simple transposition then by \cite[Exercise 4.17 (c)]{fulton-harris} or \cite[\S 7 (16.)]{frobenius1900charaktere},  \begin{equation}\label{e:valueatreflexion}
	\chi_\lambda^\SS ( s )  = \frac{ 2 \chi_\lambda^\SS (1) }{n(n-1)}\sum_{i=1}^{r}\pare{ \binom{b_i +1}{2} - \binom{a_i+1}{2} }
	\end{equation}
	where the $a_i$ and $b_i$ are the number of boxes below and to the right of the $i$th box of the diagonal in the Young diagram of $\lambda$.  By writing $j-i$ in the box $(i,j)$ and computing the sum in two ways we see at once that
	\begin{align}\label{e:formwithweightsize}
	\sum_{i=1}^{r}\pare{ \binom{b_i +1}{2} - \binom{a_i+1}{2} } = n(\lambda') - n(\lambda).
	\end{align}
	From this and \eqref{e:valueatreflexion}, we get that 
	\begin{align*}
	\binom{n}{2}\left(1+\frac{\chi_\lambda^\SS(s)}{\chi_\lambda^\SS(1)}\right)= \binom{n}{2} +n(\lambda') - n(\lambda). & 
	\end{align*}
	It remains to prove the analogous formula for $\chi_\Lambda^\rN \in \Irr \rN$ with $\Lambda \in \calQ_n$.
	
	Since we are working in $\rN$, we will omit the subscript and simply write $\chi_\Lambda$.  The description of the character $\chi_\Lambda$ was given in Section \ref{s:charnormtor} and in particular by the induction formula~\eqref{e:chiLambdaNdef}.  We will use the notation established there.  We will make the further abbreviations $\rN(\psi) := \T \times \Stab \psi = \prod_i \rN_i$ and
	\begin{align*}
	\chi_\Lambda^{\rN(\psi)} & := \bigotimes_i \chi_{\lambda_i, \psi_i}^{\rN_i} \in \Irr \rN(\psi) \\
	\chi_\Lambda^\psi & := \chi_{\lambda_1}^\SS \otimes \cdots \otimes \chi_{\lambda_r}^\SS \in \Irr (\Stab \psi) = \Irr( \SS_{n_1} \times \cdots \times \SS_{n_r}).
	\end{align*}
	In this notation, $\chi_\Lambda = \Ind_{\rN(\psi)}^\rN \chi_\Lambda^{\rN(\psi)}$.
	\begin{comment}
	the character $ \chi_\Lambda \in \Irr( (\F_q^\times)^{n} \rtimes \SS_n) $ associated to $\Lambda : \Gamma_1 \to \PP$ is induced from the tensor of two characters $\psi \in \Gamma_1^{n}$ and $ \chi $ an irreducible character of $ \Stab (\psi) \leeq \SS_n .$
	
	We may assume the character $ \psi \in \Gamma_1^n$ is defined by a sequence of $ k $ different characters $ \psi_{j} \in \widehat{\F}_q^\times$, with multiplicities $ n_j$, for $ j=1$, $\ldots$, $k$, such that $ \sum_{j=1}^{k} n_j =n$.  In such a case, its stabilizer
	\begin{align*} 
	\Stab \psi = \SS_{n_1}\times \ldots \SS_{n_k} \leeq \SS_n 
	\end{align*}
	and the character $ \chi \in \Irr(\Stab \psi)$ is thus a tensor product of  characters $ \chi^{\SS}_{\lambda^{j}} $ for certain partitions 
	\begin{align*}  
	\lambda^{j} = \Lambda(\psi_j)  \end{align*} of size $ n_j. $
	Hence, define
	\begin{align*}
	\chi^{\SS}_{(\lambda^{j})}:=\chi^{\SS}_{\lambda^{1}} \otimes\ldots \otimes \chi^{\SS}_{\lambda^{k}} \in \Irr(\Stab \psi ) 
	\end{align*}
	and extend it by inflation to a character in $ \Irr((\F_q^{\times})^{n} \rtimes \Stab \psi )$.
	
	Extending $ \psi $ to $ (\F_q^{\times})^{n}\rtimes\Stab(\psi) $ in the natural way we have
	\begin{align} \label{e:indformchi}
	\chi_{\Lambda}= \Ind_{(\F_q^\times)^{n} \Stab \psi }^{(\F_q^\times)^{n} \SS_n } \left( \psi \otimes \chi^{\SS}_{ (\lambda^{j}) } \right).
	\end{align}
	\end{comment}
	
	Let us evaluate $\chi_{\Lambda}$ at any transposition $\sigma \in \SS_n \leeq \T \rtimes \SS_n = \rN$.  Since 
	\begin{align*}   
	%(\psi\otimes\chi^{\SS}_{(\lambda^{j})})(1) = \chi^{\SS}_{(\lambda^{j})}(1)  
	\chi_\Lambda^{\rN(\psi)} (1) & = \chi_\Lambda^\psi(1) & & \textnormal{ and } &
	[\rN : \T \rtimes \Stab \psi ] & =  [\SS_n : \Stab \psi]
	\end{align*}
	we have
	\begin{align} \label{e:remindex}
	\chi_\Lambda(1) 
	%= \chi^{\SS}_{(\lambda^{j})}(1)[(\F_q^\times)^{n} \rtimes\SS_n : (\F_q^\times)^{n}\rtimes\Stab(\psi)]
	= \chi_\Lambda^\psi(1) [\SS_n : \Stab \psi ].
	\end{align}
	
	Throughout the remaining of the proof, we will write $ g= \xi \pi \in \T \times \SS_n = \rN$ for a general element of $\rN$.  By~\eqref{e:chiLambdaNdef} we have
	\begin{align}\label{e:defindchi}
	\chi_{\Lambda}(\sigma) = \frac{1}{\size{ \T \rtimes \Stab \psi }  }
	\sum_{g\in \T \rtimes \Stab \psi } \chi_\Lambda^{\rN(\psi)} ({}^{g}\sigma)
	\end{align}
	where $\chi_\Lambda^{\rN(\psi)}$ is extended by $0$ outside of $\T \rtimes\Stab \psi$, and 
	\begin{align*}
	{}^{g}\sigma = g \sigma g ^{-1} = \xi ({}^{\pi}\sigma) \xi^{-1}
	= \xi\;.\; {}^{({}^{\pi}\sigma) }(\xi^{-1})\;.\;( {}^{\pi}\sigma). 
	\end{align*}
	Note that $ {}^{\pi}\sigma \in \SS_n  $ and hence ${}^{({}^{\pi}\sigma) }(\xi^{-1}) \in \T$.  When $({}^{\pi}\sigma) \in \Stab \psi $ we have
	\begin{align*}   
	\psi(\xi\;.\;{}^{({}^{\pi}\sigma) }(\xi^{-1}) ) = \psi(\xi) \psi({}^{({}^{\pi}\sigma) }(\xi^{-1}) ) = \psi(\xi) {\psi}^{({}^{\pi}\sigma)   }(\xi^{-1} )  = \psi(\xi) \psi(\xi ^{-1})=1.
	\end{align*}
	
	Then
	\begin{align*}
	\chi_\Lambda^{\rN(\psi)}  = 
	\begin{cases}
	\chi^{\psi}_\Lambda({}^{\pi}\sigma ) & \textit{ if  } {}^{\pi}\sigma \in \Stab \psi \\
	0 & \textit{ otherwise }
	\end{cases}
	\end{align*}
	and \eqref{e:defindchi} becomes
	\begin{align*} %\label{e:indchibetter}
	\chi_{\Lambda}(\sigma) = \frac{1}{\size{\Stab \psi }  }
	\sum_{\pi \in \Stab \psi} \chi^{\psi}_\Lambda ({}^{\pi}\sigma) = \Ind_{\Stab \psi}^{\SS_n}(\chi^\psi_\Lambda )(\sigma).
	\end{align*}
	Since $ \sigma $ is a transpostion, the latter can be computed as
	\begin{align} \label{e:indchibetter}
	\chi_{\Lambda}(\sigma) = 
	\frac{1}{\size{\Stab \psi }   } \sum_{j=1}^4 \pare{\chi^{\SS}_{\lambda_j}({}^{\pi}\sigma) \prod_{i\neq j} \chi^{\SS}_{\lambda_i}(1)}
	\left| \set{ \pi \in \SS_n : \: {}^{\pi}\sigma \textrm{ is in the $j$th factor} } \right|.
	\end{align}
	The quantity in the summation is the order of the stabilizer of the set of transpositions $s$ in the $j$th factor of $ \Stab \psi = \SS_{n_1}\times\ldots \times \SS_{n_r}$, acted upon by $\SS_n$.  It is thus equal to
	\begin{align*}
	{n_j \choose 2 } 2(n-2)! = {n_j \choose 2 } {n \choose 2 }^{-1} n!
	\end{align*}
	and \eqref{e:indchibetter} becomes
	\begin{align*} % \label{e:indchibetter}
	{n \choose 2 } \frac{\chi_{\Lambda}(\sigma)}{\chi_\Lambda^\psi(1)}= 
	\frac{ \size{\SS_n}}{\size{\Stab \psi }   } \sum_{j=1}^{k} {n_j \choose 2 }\frac{\chi^{\SS}_{\lambda_j}(s_j)}{\chi^{\SS}_{\lambda_j}(1)}
	\end{align*}
	where $ s_j\in \SS_{n_j} $ is any transposition.  By \eqref{e:remindex} this becomes
	\begin{align*} % \label{e:indchibetter}
	{n \choose 2 } \frac{\chi_{\Lambda}(\sigma)}{\chi_{\Lambda}(1)}=  
	\sum_{j=1}^{k} {n_j \choose 2 }\frac{\chi^{\SS}_{\lambda_j}(s_j)}{\chi^{\SS}_{\lambda_j}(1)}
	\end{align*}
	and finally, by \eqref{e:valueatreflexion} and \eqref{e:formwithweightsize} 
	this becomes
	\begin{align*} % \label{e:indchibetter}
	{n \choose 2 } \frac{\chi_{\Lambda}(\sigma)}{\chi_{\Lambda}(1)}=  
	\sum_{j=1}^{k} n(\lambda_j') - n(\lambda_j) = n(\Lambda') - n(\Lambda)
	\end{align*}
	from which \eqref{e:formula} follows.
\end{proof}

%\subsection{Number of $\F_q-$points of $\M_{g, E, r, \calC}$}

%We are interested in the quantity
%\begin{align}
% N_U(v,g) := \# \set{ h \in G / g^h \in vU^{-}U^{+}}
%\end{align}
%as a class function for $g\in G,$ where $v \in T,$
%$U^{-}$ and $U^{+}=U$ stand for unipotent lower and upper triangular matrices, respectively.

%Since that $U^{-}=w_0U^{+}w_0,$ by~\eqref{counttr} and Lemma~\ref{l:secondcount},
%formula \eqref{trexp} in this case becomes
%\begin{align}%\label{e:almostfinalformulaU}
%N_U(v,g)=\sum_{ \substack{V \in \Irr \, G \\ V^U \neq \{ 0 \}}  } 
%\chi_V(g)  \mathcal{X}_{V^U} (T_vT_{\tilde{w}_0}^{2}),
%\end{align}
%which, by centrality of $T_{\tilde{w}_0}^{2}$ (Lemma~\ref{l:t2central} above),
%becomes
%\begin{align}%\label{e:almostfinalformulaU}
%N_U(v,g)=\sum_{ \substack{V \in \Irr \, G \\ V^U \neq \{ 0 \}}  } 
%\chi_V(g)  \mathcal{X}_{V^U} (T_v) \pare{ \frac{ \mathcal{X}_{V^U}( T_{\tilde{w}_0}^{2} ) }{ \mathcal{X}_{V^U}(1) }  },
%\end{align}
%and thanks to Theorem~\ref{t:scalarvaluecentral} it simplifies to
%\begin{align}\label{e:almostfinalformulaU}
%N_U(v,g)=\sum_{ \substack{V \in \Irr \, G \\ V^U \neq \{ 0 \}}  } 
%\chi_V(g)  \mathcal{X}_{V^U} (T_v) q^{f_\Lambda}.
%\end{align}
%Here the sum is indexed $V$'s associated to functions $\Lambda:\Gamma \to \PP$ 
%as in subsection~\ref{s:charGL}, but only those supported in $\Gamma_1$
%give $U-$invariant components % REF NEEDED
%which, by remark~\ref{r:tits} are in bijection with 
%characters of $N(T)(\F_q)$ as in~\ref{s:charnormtor}.

%With the notation as in \ref{ss:WChVar}
%following \ref{Halg}

\subsection{Counting at higher order poles}

For $r \in \Z_{>0}$, $v \in \T$ and $g \in \G$ let us now define
\begin{comment}
\begin{align*}
N^r_v(g):=\left| \left\{ (a,u^{-}_1,u^{+}_{1},\ldots,u^{-}_r,u^{+}_{r}) \in \G\times (\U^{-}\times \U^{+})^{r} \ \ |\ \  aga^{-1} =v \prod_{i=1}^{r}(u^{-}_{i}u^{+}_{i}) \right\} \right|.
\end{align*}
\end{comment}
\begin{align*}
N^r_v(g):=\left| \left\{ (a, S_1, S_2,\ldots, S_{2r-1}, S_{2r}) \in \G \times (\U_+ \times \U_-)^{r} \ \ |\ \  aga^{-1} =v \prod_{i=1}^r (S_{2(r+1-i)} S_{2(r-i)+1}) \right\} \right|,
\end{align*}
where $\U_+ := \U$ and $\U_- := \bU_-(\F_q)$, $\bU_-$ being the unipotent radical of $\bB_-$, the Borel opposite to $\bB$.
%and
%\begin{align}
% N^0_v(g):=\frac{\# \set{a, u \in G \times U  |  aga^{-1} = v u }}{\size{U}}
%\end{align}

\begin{theorem} We have
	\beq\label{e:Nrfinal} 
	N^r_v(g) = \sum_{ \Lambda \in \calQ_n  } \chi_\Lambda^\G(g)  \chi_\Lambda^\HHH (T_v)  {q}^{r f_{\Lambda} } .
	\eeq
\end{theorem}

\begin{proof}
	Recall that the multiplication map $\bT \times \bU_- \times \bU_+ \to \bG$ is an open immersion and that $\U_- = \omega_0 \U_+ \omega_0$, so that every element of $\omega_0 \U \omega_0 \U$ has a unique factorization in to a pair from $\U_- \times \U_+$ and similarly for $v \omega_0 \U \omega_0 \U$ and $v\U_- \times \U_+$.  
	%Thus, $1=|H\cap wHw^{-1}|=\frac{|H|^2}{|HwH|}=1$ for $w=w_0$ or $w=vw_0$.  
	Thus, it follows that \begin{align*}N^r_v(g)=N(g,\v),\end{align*} where  $\v=(v_1,w_1,\dots, v_r,w_r)$ with $v_1=v\omega_0$, $v_i=\omega_0$ for $i>1$ and $w_i=\omega_0$ for all $i$. So by Proposition~\ref{count} we have
	\begin{align*} %\label{nrcompute}
	N^r_v(g)=N(g,\v)= \sum_{ \Lambda \in \calQ_n  } \chi_\Lambda^\G(g)  \chi_\Lambda^\HHH (T_{v \omega_0}\ast T_{\omega_0}\ast T_{\omega_0}^{2r-2})  =\sum_{ \Lambda \in \calQ_n  } \chi_\Lambda^\G(g) \chi_\Lambda(T_v\ast T_{\omega_0}^{2r}).  
	\end{align*}
	Theorem~\ref{t:scalarvaluecentral} states that $T_{\omega_0}^2$ acts by the scalar $q^{f_\Lambda}$ in the irreducible representation corresponding to $\Lambda$, and the result follows.
\end{proof}

\subsection{Values at generic regular semisimple \texorpdfstring{$\F_q$}{Lg}-rational elements}

Here we compute our count function in the case when $v\in \T^{\reg}$, i.e., when $v$ has distinct eigenvalues.

\begin{proposition}\label{regular} Let $v \in \T^{\reg}$.  Then
	\begin{align}
	N^r_v(g) = \sum_{   \Lambda \in \calQ_n  } \frac{\chi^{\G}_\Lambda(g)}{\chi^{\G}_\Lambda (1)} 
	\pare{{q}^{r f_{\Lambda} } \frac{\overline{\chi^{\G}_\Lambda(v)} \size{\G} }{\chi^{\G}_\Lambda(1)}}
	\frac{\chi^{\G}_\Lambda(1)^{2}}{\size{\G}}= 
	\sum_{ \chi^\G_\Lambda\in \Irr \G  } \frac{\chi^{\G}_\Lambda(g)}{\chi^{\G}_\Lambda(1)}  
	\pare{{q}^{r f_{\Lambda} } \frac{\overline{\chi^{\G}_\Lambda(v)} \size{\G} }{\chi^{\G}_\Lambda(1)}} \frac{\chi^{\G}_\Lambda(1)^{2}}{\size{\G}}.
	\end{align}
	where the first sum is over the characters $\chi^{\G}_\Lambda \in \Irr \G$ defined in~\ref{s:charGL} for functions $\Lambda \in \calQ_n$, while the second sum is over all irreducible characters parametrised by $\Lambda:\Gamma\to\PP$ of size $n$. 
\end{proposition}

\begin{proof}
	%Now consider the $r=0$ case.
	For $v \in \T$ and $g \in \G$ set
	\begin{align}
	N^0_v(g):=\left| \left\{ (a, u) \in \G \times \U  \, \big| \, aga^{-1} = v u \right\} \right| \big/|\U|.
	\end{align}
	Then similarly as in the proof of Proposition~\ref{count} we compute
	\begin{align*}
	N^0_v(g) & = \frac{1}{|\U|^2}\sum_{h\in \U} \left| \left\{ (a, u) \in \G \times \U \, | \, aga^{-1} = hv u \right\} \right| = \frac{1}{|\U|^2} \sum_{a\in \G} \left( \II_\U \ast_\G \II_{v\U} \right) (aga^{-1})\\ 
	&=\frac{1}{|\U|} \sum_{a\in \G} \II_{\U v\U}(aga^{-1}) = 
	\tr(gT_v)= \sum_{ \Lambda\in \calQ_n }
	\chi^{\HHH}_\Lambda(v) \chi_\Lambda^{\G}(g),
	\end{align*}
	using that $\U \cap v\U v^{-1}=\U$.
	
	As $v$ has different eigenvalues every matrix in the coset $v\U$ is conjugate to $v$ and hence
	\begin{align}
	\left| \left\{ (a, u) \in \G \times \U \, | \, aga^{-1} = v u \right\} \right| = \size{\U} \left| \left\{ a \in \G \, | \, aga^{-1} = v \right\} \right| = |\U||C_\G(v)|
	\end{align}
	from which
	\begin{align}
	\sum_{\Lambda \in \calQ_n  } \chi_\Lambda^{\HHH} (v) \chi_{\Lambda}^{\G}(g) = \sum_{\chi^{\G} \in \Irr \G} \overline{\chi^{\G} (v)} {\chi^{\G} (g)}.
	\end{align}
	Since this is true for every $g \in \G$ we conclude
	\begin{align}
	\chi^{\G} (v) =
	\begin{cases}
	\overline{\chi^{\HHH}_\Lambda (v)} & \text{if $\chi^{\G} = \chi^{\G}_{\Lambda},$ for some $ \Lambda\in\calQ_n $ }
	\\
	0 &\text{otherwise.}  
	\end{cases}
	\end{align}
	Now the first equation in Proposition~\ref{regular} follows from Theorem~\ref{count}. The second equation follows as $\chi^{\G}_\Lambda(v) =0$ unless $\Lambda\in \calQ_n$.
\end{proof}

% This suggests that in the $g=0$ case with simple poles with regular semisimple monodromy, every solution
% is conjugate to one in the normalizer of the torus (possible proof with ideas from 
% \href{http://arxiv.org/abs/0903.0517}{http://arxiv.org/abs/0903.0517} )

\subsection{Counting formulas for wild character varieties}

For $i=1\dots k$ let $\calC_i\subset \G$ be a semisimple conjugacy class with eigenvalues in $\F_q$.  As usual, $\II_{\calC_i}:\G\to \C$ will denote its characteristic function. Fix $g\in \Z_{\geq 0}$ and define $D : \G\to \C$ by $D(g)=|\mu^{-1}(g)|$, where $\mu:\G\times \G\to \G$ is given by $\mu(g,h)=g^{-1}h^{-1}gh$. Finally let $m\in \Z_{\geq 0}$,  fix $$\r=(r_1,\dots,r_m)\in \Z^m_{>0}$$ and for each $j=1\dots m$ we fix $v_i\in \T^{\reg}$ and consider the count function $N_{v_i}^{r_i}:\G\to \C$. 

With this notation we have the following

\begin{proposition}\label{countwild} 
	\begin{align*} 
	D^{\ast_\G g}\ast_\G\II_{\calC_1}\dots \ast_\G \II_{\calC_m} & \ast_\G  N_{v_1}^{r_1}\ast_\G\dots \ast_\G N_{v_k}^{r_k} (1) \\
	& = \sum_{ \chi\in \Irr \G} \pare{\frac{\size{\G}}{\chi(1)}}^{2g} \prod_{j=1}^{m}\pare{ \frac{\overline{\chi(C_j)} \size{C_j} }{\chi(1)}} \prod_{i=1}^{k}\pare{{q}^{r_i f_{\chi} } \frac{\overline{\chi(v_i)} \size{\G} }{\chi(1)}} \frac{\chi(1)^{2}}{\size{\G}}.
	\end{align*}
\end{proposition}

\begin{proof}
	Recall that for a conjugacy class $C$ we have the characteristic function
	\begin{align}\label{conjcount}
	1_C(g) = \frac{\size{C}}{\size{\G}}\sum_{\chi \in \Irr \G} \chi(g)\overline{\chi(C)} =
	\sum_{\chi \in \Irr \G} \frac{\chi(g)}{\chi(1)}  \pare{\frac{\overline{\chi(C)} \size{C} }{\chi(1)}} \frac{\chi(1)^{2}}{\size{\G}}.
	\end{align}
	By \cite[Lemma 3.1.3]{hausel-aha1} we have that \beq\label{doublecount} D(g)=\sum_{\chi\in \Irr \G} \frac{\chi(g)}{\chi(1)}\left(\frac{|\G|}{\chi(1)}\right)^2\frac{\chi(1)^2}{|\G|}.
	\eeq 
	Combining Proposition~\ref{countwild}, \eqref{conjcount} and \eqref{doublecount} with the usual arithmetic harmonic analysis of \S~\ref{aha} we get the theorem. 
\end{proof}

With this we have our final count formula:

\begin{theorem}\label{finalcount} With notation as above let $(\calC_1,\dots,\calC_m)$ be of type $\muhat=(\mu_1,\dots,\mu_m)\in \calP_n^m$. Let $(\calC_1,\dots,\calC_m,v_1,\dots,v_k)$ be generic and $\tilde{\muhat}=(\mu_1,\dots,\mu_m,(1^n),\dots,(1^n))$ be its type. Finally denote $r=r_1+\dots+r_k$.  Then
	\bes \frac{q-1}{|\G||\T|^k} D^{\ast_\G g}\ast_\G\II_{\calC_1}\dots \ast_\G \II_{\calC_m}\ast_\G  N_{v_1}^{r_1}\ast_\G\dots \ast_\G N_{v_k}^{r_k} (1) = q^{d_{\tilde{\muhat},r}} \H_{\tilde{\muhat},r}(q^{-1/2},q^{1/2})
	\ees
\end{theorem}

Here, ``generic'' is in the sense of Definition \ref{genericconjugacyA}, which has an obvious analogue over an arbitrary field.

\begin{proof} Denote $\calC_{m+i}=\calC(v_i)$ then $|\calC(v_i)|=\frac{|\G|}{|\T|}$ as $v_i\in \T^{\reg}$. Furthermore let $r=r_1+\dots +r_k$.  This way we get 
	\begin{multline} \label{nosign} \frac{q-1}{|\G||\T|^k} D^{\ast_\G g}\ast_\G\II_{\calC_1}\dots \ast_\G \II_{\calC_m}\ast_\G  N_{v_1}^{r_1}\ast_\G\dots \ast_\G N_{v_k}^{r_k} (1) \\ 
	= \frac{q-1}{|\G|}\sum_{ \chi\in \Irr \G} \pare{\frac{\size{\G}}{\chi(1)}}^{2g} \prod_{j=1}^{m}\pare{ \frac{\overline{\chi(C_j)} \size{C_j} }{\chi(1)}} \prod_{i=1}^{k}\pare{{q}^{r_i f_{\chi} } \frac{\overline{\chi(v_i)} \size{\G} }{\chi(1)}|\T|} \frac{\chi(1)^{2}}{\size{\G}} \\ 
	= \frac{q-1}{|\G|}\sum_{ \chi\in \Irr \G} q^{rf_\chi} \pare{\frac{\size{\G}}{\chi(1)}}^{2g} \prod_{j=1}^{m+k}\pare{ \frac{\overline{\chi(C_j)} \size{C_j} }{\chi(1)}} \frac{\chi(1)^{2}}{\size{\G}}.
	\end{multline}
	As $(\calC_1,\dots,\calC_m,\calC_{m+1},\dots,\calC_{m+k})$ is assumed to be generic, we can compute exactly as in \cite[Theorem 5.2.3]{hausel-aha1}. The only slight difference is the appearance of $q^{rf_\chi}$.  We observe that the quantity $f_\Lambda$ in \eqref{e:formula} behaves well with respect to taking Log and the same computation as in \cite[Theorem 5.2.3]{hausel-aha1} will give Theorem~\ref{finalcount}.
\end{proof}

\begin{remark}
	In the definition of $\H_{\tilde{\muhat},r}(z,w)$ in \eqref{hmur} we have a sign $(-1)^{rn}$ in the LHS and a sign $(-1)^r$ in the RHS in \eqref{calhgr}. As $\tilde{\muhat}$ contains the partition $(1^n)$ one will not have to compute the plethystic part of the $\Log$ function to get  $\H_{\tilde{\muhat},r}(z,w)$ and in this case  the signs on the two sides will cancel. That is why we do not see the sign in \eqref{nosign} in front of $q^{rf_\chi}$. 
\end{remark}

	\section{Main theorem and conjecture}
	\label{mains}
\subsection{Weight polynomial of wild character varieties}

Let $\muhat\in \P_n^k$ and $\r\in \Z_{>0}^m$. Let $\M_\B^{\muhat,\r}$ be the generic complex wild character variety defined in \eqref{wilddef}. Here we prove our main Theorem~\ref{maint}.

\begin{proof}[Proof of Theorem~\ref{maint}] The strategy of the proof is as follows. First we construct a finitely generated ring $R$ over $\Z$, which will have the parameters corresponding to the eigenvalues of our matrices. Then we construct a spreading out of $\M_\B^{\muhat,\r}$ over $R$. We finish by counting points over $\F_q$ for the spreading out, find that it is a polynomial in $q$ and deduce that it is the weight polynomial of $\M_\B^{\muhat,\r}$ by \cite[Appendix A]{hausel-villegas}.
	
	As in \cite[Appendix A]{hausel-aha1} first construct $R$ the finitely generated ring of generic eigenvalues of type $\tilde{\muhat}$, where  $$\tilde{\muhat}=(\mu^1,\dots,\mu^k,(1^n),\dots,(1^n))\in \calP_n^{k+m}.$$  In particular, we have variables $\{a_j^i\}\in R$ for $i=1, \dots, k+m$ and $j=1, \dots, l(\mu^i)$ representing the eigenvalues of our matrices.  They are already generic in the sense that they satisfy the non-equalities of $a^i_{j_1}\neq a_{j_2}^i$ when $0<j_1<j_2\leq l(\tilde{\mu}^i)$ and the ones in \eqref{noneq}. 
	
	%We will need to ensure the existence of primitive fourth roots of unity, %thus we will add them to $R$ by considering the ring $R[1/4,\zeta_4]$. %If now we have $\phi:R[1/4,\zeta_4] \to \K$ to a field $\K$ then %$\phi(\zeta_4)$ is a primitive fourth root of unity.  In particular, %when $|\K|=q$ is finite then $4|q-1$.  Slightly abusing notation, we %will denote $R[1/4,\zeta_4]$ by $R$ in what follows. 
	
	Generalising  \cite[Appendix A]{hausel-aha1} we consider the algebra $\calA_0$ over $R$ of polynomials in $n^2(2g+k+m)+r(n^2-n)$ variables, corresponding to the entries of $n\times n$ matrices
	$$A_1,\ldots,A_g;B_1,\ldots, B_g;X_1,\ldots,X_k,C_1,\ldots,C_m$$ and upper triangular matrices $S_{2j-1}^{i}$ and lower triangular matrices $S_{2j}^{i}$ with $1$ on the main diagonal for $i=1\dots m$ and $j=1\dots r_i$ such that
	$$
	\det A_1,\ldots,\det A_k;\quad \det B_1,\ldots,\det B_k; \quad \det X_1,\ldots,\det X_k
	$$
	are inverted. 
	
	Let $I_n$ be the $n \times n$ identity matrix, let $\xi_i$ be the diagonal matrix with diagonal elements $a_1^{i+k},\dots, a_n^{i+k}$ for $i=1, \dots, m$.  Finally, for elements $A,B$ of a group, put $(A,B):=ABA^{-1}B^{-1}$.  Define $\calI_0\leeqrad \calA_0$ to be the radical of the ideal generated by the entries of
	\begin{align*}
	&(A_1,B_1)\cdots (A_g,B_g) X_1\cdots X_k C_1^{-1}\xi_1S^1_{2r_1}\cdots S^1_1C_1\cdots C_m^{-1}\xi_mS^m_{2r_m}\cdots S^m_1C_m-I_n, \\
	& (X_i-a_1^iI_n)\cdots(X_i-a_{r_i}^iI_n), \quad i=1,\ldots,k
	\end{align*}
	and the coefficients of the polynomial
	$$
	\det(tI_n-X_i)-\prod_{j=1}^{r_i}(t-a_j^i)^{\mu_j^i}
	$$
	in an auxiliary variable $t$. Finally, let $\calA:=\calA_0/\calI_0$ and $\pazU_{\muhat,\r}:=\Spec(\calA)$ an affine $R$-scheme.
	
	Let  $\phi:R \rightarrow \K$ be a map  to a field $\K$ and let $\pazU_{\muhat,\r}^\phi$ be the corresponding base change of $\pazU_{\muhat,\r}$ to $\K$. A $\K$-point of $\pazU_{\muhat,\r}^\phi$ is a solution in $\GL_n(\K)$ to
	\beq \label{phieq}
	(A_1,B_1)\cdots(A_g,B_g) X_1\cdots X_k C_1^{-1}\xi^\phi_1S^1_{2r_1}\cdots S^1_1C_1\cdots C_m^{-1}\xi^\phi_mS^m_{2r_m}\cdots S^m_1C_m=I_n,
	\eeq
	where $X_i\in \calC_i^\phi$ and $\calC_i^\phi$  is the semisimple conjugacy class in $\GL_n(\K)$ with eigenvalues $$\phi(a^i_1),\ldots,\phi(a_{r_i}^i)$$ of multiplicities $\mu^i_1,\ldots,\mu^i_{r_i}$ and $\xi_i^\phi\in T^\reg(\K)$ is a diagonal matrix with diagonal entries $$\phi(a^{k+i}_1),\dots,\phi(a^{k+i}_n).$$ By construction $(\calC_1^\phi,\dots,\calC_k^\phi,\xi_1^\phi,\dots,\xi_m^\phi)$ is generic.
	
	Finally $\calG=\GL_n\times \T^k$ acts on $\pazU_{\muhat,\r}$ via the formulae \eqref{actionR}. We take  $$\pazM_{\muhat,\r}=\Spec(\calA^{\calG(R)})$$ the affine quotient of $\pazU_{\muhat,\r}$ by $\calG(R)$. Then for $\phi:R\to \C$ the complex variety $\pazM_{\muhat,\r}^\phi$ agrees with our $\M_\B^{\muhat,\r}$ thus $\pazM_{\muhat,\r}$ is its spreading out. 
	
	We need the following
	
	\begin{proposition} Let $\phi:R\to \K$ a homomorphism to a field $\K$. Then if $A_i,B_i,X_j,C_\alpha\in \GL_n(\K)$, a solution to \eqref{phieq} representing a $\K$-point in $\calU^\phi_{\muhat,\r}$ is stabilized by \[(y,x_1,\dots,x_m)\in \calG^\phi=\calG\otimes_\phi \K=\GL_n(\K)\otimes T(\K)^m\] then \[y=x_1=\dots=x_m\in Z(\GL_n(\K))\] is a scalar matrix. Equivalently, if $D=\{\lambda I_n,\dots,\lambda I_n\}\leeq \calG^\phi$ is the corresponding subgroup then $\overline{\calG}^\phi:=\calG^\phi/D$ acts set-theoretically freely on $\calU^\phi_{\muhat,\r}$.
	\end{proposition}
	
	\begin{proof} By assumption \beq\label{calphafix} x_\alpha C_\alpha y^{-1}= C_\alpha\eeq thus the matrices $y,x_1,\dots,x_m$ are all conjugate and split semisimple. Let $\lambda\in \K$ be one of their  eigenvalues and $V_\lambda < \K^n$ be the $\lambda$-eigenspace of $y$ then by \eqref{calphafix} $C_\alpha(V_\alpha)\leeqvs \K^n$ is the $\lambda$-eigenspace of $x_\alpha$. As $y$ commutes with all of $A_i,B_i,X_j$ we see that they leave $V_\lambda$ invariant. While $x_\alpha$ commutes with $S^\alpha_i$ and  $\xi_\alpha$ thus they leave $C_\alpha(V_\lambda)$ invariant  or equivalently $C_\alpha^{-1}S^\alpha_i C_\alpha$ and $C_\alpha^{-1} \xi_\alpha C_\alpha$ leave $V_\lambda$ invariant. As $S^\alpha_i$ is unipotent $$\det(C_\alpha^{-1}S^\alpha_i C_\alpha|_{V_\lambda})=1$$ and the determinant of the  equation \eqref{phieq} restricted to $V_\alpha$ gives $$\prod_{i=1}^k \det(X_i|_{V_\lambda})\prod_{\alpha=1}^m\det(\xi_\alpha|_{V_\lambda})=1.$$
		By assumption  $(\calC_1^\phi,\dots,\calC_k^\phi,\xi_1^\phi,\dots,\xi_m^\phi)$ is generic.   Thus, we get from \eqref{prod-conditionA} that $V_\lambda = \K^n$.
	\end{proof}
	
	Let now $\K=\F_q$ a finite field and assume we have $\phi:R\to \F_q$.  Because $\overline{\calG}$ is connected and $\overline{\calG}(\K)$ acts freely on $\pazU^\phi_{\muhat,\r}$ we have by similar arguments as in \cite[Theorem 2.1.5]{hausel-aha1}, \cite[Corollaries 2.2.7, 2.2.8]{hausel-villegas} and by Theorem~\ref{finalcount} that
	$$\#\pazM_{\muhat,\r}^\phi(\F_q)=\frac{\#\pazU^\phi_{\muhat,\r}}{\#\overline{G}(\F_q)}=q^{d_{\tilde{\muhat},r}} \H_{\tilde{\muhat},r}(q^{-1/2},q^{1/2}).$$
	As by construction $\H_{\tilde{\muhat},r}(q^{-1/2},q^{1/2})\in \Q(q)$ and $\#\pazM_{\muhat,\r}^\phi(\F_q)$ is an integer for all prime power $q$ we get that  $\#\pazM_{\muhat,\r}^\phi(\F_q)\in \Q[q]$.  Katz's Theorem~\ref{katz} applies finishing the proof.
\end{proof}

We have the following immediate
\begin{corollary} \label{palindromic}The weight polynomial of $\M_\B^{\muhat,\r}$ is palindromic: $$WH(\M_\B^{\muhat,\r};q,-1)=q^{d_{\muhat,\r}}WH(\M_\B^{\muhat,\r};1/q,-1).$$
\end{corollary}

\begin{proof} This is a consequence of Theorem~\ref{maint} and the combinatorial Lemma~\ref{curpoincomb} proved below. \end{proof}

\subsection{Mixed Hodge polynomial of wild character varieties}

In this section we discuss Conjecture~\ref{mainc}. First we recall the the combinatorics of various symmetric functions from \cite[\S 2.3]{hausel-aha1}. % In particular, $\calP=\cup_1^\infty \calP_n$ is the set of partitions of $n=1,2,\dots$.  
Let $$\Lambda({\bf{x}}):=\Lambda({\bf x}_1,\dots,{\bf x}_k)$$ be the ring of functions separately symmetric in each of the set of variables  $${\bf x}_i=(x_{i,1},x_{i,2},\dots).$$  For a partition, let $\lambda\in \calP_n$ $$s_\lambda({\bf x}_i), \quad m_\lambda({\bf x}_i), \quad h_\lambda({\bf x}_i) \in \Lambda({\bf x}_i)$$ be the Schur, monomial and complete symmetric functions, respectively. By declaring $\{s_\lambda({\bf x}_i)\}_{\lambda\in \calP}$ to be an orthonormal basis, we get the Hall pairing $\langle \, ,\rangle$, with respect to which $\{m_\lambda({\bf x}_i)\}_{\lambda\in \calP}$ and $\{h_\lambda({\bf x}_i)\}_{\lambda\in \calP}$ are dual bases.   We also have the Macdonald polynomials of \cite{garsia-haiman}  $$\tilde{H}_\lambda(q,t)=\sum_{\mu\in \calP_n} \tilde{K}_{\lambda\mu} s_\mu(\bf{x})\in \Lambda({\bf x})\otimes_\Z\Q(q,t).$$ And finally we have the plethystic operators $\Log$ and $\Exp$ (see for example \cite[\S 2.3.3.]{hausel-aha1}). 

With this we can define for $\muhat=(\mu^1,\dots,\mu^k)\in \calP_n^k$ and $r\in \Z_{>0}$ the analogue of the Cauchy kernel:
$$\Omega_k^{g,r}(z,w):=\sum_{\lambda\in \calP}\calH^{g,r}_\lambda(z,w)\prod_{i=1}^k \tilde{H}_\lambda(z^2,w^2;{\bf x}_i)\in \Lambda({\bf x}_1,\dots,{\bf x}_k)\otimes_\Z \Q(z,w),$$
where the hook-function $\calH^{g,r}_\lambda(z,w)$ was defined in \eqref{calhgr}. This way we can define $$\H_{\muhat,r}(z,w):=(-1)^{rn}(z^2-1)(1-w^2)\left\langle \Log(\Omega_k^{g,r}(z,w)),h_{\mu^1} ({\bf{x}}_1)\otimes\dots\otimes h_{\mu^k} ({\bf x}_k) \right\rangle,$$ which is equivalent with the definition in \eqref{hmur}. 

We have an alternative formulation of the polynomials $\H_{\muhat,r}(z,w)$ using only Cauchy functions $\Omega^{g,0}_k$ for $r=0$, which we learnt from F.R.\ Villegas. 

\begin{lemma} \label{fernandol} One has
	\begin{align*}
	\H_{\muhat,r}(z,w)=(z^2-1)(1-w^2)\left\langle \Log\left(\Omega_{k+r}^{g,0}(z,w)\right), h_{\muhat}({\bf{x}})\otimes s_{(1^n)}({\bf x}_{k+1})\otimes \dots \otimes s_{(1^n)}({\bf x}_{k+r})\right\rangle.
	\end{align*}
\end{lemma}
\begin{proof} Recall \cite[Proposition 3.1, Lemma 3.3]{hausel-dtkac} that the operation \[F\mapsto [F]=(-1)^n\langle F, s_{(1^n)}({\bf x})\rangle \] for $F\in \Lambda({\bf x})\otimes_\Z \Q(z,w)$ commutes with taking the $\Log$, i.e. \beq \label{commute}[\Log(F)]=\Log([F]).\eeq  We also have  $$\langle \tilde{H}_\lambda(q,t;{{\bf x}_i}), s_{(1^n)}({{\bf x}_i})\rangle=t^{n(\lambda)}q^{n(\lambda^\prime)},$$which is \cite[I.16]{garsia-haiman}. This implies $$(-1)^{rn}\left\langle \Omega^{g,0}_{k+r}(z,w), s_{(1^n)}({\bf x}_{k+1})\otimes \dots \otimes s_{(1^n)}({\bf x}_{k+r}) \right  \rangle=\Omega^{g,r}_k(z,w).$$ In turn \eqref{commute} gives the result.  \end{proof}

We can now claim Conjecture~\ref{mainc}, which predicts the mixed Hodge polynomial \beq\label{mainceq} WH(\M_\B^{\muhat,\r};q,t)=(qt^2)^{d_{\muhat,\r}}\H_{\tilde{\muhat},r}(q^{-1/2},-tq^{1/2}),\eeq where again $\tilde{\muhat}=(\mu^1,\dots,\mu^k,(1^n),\dots,(1^n))\in \calP_n^{k+m}$ and $r=r_1+\dots+r_m$. Here we are going to list some evidence and consequences of this conjecture. The main evidence for Conjecture~\ref{mainc} is naturally Theorem~\ref{maint} showing the $t=-1$ specialization of \eqref{mainceq} is true. 

The first observation is the following 

\begin{lemma}\label{curpoincomb} $\H_{\muhat,\r}(z,w)=\H_{\muhat,\r}(-w,-z)$
\end{lemma}
\begin{proof}
	As the Macdonald polynomials satisfy the symmetry 
	$$\tilde{H}_{\lambda^\prime}(w^2,z^2;{\bf{x}})=\tilde{H}_{\lambda}(z^2,w^2;{\bf{x}})$$ with $\lambda^\prime$ the dual partition, and the Hook polynomials $\calH_\lambda^{g,r}(z,w)=\calH_\lambda^{g,r}(w,z)=\calH_\lambda^{g,r}(-w,-z)$ the result follows.
\end{proof}

Together with Conjecture~\ref{mainc} this implies the following
{\em curious Poincar\'e duality} 
\begin{conjecture} \label{cpd} $WH(\M_\B^{\muhat,\r};q,t)=(qt)^{d_{\muhat,\r}}WH(\M_\B^{\muhat,\r};1/(qt^2),t).$
\end{conjecture}

Next we have 

\begin{theorem} \label{empty} Let $g=0$, $k=0$, $m=1$, $r_1=1$ and $n\in \Z_{>1}$  then $\M_\B^{\emptyset,(1)}=\emptyset$. Correspondingly in this case $\H_{(1^n),1}(z,w)=0$. In other words in this case Conjecture~\ref{mainc} holds. 
\end{theorem}

\begin{proof} As $\T\times \U^-\times \U^+\to G$ given by $\xi,S_1,S_2\mapsto \xi S_1S_2$ is an embedding, $\xi_1 S_1 S_2 =1$ implies $\xi_1=1\notin \T^\reg$ showing that $\M_\B^{\emptyset,1}=\emptyset$. 
	
	As $s_{(1^n)}=\sum_{\lambda\in \calP_n} K^*_{(1^n)\lambda} h_\lambda$, we get by Lemma~\ref{fernandol} that \bes \frac{\H_{(1^n),1}(z,w)}{(z^2-1)(1-w^2)}&=& (-1)^n\left
	\langle \Log\left( \Omega^{0,1}_1(z,w)\right),h_{(1^n)}({\bf x}_1)   \right\rangle \\ 
	&=& \left\langle \Log\left(\Omega^{0,0}_2(z,w)\right),h_{(1^n)} ({\bf x}_1) \otimes s_{(1^n)} ({\bf x}_2)  \right\rangle \\ 
	&=& \sum_{\lambda\in \calP_n} K^*_{(1^n)\lambda}  \left \langle \Log \left( \Omega^{0,0}_2(z,w)\right),h_{(1^n)} ({\bf x}_1) \otimes  h_\lambda({\bf x}_2)  \right\rangle \\ &=& 0 \ees The last steps follows from \cite[(1.1.4)]{hausel-aha1} the orthogonality property of the usual Cauchy function $\Omega^{0,0}_2$.\end{proof}

After the case in Theorem~\ref{empty} the next non-trivial case is when $g=0$, $k=1$,  $m=1$, $r=1$ and $\muhat=(\mu) \in \calP_n$. The corresponding wild character variety $\M_\B^{(\mu),(1)}$ is known by \cite[Corollary 9.10]{boalch-multquiv} to be isomorphic to a tame character variety $\M_\B^{\muhat^\prime}$, where $$\muhat^\prime=((n^\prime-1,1),\dots,(n^\prime-1,1),\mu^\prime)\in\calP_{n^\prime}^{n+1}$$ with $n^\prime=n-\mu_1$ and $\mu^\prime=(\mu_2,\mu_3,\dots)\in \calP_{n^\prime}$.  Combining Boalch's $\M_\B^{(\mu),(1)}\cong \M_\B^{\muhat^\prime}$ with Conjecture~\ref{mainc} we get the following combinatorial 

\begin{conjecture} \label{21tame} With the notation as above $\H_{(\mu,(1^n)),1}(z,w)=\H_{\muhat^\prime}(z,w)$.
\end{conjecture}

\begin{remark} In a recent preprint \cite[Corollary 7.2]{mellit} Anton Mellit gives a combinatorial proof of this conjecture. 
	From our results we see that  Theorem~\ref{maint} and \cite[Corollary 9.10]{boalch-multquiv} imply  the $t=-1$ specialization \[\H_{(\mu,(1^n)),1}(q^{-1/2},q^{1/2})=\H_{\muhat^\prime}(q^{-1/2},q^{1/2}).\]   
	
	When $n=2$ we will check Conjecture~\ref{21tame}, as well as our main Conjecture~\ref{mainc} in some particular cases in the next section. 
\end{remark}

	\section{Examples when \texorpdfstring{$n=2$}{n=2}} \label{exspec}
	
\setcounter{subsection}{1}
In this section we set  $n=2$, $g\in \Z_{\geq 0}$,  $k+m>0$, $\r=(r_1,\dots,r_m)\in \Z_{>0}^m$, $r=r_1+\cdots+r_m$ and $\muhat=((1^2),\dots,(1^2))\in \calP_m$.
Conjecture~\ref{mainc} in this case predicts that the mixed Hodge
polynomial $WH(\M_\B^{\muhat,\r};q,t)$ is given by
\begin{align}\label{mhpn2} (qt^2)^{d_{\muhat,\r}}\H_{\tilde{\muhat},r}(q^{-1/2},-(qt^2)^{1/2})= &
\frac{(qt^2+1)^{k+m}(q^2t^3+1)^{2g}(1+qt)^{2g}}{(q^2t^2-1)(q^2t^4-1)} \\ \nonumber &-\frac{2^{{k+m}-1}(qt^2)^{2g+r-2+{k+m}}(qt+1)^{4g}}{(q-1)(qt^2-1)} \\  \nonumber &+\frac{t^{-2r}(qt^2)^{2g+2r-2+k+m}(q+1)^{k+m}(q^2t+1)^{2g}(1+qt)^{2g}}{(q^2-1)(q^2t^2-1)}
.
\end{align}
Note, in particular, that the formula depends only on $k+m$ and $r$. 

Substituting $t=-1$ gives by Theorem~\ref{maint} the following
\begin{align}\label{epn2} WH(\M_\B^{\muhat,\r};q,-1)=& {(q+1)^{k+m}(q^2-1)^{2g-2}(q-1)^{2g}} - {2^{{k+m} - 1}q^{2g+r-2+{k+m}}(q-1)^{4g-2}} \nonumber \\ &+{q^{2g+2r-2+k+m}(q+1)^{k+m}(q^2-1)^{2g-2}(q-1)^{2g}}
\end{align}

Fix now $g=0$ in the remainder of this section. Then from \eqref{dimensionformula} we get that \beq \label{dimension} \dim \M_\B^{\muhat,\r}=4(k-2)-2k+2(m+r)+2= 2(k+r+m)-6.\eeq
When $k+r+m<3$ the moduli spaces are empty and the corresponding formula in \eqref{mhpn2} gives indeed $0$. 

When 
$k+m+r=3$ then we have $k=0$, $m=1$ and $r=2$ or
$k=m=r=1$ or $k=3$ and $r=m=0$. In these cases we get $1$ in \eqref{mhpn2}. This corresponds to the fact that the moduli spaces are single points in these cases. This follows  as they are $0$-dimensional by \eqref{dimension} and we have $$WH(\M_\B^{\muhat,\r};q,-1)=1$$ by \eqref{epn2}. In particular, we have that $$\H_{((1^2),(1^2)),1}(z,w)=1=\H_{((1^2),(1^2),(1^2))}(z,w),$$ confirming Conjecture~\ref{21tame} when $n=2$. 

Finally when $k+m+r=4$ the moduli spaces are $2$-dimensional from \eqref{dimension}. 
In the tame case when $k=4$ and $m=0$ we get the familiar $\hat{D}_4$ case discussed at \cite[Conjecture 1.5.4]{hausel-aha1} with mixed Hodge polynomial \bes WH(\M^{((1^2),(1^2),(1^2),(1^2))}_\B;q,t)=1+4qt^2+q^2t^2.\ees In fact the corresponding Higgs moduli space $\M^{((1^2),(1^2),(1^2),(1^2))}_\Dol$ served as the toy model in \cite{hausel-com} and has the same perverse Hodge polynomial. 

We have four wild cases with $k+m+r=4$. When $k=2$ and $m=r=1$ \eqref{mhpn2} predicts \beq\label{211}WH(\M_\B^{((1^2),(1^2)),(1))};q,t)=1+3qt^2+q^2t^2.\eeq We can prove this by looking at \cite[pp.2636]{vanderput-saito} and read off the wild character variety of type $(0,0,1)$ given as an affine cubic surface $f(x_1,x_2,x_3)=0$ with leading term
$x_1x_2x_3$.  By computation we find that $f$ has isolated singularities, moreover the leading term has isolated singularities at infinity. Thus \cite[Theorem 3.1]{siersma-tibar} applies  showing that $\M_\B^{((1^2),(1^2)),(1))}$ has the homotopy type of a bouquet of $2$-spheres. In particular it is simply connected and there is only one possibility for the weights on $H^2(\M_\B^{((1^2),(1^2)),(1))})$ to give the weight polynomial  $$WH(\M_\B^{((1^2),(1^2)),(1))};q,-1)=1+3q+q^2,$$ which we know from 
\eqref{epn2},  namely the one giving \eqref{211}. 

When $k=1$, $m=1$ and $r=2$ the formula \eqref{mhpn2} predicts the mixed Hodge polynomial
\beq\label{112}WH(\M_\B^{(1^2),(2)};q,t)=1+2qt^2+q^2t^2.\eeq
This again we can prove by looking at \cite[pp. 2636]{vanderput-saito} the line of $(0,-,2)$ we get an affine cubic surface with leading
term $x_1x_2x_3$. Again \cite[Theorem 3.1]{siersma-tibar} implies that $\M_\B^{(1^2),(2)}$ is homotopic to a bouquet of $2$-spheres thus the only possible weights on $H^2(\M_\B^{(1^2),(2)})$ to give the known specialization $$WH(\M_\B^{(1^2),(2))};q,-1)=1+2q+q^2$$ is the one claimed in \eqref{112}.

When $k=0$, $m=r=2$; then we get again a cubic surface \cite[pp.2636]{vanderput-saito} corresponding to $(1,-,1)$ and thus the same \bes WH(\M_\B^{\emptyset,(2,2)};q,t)=1+2qt^2+q^2t^2,\ees which we can prove in an identical way as above. 

Finally when $k=0$, $m=1$ and $r=3$ we get \bes WH(\M_\B^{\emptyset,(3)};q,t)=1+qt^2+q^2t^2.\ees Here the same argument applies using the explicit cubic equation in  \cite[pp.2636]{vanderput-saito} corresponding to the case $(-,-,3)$.

\end{document}